\documentclass[final]{siamltex}
\usepackage{times}
\usepackage{algorithmicx, algpseudocode}
\usepackage{algorithm}
\usepackage{color}
\usepackage{graphicx}
\usepackage[font={small}]{caption}
\usepackage{ulem}
\usepackage{verbatim}
\usepackage{wrapfig,lipsum,booktabs}
\usepackage[mathscr]{euscript}
\usepackage{amssymb}
\usepackage{amsmath}
\usepackage{cleveref}
\usepackage{multirow}
\usepackage{bbm}
\usepackage[titletoc,title]{appendix}
\usepackage{mathabx}
\usepackage{spreadtab}
\usepackage{diagbox}
\usepackage{subcaption}
\usepackage{xcolor}
\usepackage{mathtools}
\usepackage{pdflscape}
\usepackage{rotating}

\def\nh{{\small \sc \bf Nh2D}}
\def\sky{{\small \sc \bf Sky2D}}
\def\skyt{{\small \sc \bf Sky3D}}
\def\ani{{\small \sc \bf Ani3D}}

\def\nho{{\small \sc  Nh2D}}
\def\skyo{{\small \sc  Sky2D}}
\def\skyto{{\small \sc  Sky3D}}
\def\anio{{\small \sc  Ani3D}}
\newcommand{\ltwonorm}[1]{\left \| #1 \right \|_{2}}

\newcommand{\sq}{\hbox{\rlap{$\sqcap$}$\sqcup$}}
\newcommand{\qed}{\hspace*{\fill}\sq}

\newlength{\algrhswidth}
\title{Flexibly  Enlarged Conjugate Gradient methods}

\author{Sophie Moufawad \thanks{American University of Beirut (AUB), Beirut, Lebanon.  (sm101@aub.edu.lb) } }

\begin{document}
\maketitle \thispagestyle{empty}

\begin{abstract}
Enlarged Krylov subspace methods and their s-step versions were introduced \cite{EKS} in the aim of reducing communication when solving systems of linear equations $Ax=b$. These enlarged CG methods consist of enlarging the Krylov subspace by a maximum of $t$ vectors per iteration based on the domain decomposition of the graph of $A$.
As for the s-step versions, $s$ iterations of the enlarged Conjugate Gradient methods are merged in one iteration. 
The Enlarged CG methods and their s-step versions converge in less iterations than the classical CG, but at the expense of requiring more memory storage than CG. Thus in this paper we explore different options for reducing the memory requirements of these enlarged CG methods without affecting much their convergence.  
\end{abstract}


\section{Introduction}
\addtolength{\belowdisplayskip}{-3mm}
\addtolength{\abovedisplayskip}{-3mm}
Conjugate Gradient (CG) \cite{cgor} is a well-known Krylov subspace iterative method used for solving systems of linear equations $Ax = b$ with $n\times n$ symmetric positive definite (spd) sparse matrices $A$. Other Krylov subspace methods that solve general systems are Generalized Minimal Residual (GMRES) \cite{ssch}, bi-Conjugate Gradient \cite{bicg1,bicg2}, and bi-Conjugate Gradient Stabilized \cite{bicgstab}. CG is known for its short recurrence relations for defining the approximate solution $x_k \in x_0 + \mathcal{K}_k(A,r_0)$  at the $k^{th}$ iteration, where $\mathcal{K}_k(A,r_0) = \{r_0, Ar_0, \cdots, A^{k-1}r_0\}$ is the Krylov subspace of dimension at most $k$, $r_0 = b-Ax_0$ is the initial residual, and $x_0$ is the initial iterate or guess.
 This short recurrence leads to limited memory requirements (4 vectors of length $n$ and sparse matrix $A$).
However, CG iterations consist of BLAS 1 operations, specifically dot products and SAXPY's, and one BLAS2 operation, GAXPY or matrix-vector multiplication. Thus, when implementing CG on modern computer architectures, communication is a bottleneck in these BLAS operations due to technological reasons \cite{memTech}, where communication refers to data movement between different levels of memory hierarchy (sequential) and different processors/cores (parallel).

Different approaches have been adopted to reduce this effect such as pipelining communication and computation  in order to hide communication \cite{hidecg,hidecg2}, or replacing BLAS1 and BLAS2 operations by denser operations that may be parallelized more efficiently with less communication. The latter can be achieved using different strategies. The first is to merge $s$ iterations of CG to obtain the s-step method  \cite{sstepcg1,chronop}. Moreover, communication avoiding methods \cite{hoemmen}, which are based on s-step methods, introduce communication avoiding kernels that further reduce communication, even if at the expense of performing redundant computations. Other variants of s-step CG methods where introduced like \cite{adaptCG}.
The second strategy is to introduce new variants of CG by using different and larger Krylov subspaces, 
such as augmented CG \cite{augmKSM, augmKSM2, erhelASCG}, block CG \cite{bcg}, and enlarged CG \cite{EKS, sophiethesis} that are based on augmented, block, and enlarged Krylov subspaces, respectively. For a detailed comparison between these methods, refer to \cite{sstepECG}.

Several Enlarged Krylov subspace Conjugate Gradient methods were introduced \cite{EKS, sophiethesis} with the goal of obtaining  methods that converge faster than classical Krylov methods in terms of iterations and consequently parallel runtime, by performing denser operations per iteration that require less communications when parallelized. These methods approximate the solution of the system $Ax = b$
 iteratively by a sequence of vectors $x_k \in x_0 + \mathscr{K}_{k,t}(A,r_0)$ ($k>0$), obtained by imposing the Petrov-Galerkin condition, $\; r_k \perp \mathscr{K}_{k,t}(A,r_0)$, where $\mathscr{K}_{k,t} (A, r_0)$ is the enlarged Krylov subspace of dimension at most $tk$. $\mathscr{K}_{k,t} (A, r_0)$ is based on a domain decomposition of $A$ into $t$ disjoint subdomains, $\delta_i$ with $\delta = \bigcup\limits_{i = 1}^t \delta_i = \{1,2,\cdots,n\}$,  and for $i\neq j$,  $\delta_i\bigcap\delta_j=\phi$  \vspace{1mm}
\begin{eqnarray}\mathscr{K}_{k,t} (A, r_0) &=& \mbox{span}\{T^t(r_0), AT^t(r_0), A^2 T^t(r_0),\cdots , A^{k-1} T^t(r_0) \} \nonumber \\
&=&\mbox{span}\{ T_1^t(r_0),T_2^t(r_0), \cdots , T_t^t(r_0), \nonumber \\
&& \qquad \;\; AT_1^t(r_0),AT_2^t(r_0), \cdots , AT_t^t(r_0), \nonumber \\
&& \qquad \;\; \cdots  , \nonumber\\ \vspace{8mm}
&& \qquad \;\; A^{k-1} T_1^t(r_0), A^{k-1} T_2^t(r_0), \cdots , A^{k-1} T_t^t(r_0) \}. \nonumber \\ \nonumber
\end{eqnarray}\vspace{-6mm}

\noindent
$T^t(r_0) = \{ T_1^t(r_0),T_2^t(r_0), \cdots , T_t^t(r_0)\}$ is an operator that splits $r_0$ into $t$ vectors, 
with $T_i^t(r_0)$ being the operator that projects vector $r_0$ on subdomain $i$ of matrix $A$, $\delta_i$. The matrix $[T^t(v)]$ is shown in \eqref{eq:Tt}.

 The enlarged CG versions introduced in \cite{EKS} are MSDO-CG, LRE-CG, SRE-CG, SRE-CG2, and the truncated SRE-CG2. Note that in \cite{ECG} a block CG method based on SRE-CG is proposed, whereby one system with multiple right-hand sides is solved by defining $R_0$  using the enlarged Krylov subspace, i.e. $R_0 = [T^t(r_0)]$, and the solution of $Ax = b$ is the sum of the block solution $X_k$.   
The enlarged CG versions introduced in \cite{EKS} perform block operations, yet they are not Block methods as they do not solve a  system with multiple right-hand sides. Moreover, the residual vectors $r_k = b-Ax_k$ are not orthogonal to the enlarged Krylov subspace by definition. Hence the need for A-orthonormalizing the basis vectors to impose the Petrov-Galerkin condition.

Recently, s-step Enlarged CG versions were introduced \cite{sstepECG} by merging $s$ iterations of the enlarged CG methods.
As shown in \cite{EKS} and \cite{sstepECG} the Enlarged CG methods and their s-step versions converge in less iterations than the classical CG, but at the expense of requiring more memory storage than CG. Thus in this paper we explore different options for reducing the memory requirements of these enlarged CG methods, specifically SRE-CG2 (section \ref{sec:SRECG2}) and MSDO-CG (section \ref{sec:msdocg}), without affecting much their convergence. Apart from the truncated versions introduced in \cite{EKS} and the well-known restarting versions (section \ref{sec:trunc}), we introduce flexibly enlarged versions (section \ref{sec:flex}) where after some iteration (to be determined) the number of computed basis vectors is reduced to half. Then, all these versions are tested in section \ref{sec:testing}, to check their convergence along with the reduced memory requirements.

\section{Enlarged CG Methods}
The enlarged CG methods are projection iterative methods that seek at iteration $k$ an approximate solution to the $n\times n$ system of linear equations $Ax=b$ from the enlarged Krylov Subspace by imposing the Petrov-Galerkin constraint on the $k^{th}$ residual $r_k$,
\begin{enumerate}
    \item Subspace Condition $x_k\in x_0+\mathscr{K}_{k,t}(A,r_0)$
    \item Orthonality Condition $r_k \perp \mathscr{K}_{k,t}(A,r_0)$
\end{enumerate}
where $x_0$ is the initial guess, $r_0=b-Ax_0$ the initial residual, and $r_k=b-Ax_k$. The enlarged Krylov subspace is based on a domain decomposition of the matrix $A$, where the index domain $\delta = \{1,2,...,n\}$ is divided into $t$ disjoint subdomains $\delta_i$, with $\delta_i\bigcap\delta_j=\phi$ for $i\neq j$, and $\delta = \bigcup\limits_{i = 1}^t \delta_i$. Then, the enlarged Krylov subspace is given by
$$\mathscr{K}_{k,t}(A,r_0) = \mbox{span}\,\{T^t(r_0), AT^t(r_0),...,A^{k-1}T^t(r_0)\},$$ where $T^t(r_0)$ is an operator that projects $r_0$ on the $t$ subdomains $\delta_i$ producing a set of $t$ vectors $\{T_1^t(r_0),T_2^t(r_0),\cdots ,T_t^t(r_0)\}$ as shown in \eqref{eq:Tt} with $T_j^t(r_0)$ being the  operator that projects $r_0$ on the subdomain $\delta_j$.
\subsection{SRE-CG2}\label{sec:SRECG2}
 The short recurrence enlarged CG versions are iterative enlarged Krylov subspace projection methods that build at the  $k^{th}$ iteration, an A-orthonormal candidate basis $Q_k$ ($Q_k^\mathsf{T}AQ_k = I$) for the enlarged Krylov subspace $\mathscr{K}_{k,t}(A,r_0)$  and approximate the solution, $x_k = x_{k-1}+Q_k\alpha_k$, by imposing the orthogonality condition on $r_k = r_{k-1} - AQ_k\alpha_k$, ($r_k \perp \mathscr{K}_{k,t}$). Then, $\alpha_k=(Q_k^\mathsf{T}AQ_k)^{-1}(Q_k^\mathsf{T}r_k)$ is obtained by minimizing $\phi(x_{k-1}+Q_k\alpha)$ where $Q_k$ is an $n\times kt$ matrix, $.^\mathsf{T}$ denotes the transpose of a matrix, 
 and $$\phi(x) = \frac{1}{2}x^\mathsf{T}Ax - x^\mathsf{T}b.\vspace{1mm}
 $$
 
It was proven in \cite{EKS} that at each iteration $k\geq 4$, the $t$ newly computed basis vectors $W_k = AW_{k-1}$ are A-orthogonal to $W_i$ for $i\leq k-3$, where $W_1 = [T^t(r_0)]$ is the matrix containing the $t$ vectors of $T^t(r_0)$. This leads to a short recurrence in $x_k = x_{k-1}+W_k\alpha_k$ and $r_k = r_{k-1} - AW_k\alpha_k$, where $\alpha_k = W_k^\mathsf{T}r_{k-1}$. Moreover, $W_k$ has to be A-orthonormalized only against $W_{k-1}$ and $W_{k-2}$ using CGS2 A-orthonormalization method (Algorithm 18, \cite{sophiethesis}), and then against itself using A-CholQR \cite{A-ortho} or Pre-CholQR \cite{A-ortho2, sophiethesis}. This is the SRE-CG method that requires storing at most three block vectors $W_{k-2}, W_{k-1}, W_k$.

However, in finite arithmetic there is a loss of A-orthogonality at the $k^{th}$ iteration between the vectors of $Q_k = [W_1, W_2, ..., W_k]$ which is reflected on the number of iterations needed till convergence in the SRE-CG method. Moreover, once $W_k = AW_{k-1}$ is A-orthonormalized against all previous $W_i$'s for $i=1,2,..,k-1$ in the SRE-CG2 method (Algorithm \ref{alg:SRE-CG2}), the number of iterations was vastly reduced as compared to SRE-CG. Yet, this comes at the expense of performing more operations for the A-orthonormalization per iteration and requiring more memory storage, specifically storing all the block vectors $W_{1}, W_{2}, \cdots , W_k$.

 \begin{algorithm}[h!]
\centering
\caption{ SRE-CG2  }
{\renewcommand{\arraystretch}{1.3}
\begin{algorithmic}[1]
\Statex{\textbf{Input:} $A$,  $n \times n$ symmetric positive definite matrix; $k_{max}$, maximum allowed iterations}
\Statex{\qquad \quad $b$,  $n \times 1$ right-hand side; $x_0$, initial guess; $\epsilon$, stopping tolerance}
\Statex{\textbf{Output:} $x_k$, approximate solution of the system $Ax=b$}\vspace{2mm}
\State$r_0 = b - Ax_0$,\;\; $\rho_0 = ||r_0||_2$ ,
\;\; $k = 1$; 
\While {( ${\rho}_{k-1} > \epsilon {\rho_0}$ and $k < k_{max}$ )}
\If {($k==1$)}\vspace{1mm}
\State A-orthonormalize $W_k = [T^t(r_{k-1})]$,  and let $Q = W_k$ 
\Else \vspace{0.5mm}
\State Let $W_k = AW_{k-1}$
\State A-orthonormalize $W_k$ against $Q$
\State A-orthonormalize $W_k$ and let $Q = [Q \; W_k]$ \vspace{0.5mm}
\EndIf\vspace{1mm}
\State ${\alpha}_k = W_k^\mathsf{T} r_{k-1}$
\State $x_k = x_{k-1} + W_k{\alpha}_k $  
\State $r_k = r_{k-1} - AW_k{\alpha}_k $ 
\State $\rho_k = ||r_{k}||_2$, \;\;   $k = k+1$
\EndWhile
\end{algorithmic}}
\label{alg:SRE-CG2}
\end{algorithm}

\subsection{MSDO-CG}\label{sec:msdocg}
The MSDO-CG method \cite{EKS} computes search directions that belong to the enlarged Krylov Subspace $ \mathscr{K}_{k,t} (A,r_0)$, rather than computing basis vectors. Specifically, at each iteration $k$,  $t$ search directions $P_k$ are computed as in \eqref{eq:pk} and then \\ A-orthonormalized against all $P_i$'s ($i<k$), to impose the orthogonality condition on \\$r_k = r_{k-1} - AP_k\alpha_k$.

\begin{equation}\label{eq:pk}
    \begin{cases}
    P_k =  \mathscr{T}^t_{k-1} + P_{k-1}diag(\beta_k)&\\
    P_0 = \mathscr{T}^t_{0}&
\end{cases}\end{equation}
 where $\beta_k = - P_{k-1}^t A r_{k-1}$, and $\mathscr{T}^t_{i} = [T^t(r_i)]$ is the matrix containing the $t$ vectors of $T^t(r_i)$. 
  Then, the approximate solution is defined as $x_k = x_{k-1} + P_k\alpha_k$, where $\alpha_k  =  P_k^tr_{k-1}$ is computed by minimizing $\phi(x_{k-1} + P_k\alpha)$, as shown in Algorithm \ref{alg:MSDO-CG}. 
 
   \begin{algorithm}[h!]
\centering
\caption{ MSDO-CG  }
{\renewcommand{\arraystretch}{1.3}
\begin{algorithmic}[1]
\Statex{\textbf{Input:} $A$,  $n \times n$ symmetric positive definite matrix; $k_{max}$, maximum allowed iterations}
\Statex{\qquad \quad $b$,  $n \times 1$ right-hand side; $x_0$, initial guess; $\epsilon$, stopping tolerance }
\Statex{\textbf{Output:} $x_k$, approximate solution of the system $Ax=b$}\vspace{2mm}
\State$r_0 = b - Ax_0$,\;\; $\rho_0 = ||r_0||_2$ ,
\;\;$k = 1$,
\While {( ${\rho}_{k-1} > \epsilon\; {\rho_0} $ and $k < k_{max}$ )}
\If {($k==1$)}
\State A-orthonormalize $P_1 = [T^t(r_{k-1})]$,   let $V_1=AP_1$ and $Q = P_1$ 
\Else 
\State Let $\beta_k = - V_{k-1}^\mathsf{T} r_{k-1}$
\State Let $P_k = [{T}^t(r_{k-1})]+P_{k-1}diag(\beta_k)$ 
\State A-orthonormalize $P_k$ against $Q$
\State A-orthonormalize $P_k$,  let $V_k=AP_k$ and $Q = [Q \; P_k]$ 
\EndIf
\State ${\alpha}_k = P_k^\mathsf{T} r_{k-1}$ 
\State $x_k = x_{k-1} + P_k{\alpha}_k $ 
\State $r_k = r_{k-1} - V_k{\alpha} $  
\State $\rho_k = ||r_{k}||_2$, $k = k+1$  
\EndWhile
\end{algorithmic}}
\label{alg:MSDO-CG}
\end{algorithm}

In \cite{sstepECG}, with the aim of introducing s-step versions of the enlarged CG method, a modified MSDO-CG was introduced that searches for the solution in the modified Krylov subspace 
$$x_k\in x_0+ \overline{\mathscr{K}}_{k,t} (A,r_0)\vspace{-1mm}$$
$$ \overline{\mathscr{K}}_{k,t} (A,r_0) = \mbox{span}\,\{T^t(r_0), T^t(r_1), \cdots, T^t(r_{k-1})\}$$ by building an A-orthonormal basis similarly to SRE-CG2, with the difference that at each iteration  $W_k = [T^t(r_{k-1})]$. Thus, the Modified MSDO-CG algorithm is the same as the SRE-CG2 algorithm \ref{alg:SRE-CG2} with the exception of line 6 that is replaced by:\vspace{1mm}
$$\mbox{Let } W_k = [T^t(r_{k-1})].\vspace{1mm} $$ 

At iteration $k$ in both MSDO-CG and modified MSDO-CG algorithms, there is a need to A-orthonomalize a block of $t$ vectors, be it $P_k$ or $W_k$, against all $k-1$ previous  blocks of $t$ vectors. Thus, there is a need to  store all these $k$ block vectors, even though we have short recurrence formulae for $x_k$ and $r_k$.

\section{Truncated and Restarted Enlarged CG methods}\label{sec:trunc}
In the three methods, SRE-CG2, MSDO-CG, and Modified MSDO-CG, the iterates have short recurrence formulae. However, due to the A-orthonormalization process there is a need to store all  the generated block vectors. One option is to truncate the A-orthonormalization process by A-orthonormalizing $W_k$ against the previous $trunc$ blocks only, for $  2 \leq trunc \leq k_{max}$. Thus, requiring the storage of only the last $trunc+1$ block vectors. 

The truncated SRE-CG2 method was introduced in \cite{EKS} to reduce the memory requirements of SRE-CG2. Note that for $trunc = 2$, we get the SRE-CG method, and for  $trunc = k_{max}$ we get the SRE-CG2 method. The truncated SRE-CG2 algorithm differs from Algorithm \ref{alg:SRE-CG2} in line 8, where for $k>trunc$\vspace{1mm} $$Q= [W_{k-trunc+1}, W_{k-trunc+2}, \cdots, W_k].$$  

As for MSDO-CG or Modified MSDO-CG, it is possible to truncate the A-orthonormalization process similarly to SRE-CG2.  The truncated MSDO-CG algorithm differs from Algorithm \ref{alg:MSDO-CG} in line 9, where for $k>trunc$\vspace{1mm}  $$Q= [P_{k-trunc+1}, P_{k-trunc+2}, \cdots, P_k].$$ Yet, unlike SRE-CG2, there is no guarantee that the method will converge, as theoretically there is a need for all the search directions $P_k$ or all the basis vectors $W_k$ to be A-orthonormal. 

On the other hand, restarting SRE-CG2 or MSDO-CG or Modified MSDO-CG after $j$ iterations implies that all the information obtained in the matrix $Q$ is lost, be it the basis vectors of $\mathscr{K}_{j,t}$, or the search directions from $\mathscr{K}_{j,t}$ or the basis vectors of $\overline{\mathscr{K}}_{j,t}$ respectively.  Then the last approximate solution $x_j$ is used as an initial guess for the new cycle of $j$ iterations. Thus, at most $j+1$ block vectors are stored. However, for the stopping criteria we always compare to the initial residual ($\ltwonorm{r_k}/\ltwonorm{r_0} < tol$ for $k\geq j$). Thus, restarted SRE-CG2 algorithm and restarted MSDO-CG are respectively Algorithms \ref{alg:SRE-CG2} and \ref{alg:MSDO-CG} where line 3 is replaced by \vspace{1mm} $$ \textbf{if}\;(k \;mod\; j \bf == 1) \;\textbf{ then}$$
Both truncated and restarted versions have the same fixed memory requirement for $j = trunc$. However, restarted versions requires less flops than their corresponding truncated versions, specifically in the A-orthonormalization  (line 7 of Algorithm\ref{alg:SRE-CG2}, line 8 of Algorithm\ref{alg:MSDO-CG}), since after $j = trunc$ iterations the size of the $Q$ matrix in the restarted versions varies from $n\times t$ to $n\times tj$, whereas it is fixed to $n\times tj$ in the truncated version.

Note that it is possible to restart SRE-CG2 and MSDO-CG not every $j$ iterations, but once some measure of change in the residuals is less than a restartTol. We use, similarly to the flexibly enlarged CG versions, the relative difference of the residual norm, i.e we restart once \vspace{-3mm}$$\dfrac{| \ltwonorm{r_{k+1}}- \ltwonorm{r_{k}}|}{\ltwonorm{r_{0}}} < restartTol.$$
These restarted SRE-CG2 and restarted MSDO-CG are respectively Algorithms \ref{alg:SRE-CG2} and \ref{alg:MSDO-CG} where line 3 is replaced by \vspace{1mm} $$ \textbf{if}\;(k \bf == 1) \;or\;(\,|\ltwonorm{r_{k+1}}- \ltwonorm{r_{k}}|/\ltwonorm{r_{0}}<restartTol\,)\textbf{ then}$$
However, in this case the maximum needed memory requirement is not known beforehand, as the restarts will depend on the used $restartTol$ and the matrix at hand.

\section{Flexibly  Enlarged CG methods} \label{sec:flex}

The concept behind  ``flexibly" enlarged CG versions is that after some iterations $k_F$ (to be determined) the number of computed  vectors, be it basis vectors or search direction vectors, is reduced to half. This implies that at iteration $k_F+1$ the $t/2$ computed basis or search direction vectors are based on the last residual $r_{k_F}$,  specifically $T^{t/2}(r_{k_F})$. Similarly to $T^t(r_0)$,  $T^{t/2}(r_{k_F})$ is the operator that projects the vector $r_{k_F}$ over the $t/2$ subdomains 
$\tilde{\delta}_i$  where for $i = 1,2,...,t/2$ \vspace{-1mm}
$$\tilde{\delta}_i = \delta_{2i-1} \bigcup  \delta_{2i} \mbox{ \;with\;} \delta = \bigcup\limits_{i = 1}^{t/2} \tilde{\delta}_i = \bigcup\limits_{i = 1}^{t} {\delta}_i.\vspace{-1mm}$$ The comparison between the $n\times t$ matrix $[T^{t}({v})]$ and the $n\times t/2$ matrix $[T^{t/2}({v})]$ for some vector $v$ that is partitioned into $t$ parts is shown in \eqref{eq:Tt}
\begin{equation}\label{eq:Tt}
v=  \begin{bmatrix}
*  \\
\vdots\\
*  \\
 {\color{cyan}*} \\
 {\color{cyan}\vdots}   \\ 
 {\color{cyan}*}\\
\vdots\\
 {\color{olive}*} \\
 {\color{olive}\vdots} \\
 {\color{olive}*}\\ 
{\color{violet}*}\\
{\color{violet}\vdots}\\
{\color{violet}*}\\
\end{bmatrix}\quad
[T^{t}(v)] = 
\begin{bmatrix}
* & 0 & &0  & 0 \\
\vdots & \vdots&&\vdots &\vdots\\
* & 0 &&0 & 0 \\
0 & {\color{cyan}*} && \vdots &\vdots\\
\vdots  &  {\color{cyan}\vdots}  &&\vdots&   \vdots\\ 
0& {\color{cyan}*}&&0&0\\
&&\ddots&&\\
0 & 0 &&   {\color{olive}*}&0\\
\vdots & \vdots  &&{\color{olive}\vdots}&\vdots\\
0& 0  &&{\color{olive}*}&0\\ 
0&0&&0&{\color{violet}*}\\
\vdots& \vdots&&\vdots&{\color{violet}\vdots}\\
0&0&&0&{\color{violet}*}\\
\end{bmatrix}_{n \times t} \quad
[T^{\frac{t}{2}}(v)]=
\begin{bmatrix}
*  && 0 \\
\vdots && \vdots\\
* && 0  \\
 {\color{cyan}*} && 0\\
 {\color{cyan}\vdots}  && \vdots  \\ 
 {\color{cyan}*}&&0\\
&\ddots&\\
0 && {\color{olive}*} \\
\vdots && {\color{olive}\vdots} \\
0&& {\color{olive}*}\\ 
0&&{\color{violet}*}\\
\vdots&& {\color{violet}\vdots}\\
0&&{\color{violet}*}\\
\end{bmatrix}_{n \times \frac{t}{2}}
\end{equation}
We start by defining the Flexibly  enlarged Krylov subspace and proving it is a superset to the Krylov subspace, then we define the switching condition and the flexibly SRE-CG2 and MSDO-CG algorithms. Finally, we derive the preconditioned flexibly SRE-CG2 and MSDO-CG methods.  
\subsection{Flexibly  Enlarged Krylov Subspace}
\noindent With this reduction in dimension, the approximate solution $x_k$ no longer belongs to the enlarged subspace $x_0+\mathscr{K}_{k,t}(A,r_0)$, but it belongs to the flexibly enlarged subspace\vspace{-5mm}\\
\begin{equation}
x_k \in 
    \begin{cases}
        x_0 + \mathscr{K}_{k,t}(A,r_0),& \quad if\; k\leq k_F\\
        x_0 + \mathscr{K}_{k_F,t}(A,r_0)+\mathscr{K}_{k-k_F,t/2}(A,r_{k_F}),& \quad if\; k> k_F
    \end{cases}
\end{equation}

If the flexibly enlarged method converges before $k_F$, then it is equivalent to the original enlarged method where $x_k \in x_0 + \mathscr{K}_{k,t}(A,r_0)$. In this case, as shown in \cite{EKS} the Krylov subspace $\mathcal{K}_k(A,r_0) \subseteq \mathscr{K}_{k,t}(A,r_0)$, which validates the use of the enlarged Krylov subspace.
If the flexibly enlarged method converges in $k>k_F$ iterations, then we need to prove that the flexibly enlarged Krylov Subspace is a superset to $\mathcal{K}_k(A,r_0)$.  The following theorems lead to this conclusion.

\begin{theorem}\label{thrm:kry1}
    Let $1 \leq k_F < k$. Then, the Krylov subspace $$\mathcal{K}_k(A,r_0)= \mathcal{K}_{k_F}(A,r_0)+\mathcal{K}_{k-k_F}(A,r_{k_F})$$
    where $r_{k_F}\in \mathcal{K}_{k_F+1}(A,r_0)$.
\end{theorem}
\begin{proof} We prove that each set is subset of the other.\\
    1. $\mathcal{K}_{k_F}(A,r_0)+\mathcal{K}_{k-k_F}(A,r_{k_F}) \subseteq \mathcal{K}_k(A,r_0)$:\\
    Let $y\in \mathcal{K}_{k_F}(A,r_0)+\mathcal{K}_{k-k_F}(A,r_{k_F})$, then since $r_{k_F}= \sum\limits_{j=1}^{k_F+1} c_jA^{j-1}r_0$
    \begin{eqnarray}
        y &=& \sum\limits_{i=1}^{k_F} a_iA^{i-1}r_0 + \sum\limits_{i=1}^{k-k_F} b_iA^{i-1}r_{k_F} \;=\;  \sum\limits_{i=1}^{k_F} a_iA^{i-1}r_0 + \sum\limits_{i=1}^{k-k_F} \sum\limits_{j=1}^{k_F+1}  b_ic_jA^{i+j-2}r_0  \nonumber\\  
                &=&  \sum\limits_{i=1}^{k_F} a_iA^{i-1}r_0 + \sum\limits_{i=1}^{k} \alpha_i A^{i-1} r_0  \;\; = \;\; \sum\limits_{i=1}^{k} \tilde{\alpha}_i A^{i-1} r_0 \;\;\in \;  \mathcal{K}_k(A,r_0)\nonumber
        \end{eqnarray}
    2. $\mathcal{K}_k(A,r_0) \subseteq \mathcal{K}_{k_F}(A,r_0)+\mathcal{K}_{k-k_F}(A,r_{k_F})$ by induction:
    \begin{itemize}
        \item Base Case: $k = k_F + 1$\\
        Let $y\in \mathcal{K}_k(A,r_0)$, then \begin{eqnarray}
            y &=&  \sum\limits_{i=1}^{k_F+1} a_iA^{i-1}r_0\;\;=\;\;\sum\limits_{i=1}^{k_F} a_iA^{i-1}r_0 + a_{k}A^{k_F}r_0 \nonumber\\&=& \sum\limits_{i=1}^{k_F} a_iA^{i-1}r_0 + \dfrac{a_{k}}{c_k}r_{k_F} -\dfrac{a_{k}}{c_k} \sum\limits_{j=1}^{k_F} c_jA^{j-1}r_0 \;\;\in \mathcal{K}_{k_F}(A,r_0)+\mathcal{K}_{1}(A,r_{k_F})\nonumber
        \end{eqnarray}
        \item Suppose that $\mathcal{K}_k(A,r_0) \subseteq \mathcal{K}_{k_F}(A,r_0)+\mathcal{K}_{k-k_F}(A,r_{k_F})$.
        \item Prove that $\mathcal{K}_{k+1}(A,r_0) \subseteq \mathcal{K}_{k_F}(A,r_0)+\mathcal{K}_{k-k_F+1}(A,r_{k_F})$.\vspace{2mm}\\
        Let $y_{k+1}\in \mathcal{K}_{k+1}(A,r_0)$, and $y_{k}\in \mathcal{K}_{k}(A,r_0)  \subseteq \mathcal{K}_{k_F}(A,r_0)+\mathcal{K}_{k-k_F}(A,r_{k_F})$, then \vspace{-1mm}
        \begin{eqnarray}
            y_{k+1} &=&  \sum\limits_{i=1}^{k+1} a_iA^{i-1}r_0\;\;=\;\;\sum\limits_{i=1}^{k} a_iA^{i-1}r_0 + a_{k+1}A^{k}r_0 \;\;=\;\;\sum\limits_{i=1}^{k} a_iA^{i-1}r_0+ a_{k+1}A^{k-k_F}A^{k_F}r_0 \vspace{-2mm}\nonumber\\
            &=&\sum\limits_{i=1}^{k} a_iA^{i-1}r_0 + \dfrac{a_{k+1}}{c_{k_F+1}}A^{k-k_F}r_{k_F} -\dfrac{a_{k+1}}{c_{k_F+1}} \sum\limits_{j=1}^{k_F} c_jA^{k-k_F}A^{j-1}r_0  \vspace{-2mm}\nonumber\\
            &=&\sum\limits_{i=1}^{k} a_iA^{i-1}r_0 + \sum\limits_{j=k-k_F+1}^{k} \tilde{c}_jA^{j-1}r_0 + \dfrac{a_{k+1}}{c_{k_F+1}}A^{k-k_F}r_{k_F}    \vspace{-2mm}\nonumber  \\
&=&  \sum\limits_{i=1}^{k_F} a_iA^{i-1}r_0 + \sum\limits_{i=1}^{k-k_F} b_iA^{i-1}r_{k_F} +\dfrac{a_{k+1}}{c_{k_F+1}}A^{k-k_F}r_{k_F} \;\;=\;\;y_k+\dfrac{a_{k+1}}{c_{k_F+1}}A^{k-k_F}r_{k_F}\vspace{-2mm}\nonumber     \\    
        \implies  y_{k+1}  &=& \sum\limits_{i=1}^{k_F} a_iA^{i-1}r_0+ \sum\limits_{i=1}^{k-k_F+1} b_iA^{i-1}r_{k_F}  \;\;\in \;\mathcal{K}_{k_F}(A,r_0)+\mathcal{K}_{k-k_F+1}(A,r_{k_F}).\nonumber
        \end{eqnarray}
    \end{itemize}\vspace{-8mm}
\hspace{15cm}
\end{proof}\vspace{6mm}
\begin{theorem}\label{thrm:kry2}
    The Krylov subspace $\mathcal{K}_k(A,v)$ is a subset of the enlarged Krylov subspace $\mathscr{K}_{k,t}(A,v)$ for any vector $v$ and integer $t\geq 1$.
\end{theorem}

\begin{proof}
    Let $y\in \mathcal{K}_k(A,v)$, where $\mathcal{K}_k(A,v)=\{v, Av,\cdots ,A^{k-1}v\}$. Then,
    $$y = \sum\limits_{i=1}^{k} a_iA^{i-1}v = \sum\limits_{i=1}^{k} a_iA^{i-1} [T^t(v)]* \mathbbm{1}_t =  \sum\limits_{i=1}^{k} \sum\limits_{j=1}^{t} a_iA^{i-1}T^t_j(v)\;\; \in\; \mathscr{K}_{k,t}(A,v)\vspace{1mm}$$
    \noindent since by definition \eqref{eq:Tt} $v= [T^t(v)]* \mathbbm{1}_t = [T^t_1(v)\, T^t_2(v)\,\cdots T^t_t(v)]* \mathbbm{1}_t = \sum\limits_{j=1}^{t} T^t_j(v).$\quad\quad
\end{proof}

\noindent As a corollary of theorems \ref{thrm:kry1} and \ref{thrm:kry2} we get the following result.
\begin{corollary}
The Krylov subspace $\mathcal{K}_k(A,r_0)$ is a subset of the flexibly enlarged Krylov subspace for $k>k_F$ and the even integer $t\geq 2$, i.e. $$\mathcal{K}_k(A,r_0) \subseteq \mathscr{K}_{k_F,t}(A,r_0)+\mathscr{K}_{k-k_F,t/2}(A,r_{k_F})$$
 where $r_{k_F}\in \mathcal{K}_{k_F+1}(A,r_0)$.
\end{corollary}
\begin{proof}
By theorem \ref{thrm:kry1}, $\mathcal{K}_k(A,r_0)= \mathcal{K}_{k_F}(A,r_0)+\mathcal{K}_{k-k_F}(A,r_{k_F})$.\\Moreover, by theorem \ref{thrm:kry2}, $\mathcal{K}_{k_F}(A,r_0) \subseteq \mathscr{K}_{k_F,t}(A,r_0)$ and $\mathcal{K}_{k-k_F}(A,r_k) \subseteq \mathscr{K}_{k-k_F,t/2}(A,r_{k_F})$, which completes the proof.\qquad\qquad
\end{proof}

\subsection{Switch} It remains to decide when to switch to computing half the basis vectors. One option is to fix $k_F$. However, given that it is not know before hand the number of iterations till convergence, a premature or late switch may have a negative effect on the convergence or end up with an unattained memory reduction. Thus, we suggest to switch once  the relative difference of the residual norm is smaller than some predetermined switch tolerance (switchTol), i.e. 
$\dfrac{| \ltwonorm{r_{k+1}}- \ltwonorm{r_{k}}|}{\ltwonorm{r_{0}}} < switchTol$.

This is based on the observation that the norm of the residual in the ECG methods may stagnate at several stages as shown in figures \ref{fig:resAni} for $t=32$ partitions.\vspace{-2mm}
\begin{figure}[H]
  \centering
\includegraphics[width=0.35\textwidth]{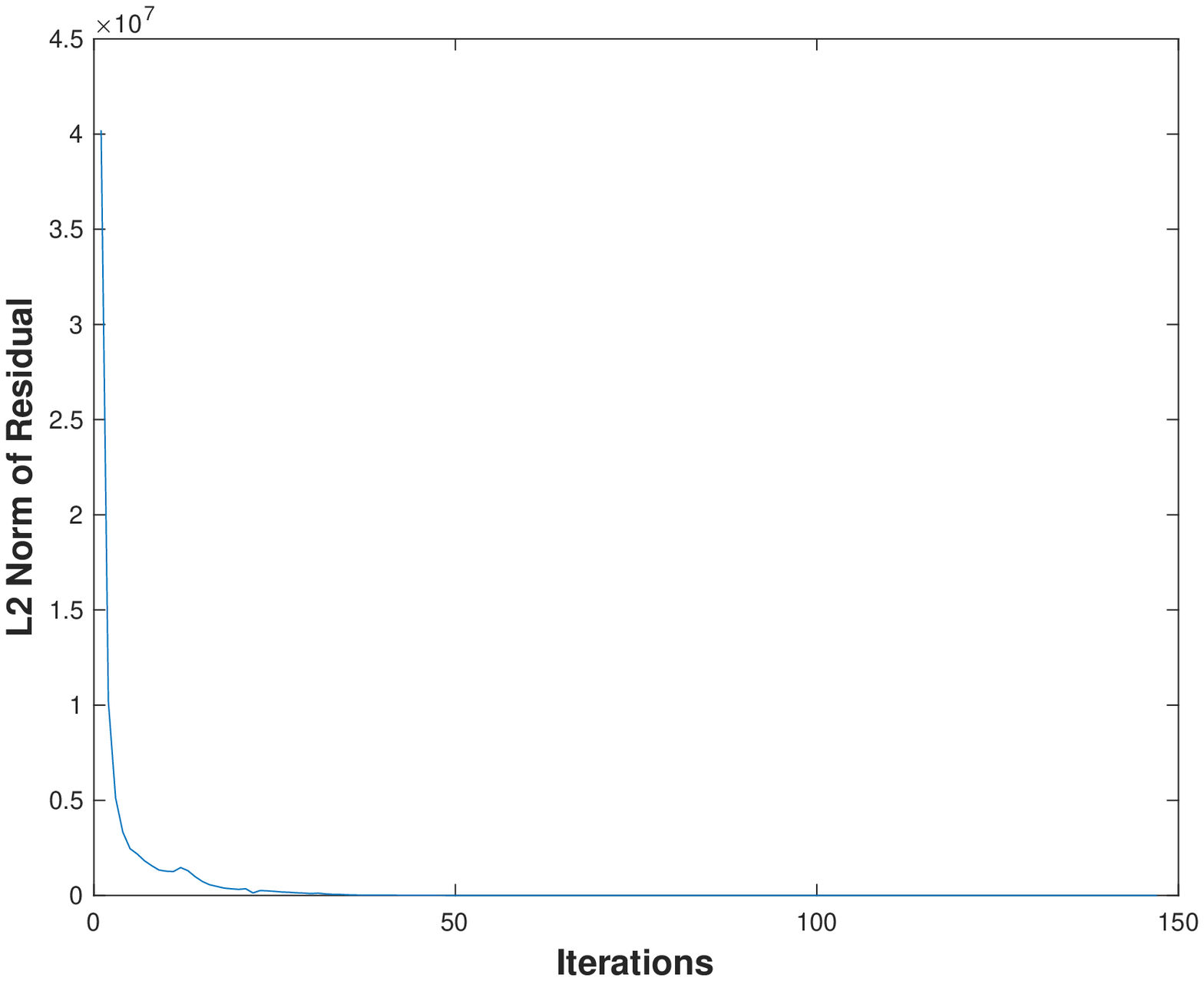} \hspace{-5mm}
\includegraphics[width=0.35\textwidth]{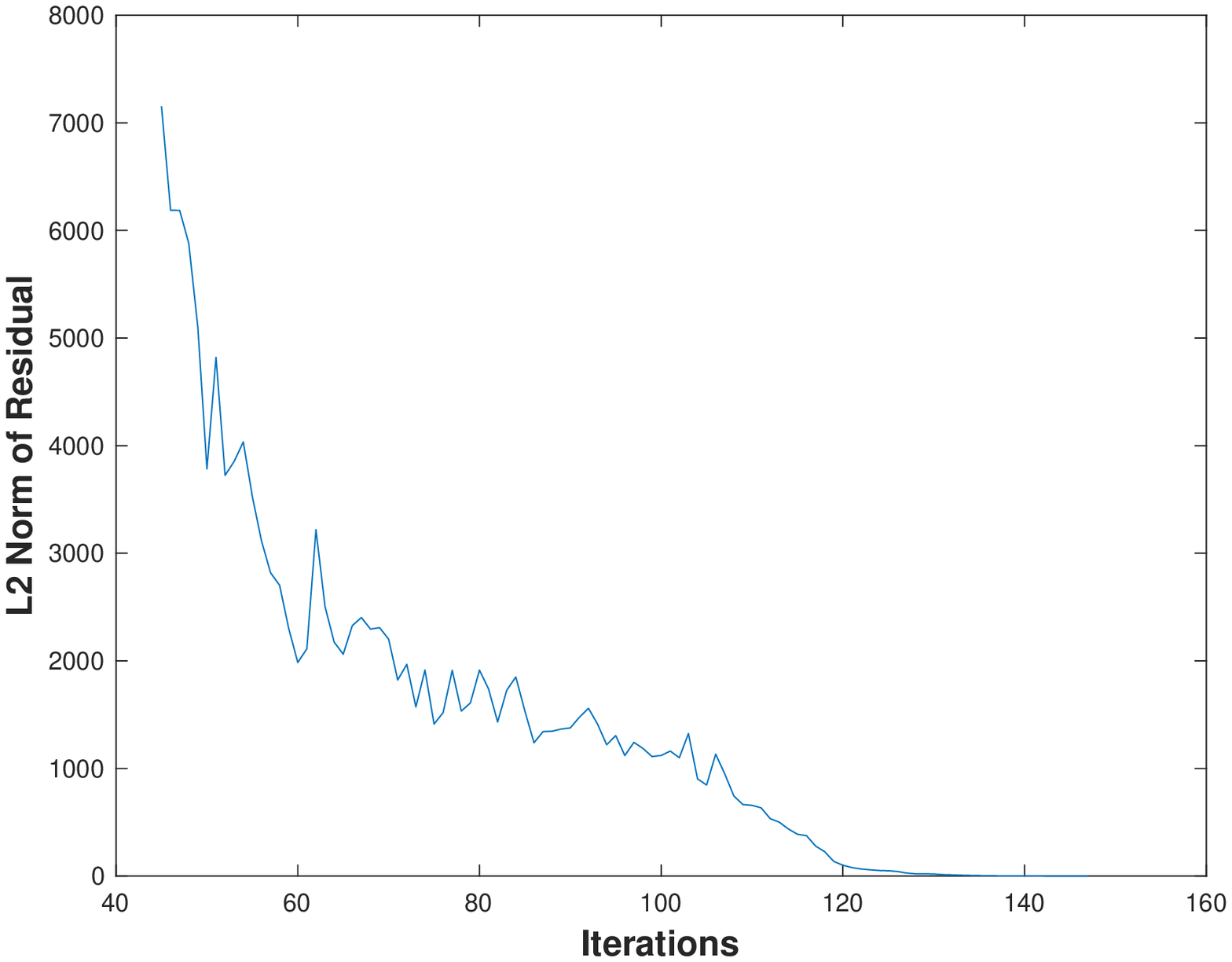} \hspace{-5mm}
\includegraphics[width=0.35\textwidth]{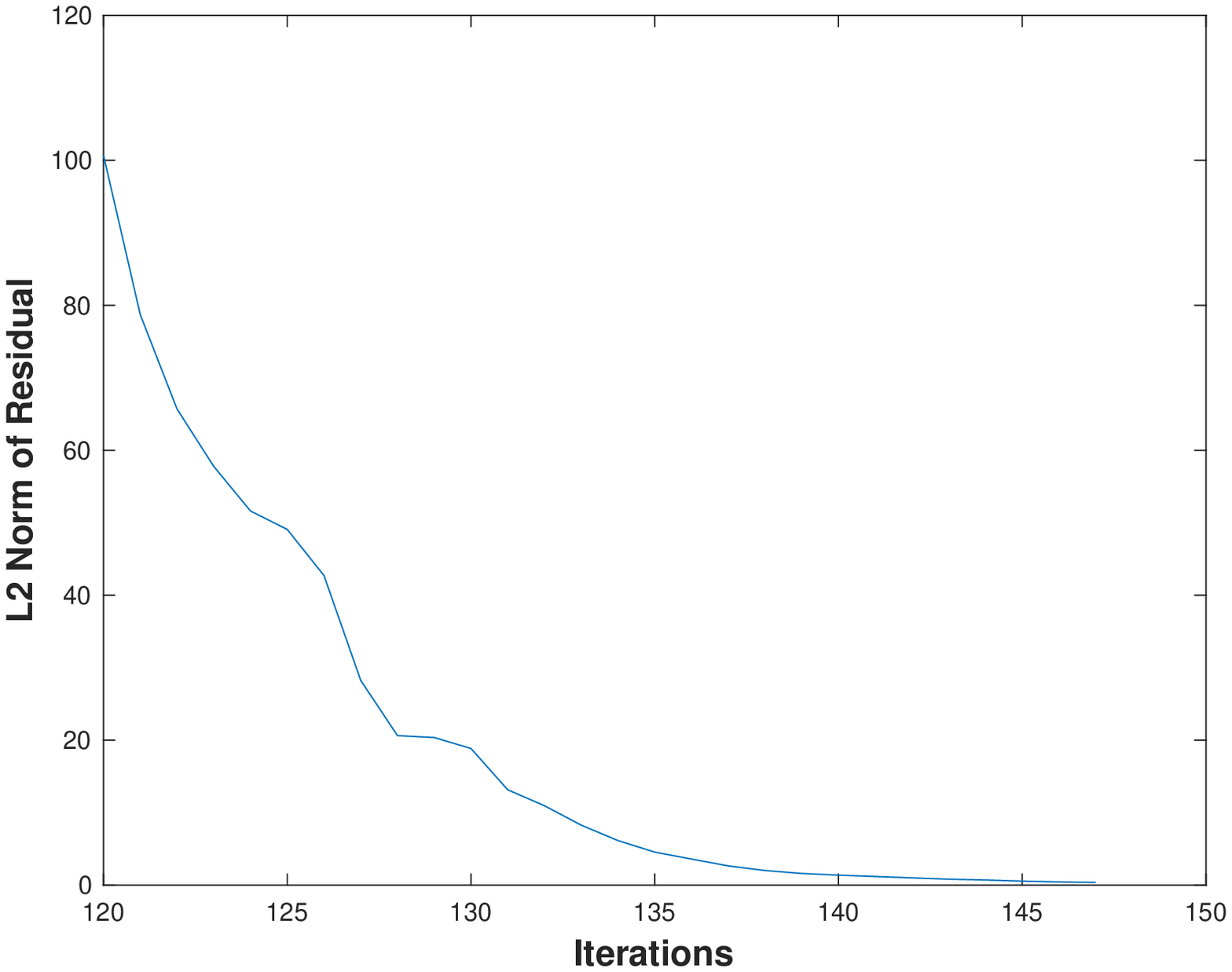}
  \caption{Norm of residual vector for Ani3D matrix using SRE-CG2 for $32$ partitions over all the 156 iterations needed to convergence, zoomed view from iteration 45, and zoomed view from iteration 120.}\label{fig:resAni}\vspace{-5mm}
\end{figure}  
\subsection{Flexibly  SRE-CG2 Algorithm} Applying this general flexibly enlarged CG method to SRE-CG2 leads to Algorithm \ref{alg:flexibleSRE-CG2}.  The main difference with Algorithm \ref{alg:SRE-CG2} is that there is at most one iteration $k$ where after which the dimension of $W_j$ for $j\geq k$ is half that of $W_{k-1}$. This iteration $k = k_F + 1$ is when the relative difference of the residual norm is less than the given switchTol, and thus $W_k$ is set to $ [T^{1/2}(r_{k-1})]$. Specifically, line 6 in Algorithm \ref{alg:SRE-CG2} is replaced by lines 6-11 in Algorithm \ref{alg:flexibleSRE-CG2}. This is similar to restarting with an initial guess set to $x_{k-1}$. Yet, the difference with restarted SRE-CG2 is that in the flexibly enlarged version  at the switch iteration $W_k$ is A-orthonormalized against all the previous vectors. 

\begin{algorithm}[H]\label{alg1}
\caption{Flexibly  SRE-CG2}
{\renewcommand{\arraystretch}{0.8}
\begin{algorithmic}[1]
\Statex{\textbf{Input:} $A$,  $n \times n$ symmetric positive definite matrix; $k_{max}$, maximum allowed iterations}
\Statex{\qquad \quad $b$,  $n \times 1$ right-hand side; $x_0$, initial guess; $\epsilon$, stopping tolerance; switchTol}
\Statex{\textbf{Output:} $x_k$, approximate solution of the system $Ax=b$}\vspace{2mm}
\State$r_0 = b - Ax_0$;\;\; $\rho_0 = ||r_0||_2$ ;\;\; $counter = 0;\;\; tol1 = 1$;\;\;$k = 1$; 
\vspace{1mm}
\While { ( $ {\rho_{k-1}} > \epsilon \;\rho_0 $ \;and\; $k < k_{max}$ )}
 \If {( $ k==1 $ )} \vspace{1mm}
\State A-orthonormalize $W_k = [T^t(r_0)]$,  and let $Q = W_k$ 
 \Else
 \If{( $ tol1 < switchTol $ \; and \; $ counter == 0 $ )}\vspace{1mm}
\State $W_{k}=[T^{t/2}(r_{k-1})]$;
 \State $counter = 1$;
 \Else \vspace{1mm}
\State Let $W_{k}=AW_{k-1}$;
\EndIf \vspace{1mm}
\State $A$-orthonormalize the vectors of $W_{k}$ against $Q$;
\State $A$-orthonormalize the vectors of $W_{k}$ and let $Q=[Q, W_{k}]$;
\EndIf\vspace{1mm}
\State ${\alpha}_{k}=W_{k}^\mathsf{T}r_{k-1}$;
\State $x_{k}=x_{k-1}+W_{k}{\alpha}_{k}$; 
\State $r_{k}=r_{k-1}-AW_{k}{\alpha}_{k}$;
\State $\rho_{k}=\|r_{k}\|_{2}$;\;\;\;$tol1 = ({\rho_{k}-\rho_{k-1}})/{\rho_{0}}$;\;\;\;$k=k+1$;
\EndWhile
\end{algorithmic}
}
\label{alg:flexibleSRE-CG2}
\end{algorithm}

\subsection{Flexibly  MSDO-CG Algorithm}
Applying the idea of flexibly enlarged CG methods to the case of MSDO-CG leads to algorithm \ref{alg:flexMSDO-CG}. Once the relative difference of the residual norm is less than the given switchTol, the number of introduced search directions at iteration $k$ is halved by setting $P_k=[T^{t/2}(r_{k-1})]$ and A-orthonormalizing it against all previous search directions. Thus, lines 6-7 in algorithm \ref{alg:MSDO-CG} are replaced by lines 6-14 in algorithm \ref{alg:flexMSDO-CG}. 
\begin{algorithm}[h!]
\centering
\caption{ Flexibly  MSDO-CG  }
{\renewcommand{\arraystretch}{1.3}
\begin{algorithmic}[1]
\Statex{\textbf{Input:} $A$,  $n \times n$ symmetric positive definite matrix; $k_{max}$, maximum allowed iterations}
\Statex{\qquad \quad $b$,  $n \times 1$ right-hand side; $x_0$, initial guess; $\epsilon$, stopping tolerance; switchTol }
\Statex{\textbf{Output:} $x_k$, approximate solution of the system $Ax=b$}\vspace{2mm}
\State$r_0 = b - Ax_0$,\;\; $\rho_0 = ||r_0||_2$ ,\;\;  $counter = 0;\;\; tol1 = 1$;\;\;$k = 1$; 
\While {( ${\rho}_{k-1} > \epsilon\; {\rho_0} $ and $k < k_{max}$ )}
\If {($k==1$)}
\State A-orthonormalize $P_1 =[T^t(r_0)]$,   let $V_1=AP_1$ and $Q = P_1$ 
\Else 
 \If{( $ tol1 < switchTol $ \; and \; $ counter == 0 $ )}\vspace{1mm}
\State $P_{k}=[T^{t/2}(r_{k-1})]$;
 \State $counter = 1$;
 \Else \vspace{1mm}
\State Let $\beta_k = - V_{k-1}^\mathsf{T} r_{k-1}$
\If {($counter==0$)\;\;}  \;\;Let $P_k = [{T}^t(r_{k-1})]+P_{k-1}diag(\beta_k)$ 
\Else \;\;Let $P_k = [{T}^{t/2}(r_{k-1})]+P_{k-1}diag(\beta_k)$ 
\EndIf
\EndIf \vspace{1mm}
\State A-orthonormalize $P_k$ against $Q$
\State A-orthonormalize $P_k$,  let $V_k=AP_k$ and $Q = [Q \; P_k]$ 
\EndIf
\State ${\alpha}_k = P_k^\mathsf{T} r_{k-1}$ 
\State $x_k = x_{k-1} + P_k{\alpha}_k $ 
\State $r_k = r_{k-1} - V_k{\alpha}_k $  
\State $\rho_k = ||r_{k}||_2$,\;\;\;$tol1 = ({\rho_{k}-\rho_{k-1}})/{\rho_{0}}$;\;\; $k = k+1$  
\EndWhile
\end{algorithmic}}
\label{alg:flexMSDO-CG}
\end{algorithm}
\subsection{Preconditioned versions} Conjugate Gradient solves systems $Ax=b$ with symmetric positive definite (spd) matrices $A$. Thus,  preconditioned matrix $\widehat{A}$ should also be spd. Hence split preconditioning is used, where the preconditioner $M=LL^\mathsf{T}$  and the preconditioned matrix $\widehat{A} = L^{-1}AL^{-\mathsf{T}}$ are spd. Then, the equivalent preconditioned system is 
\begin{equation}\label{eq:prec}  \widehat{A}\widehat{x}=\widehat{b}\end{equation} where $\widehat{x} = L^\mathsf{T}x$ and $\widehat{b} = L^{-1}b$.
   
    The first option discussed in \cite{sophiethesis}, is to first solve $\widehat{A}\widehat{x}=\widehat{b}$ by replacing  $A$ by $\widehat{A}=L^{-1}AL^{-\mathsf{T}}$ and $b$ by $\widehat{b}=L^{-1}b$ in the algorithms. Then,   the solution $x$ is obtained by solving $\widehat{x} = L^\mathsf{T}x$. The first implication is that the $A$-orthonormalization is replaced by $L^{-1}AL^{-\mathsf{T}}$-orthonormalization (Algorithms 19 and 22 of \cite{sophiethesis}). The second is the need to replace every matrix-vector multiplication $Ax$ by $y_3=\widehat{A}\widehat{x} = L^{-1}AL^{-\mathsf{T}}\widehat{x}$ which requires solving $L^\mathsf{T}y_1=\widehat{x}$, computing $y_2 = Ay_1$ and solving $Ly_3=y_2$. Assuming $L$ is lower triangular, then this requires a sequence of backward substitution, matrix-vector multiplication and forward substitution. 
    
    The second option discussed in \cite{sstepECG} is to avoid using the  $L^{-1}AL^{-\mathsf{T}}$-orthonormalization and modifying the recurrence relations of $\alpha_k, x_k, r_k$ in preconditioned algorithms \ref{alg:flexibleSRE-CG2} and \ref{alg:flexMSDO-CG}.

Given system \eqref{eq:prec}, then the corresponding recurrence  relations for both preconditioned flexibly SRE-CG2 and MSDO-CG are:
\begin{eqnarray}
 \widehat{\alpha}_k &=& \widehat{Z}^\mathsf{T}_k \widehat{r}_{k-1} \nonumber \\
 \widehat{x}_k &=& \widehat{x}_{k-1} + \widehat{Z}_k\widehat{\alpha}_k \nonumber\\
 \widehat{r}_k &=& \widehat{r}_{k-1} - \widehat{A}\widehat{Z}_k\widehat{\alpha}_k \nonumber
\end{eqnarray}
where $\widehat{Z}_k$ is $\widehat{A}$-orthonormalized against $\widehat{Z}_i$, i.e.  $\widehat{Z}_k^\mathsf{T}\widehat{A}\widehat{Z}_i=0$ and $\widehat{Z}_i^\mathsf{T}\widehat{A}\widehat{Z}_i=I$. Moreover, in flexibly SRE-CG2 and in flexibly MSDO-CG $\widehat{Z}_k$ is given respectively by \begin{eqnarray} \widehat{Z}_k = \widehat{W}_k&=&\begin{cases}
    [T^t(\widehat{r}_0)] , &k=1\\
    [T^{t/2}(\widehat{r}_{k-1})] , &k=switchIt\\
    \widehat{A}  \widehat{W}_{k-1},& else
\end{cases}\label{eq:zsre}\\
\widehat{Z}_k = \widehat{P}_k&=&\begin{cases}
    [T^t(\widehat{r}_0)] , &k=1\\
    [{T}^t(\widehat{r}_{k-1})]+  \widehat{P}_{k-1}diag(\beta_k),& 2\leq k <switchIt\\
    [T^{t/2}(\widehat{r}_{k-1})] , &k=switchIt\\
     [{T}^{t/2}(\widehat{r}_{k-1})]+  \widehat{P}_{k-1}diag(\beta_k),& else
\end{cases}\label{eq:zmsd}
\end{eqnarray}
and $\beta_k = -(\widehat{A}  \widehat{P}_{k-1})^\mathsf{T} \,\widehat{r}_{k-1}$. 
Noting that  $\widehat{r}_{k} = \widehat{b} - \widehat{A}\widehat{x}_{k} = L^{-1}b - L^{-1}AL^{-\mathsf{T}}L^\mathsf{T}x_{k} =   L^{-1}r_{k}$, 
then the corresponding equations for $x_k$, and $r_k$ are:\vspace{2mm}
 \begin{eqnarray}
 \widehat{\alpha}_k &=& \widehat{Z}^\mathsf{T}_k \widehat{r}_{k-1} = \widehat{Z}^\mathsf{T}_k L^{-1}r_{k} = (L^{-\mathsf{T}}\widehat{Z}_k)^\mathsf{T}r_k \nonumber \\
 \widehat{x}_k &=& L^\mathsf{T}x_k = \widehat{x}_{k-1} + \widehat{Z}_k\widehat{\alpha}_k = L^\mathsf{T}{x}_{k-1} + \widehat{Z}_k\widehat{\alpha}_k \;\;\;
 \implies x_k = {x}_{k-1} + (L^{-\mathsf{T}}\widehat{Z}_k)\widehat{\alpha}_k \nonumber \\
  \widehat{r}_k &=& L^{-1}r_{k} = \widehat{r}_{k-1} - \widehat{A}\widehat{Z}_k\widehat{\alpha}_k = L^{-1}{r}_{k-1} - L^{-1}AL^{-\mathsf{T}}\widehat{Z}_k\widehat{\alpha}_k 
 \implies r_{k} = {r}_{k-1} - A(L^{-\mathsf{T}}\widehat{Z}_k)\widehat{\alpha}_k \nonumber
\end{eqnarray}
Let $Z_k = L^{-\mathsf{T}}\widehat{Z}_k$, then 
\begin{eqnarray}
\widehat{\alpha}_k &=& Z_k^\mathsf{T}r_k,\nonumber \\
 x_k &=& {x}_{k-1} + {Z}_k\widehat{\alpha}_k,\nonumber\\
 r_{k} &=& {r}_{k-1} - A{Z}_k\widehat{\alpha}_k.\nonumber
 \nonumber \vspace{4mm}
 \end{eqnarray}
 
 \noindent As for the $\widehat{A}$-orthonormalization, we require that  $\widehat{Z}_k^\mathsf{T} \widehat{A} \widehat{Z}_i = 0$ for some values of $i\neq k$. But $$\widehat{Z}_k^\mathsf{T} \widehat{A} \widehat{Z}_i = \widehat{Z}_k^\mathsf{T} L^{-1}AL^{-\mathsf{T}} \widehat{Z}_i = (L^{-\mathsf{T}}\widehat{Z}_k)^\mathsf{T}A(L^{-\mathsf{T}} \widehat{Z}_i) = {Z}_k^\mathsf{T}A{Z}_i.$$ 
 Thus, it is sufficient to A-orthonormalize $Z_k = L^{-\mathsf{T}}\widehat{Z}_k$ instead of  $\widehat{A}$-orthonormalizing $\widehat{Z}_k$. 
 \noindent Therefore, the recurrence formulae of $\widehat{\alpha}_k$, $x_k$ and $r_k$ in the preconditioned versions are identical to those of the unpreconditioned versions, and in both we A-orthonormalize a block of vectors $Z_k$. The only difference is in the construction of these block of vectors $Z_k$.
 
 \noindent Since  in \eqref{eq:zsre} and \eqref{eq:zmsd}, $T^j(\widehat{r}_{k}) = T^j(L^{-1}r_{k})$ and $\widehat{A}\widehat{Z}_{k-1} =  L^{-1}A{Z}_{k-1}$  then respectively
 \begin{eqnarray} {Z}_k = L^{-\mathsf{T}}\widehat{Z}_k = L^{-\mathsf{T}}\widehat{W}_k = {W}_k&=&\begin{cases}
   L^{-\mathsf{T}} [T^t(\widehat{r}_0)] , &k=1\\
   L^{-\mathsf{T}} [T^{t/2}(\widehat{r}_{k-1})] , &k=switchIt\\
    L^{-\mathsf{T}}\widehat{A}  \widehat{W}_{k-1} = M^{-1}A{W}_{k-1},& else
\end{cases}\label{eq:zsre2}\\
{Z}_k= L^{-\mathsf{T}}\widehat{Z}_k =  L^{-t}\widehat{P}_k={P}_k&=&\begin{cases}
    L^{-\mathsf{T}}[T^t(\widehat{r}_0)] , &k=1\\
   L^{-\mathsf{T}} [{T}^t(\widehat{r}_{k-1})]+  {P}_{k-1}diag(\beta_k),& 2\leq k <switchIt\\
   L^{-\mathsf{T}} [T^{t/2}(\widehat{r}_{k-1})] , &k=switchIt\\
    L^{-\mathsf{T}} [{T}^{t/2}(\widehat{r}_{k-1})]+  {P}_{k-1}diag(\beta_k),& else
\end{cases}\qquad \qquad\label{eq:zmsd2}
\end{eqnarray}
 and $\beta_k = -(\widehat{A}  \widehat{P}_{k-1})^\mathsf{T}\,\widehat{r}_{k-1} = - (L^{-1}AP_{k-1})^\mathsf{T}\, L^{-1}r_{k-1}= (M^{-1}AP_{k-1})^\mathsf{T}\, r_{k-1}=(AP_{k-1})^\mathsf{T}\,M^{-1}\, r_{k-1}$. 
 Thus, preconditioned flexibly SRE-CG2 algorithm is simply algorithm \ref{alg:flexibleSRE-CG2} where the definitions of $W_k$ in lines 4,7,10 are modified based on \eqref{eq:zsre2}. Preconditioned flexibly MSDO-CG is 
 summarized in algorithm \ref{alg:precflexMSDO-CG}. Note that if the preconditioner is a block diagonal preconditioner,
 with $t$ blocks that correspond to the initial $t$ partitions of  matrix $A$, then $[T^t(L^{-1}r_k)] = L^{-1}[T^t(r_k)]$ and $L^{-\mathsf{T}}[{T}^t(L^{-1}r_k)] = M^{-1}[T^t(r_k)]$. Similarly,  $[T^{t/2}(L^{-1}r_k)] = L^{-1}[T^{t/2}(r_k)]$ and $L^{-\mathsf{T}}[{T}^{t/2}(L^{-1}r_k)] = M^{-1}[T^{t/2}(r_k)]$. In this case, no need for split preconditioning.\vspace{-1mm}
\begin{algorithm}[H]
\centering
\caption{ Split Preconditioned Flexibly  MSDO-CG  }
{\renewcommand{\arraystretch}{1.3}
\begin{algorithmic}[1]
\Statex{\textbf{Input:} $A$,  $n \times n$ spd matrix; $M=LL^\mathsf{T}$, preconditioner; $k_{max}$, maximum iterations}
\Statex{\qquad \quad $b$,  $n \times 1$ right-hand side; $x_0$, initial guess; $\epsilon$, stopping tolerance; switchTol }
\Statex{\textbf{Output:} $x_k$, approximate solution of the system $Ax=b$}\vspace{2mm}
\State$r_0 = b - Ax_0$,\;\; $\rho_0 = ||r_0||_2$ ,\;\;  $counter = 0;\;\; tol1 = 1$;\;\;$k = 1$; 
\While {( ${\rho}_{k-1} > \epsilon\; {\rho_0} $ and $k < k_{max}$ )}
\If {($k==1$)}
\State A-orthonormalize $P_1 = L^{-\mathsf{T}}[T^t(L^{-1}{r}_0)]$,   let $V_1=AP_1$ and $Q = P_1$ 
\Else 
 \If{( $ tol1 < switchTol $ \; and \; $ counter == 0 $ )}\vspace{1mm}
\State $P_{k}=L^{-\mathsf{T}}[T^{t/2}(L^{-1}r_{k-1})]$;
 \State $counter = 1$;
 \Else \vspace{1mm}
\State Let $\beta_k = - V_{k-1}^\mathsf{T} (M^{-1}r_{k-1})$ and $\widehat{r}_{k-1} = L^{-1}r_{k-1}$
\If {($counter==0$)\;\;}  \;\;Let $P_k = L^{-\mathsf{T}}[T^t(\widehat{r}_{k-1})]+P_{k-1}diag(\beta_k)$ 
\Else \;\;Let $P_k = L^{-\mathsf{T}}[T^{t/2}(\widehat{r}_{k-1})]+P_{k-1}diag(\beta_k)$ 
\EndIf
\EndIf \vspace{1mm}
\State A-orthonormalize $P_k$ against $Q$
\State A-orthonormalize $P_k$,  let $V_k=AP_k$ and $Q = [Q \; P_k]$ 
\EndIf
\State ${\alpha}_k = P_k^\mathsf{T} r_{k-1}$ 
\State $x_k = x_{k-1} + P_k{\alpha}_k $ 
\State $r_k = r_{k-1} - V_k{\alpha}_k $  
\State $\rho_k = ||r_{k}||_2$,\;\;\;$tol1 = ({\rho_{k}-\rho_{k-1}})/{\rho_{0}}$;\;\; $k = k+1$ 
\EndWhile
\end{algorithmic}}
\label{alg:precflexMSDO-CG}
\end{algorithm}\vspace{-5mm}
\section{Testing}\label{sec:testing}
The implementation of all the discussed methods depend on the A-orthonormalization procedures that are detailed in \cite{sophiethesis} and \cite{EKS}. In all methods there is a need to A-orthonormalize some block of vectors against another block using CGS2 A-orthonormalization  (Algorithm 18 in \cite{sophiethesis}) and then A-orthonormalize the block itself using Pre-CholQR \cite{A-ortho2} (Algorithm 23 in \cite{sophiethesis}).

We compare the convergence behavior of the different discussed truncated, restarted (section \ref{sec:trunctest}) and flexibly enlarged CG (section \ref{sec:flextest}) versions for solving the system $Ax = b$ using different number of partitions ($t = 2, 4, 8, 16, 32$, and $64$ partitions).  Moreover, we test the preconditioned versions (section \ref{sec:prectest}). 
We first reorder/permute each matrix $A$ using Metis's kway partitioning \cite{metis} for $128$ subdomains. Then, the different larger subdomains for $t=2,4,8,16,32,64$ are defined by merging $128/t$ consecutive ones. The exact solution $x$ is chosen randomly using online MATLAB's rand function (rand('twister', 5489); x = 4*rand(numeq,1);) and the right-hand side is defined as $b = Ax$. The initial iterate is set to $x_0 = 0$, and the stopping criteria tolerance is set to $tol = 10^{-8}$ for all the matrices.

We test the methods using the matrices {\nho}, {\skyo}, {\skyto}, and {\anio}, that arise from different boundary value problems of convection diffusion equations, and generated using FreeFem++ \cite{freefem}. The main characteristics of the test matrices, including the condition number, number of CG iterations to convergence, and maximum allowed iterations $k_{max}$, are summarized in Table \ref{tab:testmatrices}. For a detailed description of the test matrices, refer to \cite{EKS}.

 \begin{table}[h!] \vspace{-4mm}
\centering
{\renewcommand{\arraystretch}{1.4}\footnotesize
\caption{The test matrices}\vspace{-2mm}
\begin{tabular}{|c|c|c|c|c|c|c|c|}
\hline
\multirow{2}{*}{\textbf{Matrix}} & \multirow{2}{*}{\textbf{Size}} & \multirow{2}{*}{\textbf{Nonzeros } } & \multirow{2}{*}{\textbf{2D/3D} }& \multirow{2}{*}{\textbf{Problem}}&\textbf{Condition }&\textbf{CG }&\multirow{2}{*}{$\mathbf{k_{max}}$}\\ 
 &  &   &  &&\textbf{ Number}&\textbf{ Iterations}&\\ 
\hline 
{\nh} & 10000 &49600  & 2D & Boundary value &6.01E+3&259& 500 \\ 
{\skyt} & 8000 & 53600  & 3D &  Skyscraper&1.12E+6&902&1500 \\
{\ani} & 8000& 53600  & 3D & Anisotropic Layers &2.01E+6&4179&5000\\
{\sky} & 10000 & 49600 & 2D & Boundary value&2.91E+7&5980&6000 \\
\hline
\end{tabular}\label{tab:testmatrices}}
\end{table}
\subsection{Truncated and Restarted Enlarged CG Methods}\label{sec:trunctest} We start by comparing the convergence behavior of truncated and restarted SRE-CG2, MSDO-CG and Modified MSDO-CG  for $trunc = 2, 20, 50$ and restart every $j=trunc$ iterations. \vspace{-6mm}

\begin{table}[H]
\setlength{\tabcolsep}{2.5pt}
\renewcommand{\arraystretch}{1.1}
\caption{\label{tab:SRECG-rest} Convergence (iteration  $\bf It$, relative error $\bf RelErr$) of truncated and restarted SRE-CG2 versions, with respect to number of partitions $\bf t$, truncation values $\bf trunc$, and restart values $\bf rstrt$. }\vspace{-3mm}
    \centering
    \begin{tabular}{||c||c||c|c||c|c||c|c||c|c||c|c||c|c||}
    \cline{3-14}
     \multicolumn{1}{c}{} &\multicolumn{1}{c||}{}&\multicolumn{6}{c||}{\textbf{Truncated SRE-CG2}} &\multicolumn{2}{c||}{\textbf{ SRE-CG2}}&\multicolumn{4}{c||}{\textbf{Restarted SRE-CG2} } \\
         \cline{3-14}
    \multicolumn{1}{c}{} &\multicolumn{1}{c||}{}&\multicolumn{2}{c||}{$\bf trunc=2$}&\multicolumn{2}{c||}{$\bf trunc=20$}&\multicolumn{2}{c||}{$\bf trunc=50$}&\multicolumn{2}{c||}{$\bf trunc=k_{max}$}&\multicolumn{2}{c||}{$\bf rstrt=20$}&\multicolumn{2}{c||}{$\bf rstrt=50$} \\
        \cline{2-14}
         \multicolumn{1}{c||}{} &$\bf t$&$\bf It$&$\bf RelErr$&$\bf It$&$\bf RelErr$&$\bf It$&$\bf RelErr$&$\bf It$&$\bf RelErr$&$\bf It$&$\bf RelErr$&$\bf It$&$\bf RelErr$\\
    \hline
      \multirow{6}{*}{\rotatebox[origin=c]{90}{\nh} }  & 2 & 241 & 2.6E-7 & 241 & 2.6E-7 & 241 & 2.6E-7 & 241 & 2.6E-7 & - & 6.1E-3 & - & 4.2E-5 \\ \cline{2-14}
        ~ & 4 & 186 & 1.2E-7 & 186 & 1.2E-7 & 186 & 1.2E-7 & 186 & 1.2E-7 & - & 6.1E-3 & - & 6.1E-6 \\ \cline{2-14}
        ~ & 8 & 147 & 5.3E-8 & 147 & 5.3E-8 & 147 & 5.3E-8 & 147 & 5.3E-8 & - & 5.0E-3 & 347 & 4.9E-7 \\ \cline{2-14}
        ~ & 16 & 111 & 3.0E-8 & 111 & 3.0E-8 & 111 & 3.0E-8 & 111 & 3.0E-8 & - & 2.2E-3 & 247 & 2.7E-7 \\ \cline{2-14} 
        ~ & 32 & 82 & 2.3E-8 & 82 & 2.3E-8 & 82 & 2.3E-8 & 82 & 2.3E-8 & - & 7.9E-4 & 115 & 2.2E-7 \\ \cline{2-14}
        ~ & 64 & 59 & 1.0E-8 & 59 & 1.0E-8 & 59 & 1.0E-8 & 59 & 1.0E-8 & 483 & 1.4E-6 & 62 & 2.4E-8 \\ \hline\hline
     \multirow{6}{*}{\rotatebox[origin=c]{90}{\skyt} }  & 2 & 853 & 1.7E-5 & 839 & 1.7E-5 & 822 & 1.7E-5 & 569 & 1.3E-5 & - & 5.6E-1 & - & 2.5E-1 \\ \cline{2-14}
        ~ & 4 & 759 & 1.2E-5 & 736 & 1.2E-5 & 707 & 1.2E-5 & 381 & 1.3E-5 & - & 5.2E-1 & - & 1.8E-1 \\\cline{2-14}
        ~ & 8 & 607 & 2.8E-6 & 566 & 2.7E-6 & 512 & 3.3E-6 & 212 & 1.5E-5 & - & 2.6E-1 & - & 1.1E-2 \\ \cline{2-14}
        ~ & 16 & 424 & 1.4E-6 & 385 & 1.1E-6 & 310 & 1.6E-6 & 117 & 1.1E-5 & - & 1.6E-1 & - & 1.1E-3 \\ \cline{2-14}
        ~ & 32 & 272 & 9.5E-7 & 214 & 1.0E-6 & 145 & 1.0E-6 & 68 & 1.2E-5 & - & 2.9E-2 & 501 & 1.4E-5 \\ \cline{2-14}
        ~ & 64 & 154 & 5.5E-7 & 100 & 7.4E-7 & 42 & 8.7E-6 & 42 & 8.7E-6 & 1041 & 3.0E-5 & 42 & 8.7E-6 \\ \hline\hline
     \multirow{6}{*}{\rotatebox[origin=c]{90}{\ani} }  & 2 & 3961 & 4.1E-5 & 3968 & 3.9E-5 & 3899 & 4.6E-5 & 875 & 7.2E-5 & - & 5.2E-1 & - & 3.2E-1 \\ \cline{2-14}
        ~ & 4 & 3523 & 4.5E-5 & 3526 & 3.9E-5 & 3516 & 3.9E-5 & 673 & 8.3E-5 & - & 5.2E-1 & - & 3.2E-1 \\ \cline{2-14}
        ~ & 8 & 3127 & 4.6E-5 & 2771 & 5.3E-5 & 2677 & 5.9E-5 & 447 & 1.5E-4 & - & 3.8E-1 & - & 2.4E-1 \\ \cline{2-14}
        ~ & 16 & 2413 & 3.0E-5 & 2006 & 3.1E-5 & 1738 & 3.1E-5 & 253 & 1.8E-4 & - & 3.8E-1 & - & 1.9E-1 \\\cline{2-14}
        ~ & 32 & 1636 & 1.8E-5 & 1214 & 1.7E-5 & 547 & 6.3E-5 & 146 & 2.8E-4 & - & 3.6E-1 & 1995 & 2.4E-4 \\ \cline{2-14}
        ~ & 64 & 896 & 6.4E-6 & 457 & 1.5E-5 & 247 & 6.6E-5 & 91 & 1.7E-4 & - & 8.6E-2 & 922 & 1.1E-4 \\ \hline\hline
          \multirow{6}{*}{\rotatebox[origin=c]{90}{\sky} }  & 2 & 5476 & 3.7E-4 & 5420 & 3.7E-4 & 5357 & 3.6E-4 & 1415 & 7.4E-4 & - & 7.6E-1 & - & 6.8E-1 \\ \cline{2-14}
        ~ & 4 & 4532 & 2.6E-5 & 4420 & 2.7E-5 & 4195 & 4.0E-5 & 754 & 2.2E-4 & - & 7.4E-1 & - & 6.5E-1 \\ \cline{2-14}
        ~ & 8 & 2879 & 1.4E-5 & 2750 & 1.4E-5 & 2530 & 2.3E-5 & 399 & 1.7E-4 & - & 7.4E-1 & - & 6.1E-1 \\ \cline{2-14}
        ~ & 16 & 1852 & 9.2E-6 & 1734 & 7.5E-6 & 1563 & 7.9E-6 & 225 & 1.0E-4 & - & 7.2E-1 & - & 5.3E-1 \\ \cline{2-14}
        ~ & 32 & 984 & 5.0E-6 & 848 & 4.9E-6 & 662 & 5.4E-6 & 124 & 7.1E-5 & - & 6.4E-1 & - & 3.4E-1 \\ \cline{2-14}
        ~ & 64 & 483 & 2.4E-6 & 364 & 2.6E-6 & 158 & 1.1E-5 & 73 & 4.4E-5 & - & 4.9E-1 & 2049 & 4.0E-3 \\ \hline\hline
    \end{tabular}\vspace{-2mm}
\end{table}
\noindent The results are shown 
in Tables \ref{tab:SRECG-rest} and \ref{tab:MSDOCG-rest} for SRE-CG2 and MSDO-CG respectively. We do not show the results of the Modified MSDO-CG method as they are almost identical to that of MSDO-CG, with a difference of at most 10 iterations in a few cases. 

The results of Truncated SRE-CG2 (Table \ref{tab:SRECG-rest}) validate the theoretical discussion about $W_k$ only needing to be A-orthonormalized to the previous 2 blocks. For the ``well"-conditioned matrix \nho , all the truncated versions converge in the same number of iterations. For the other 3 ``ill"-conditioned matrices (in increasing order), we observe the effect of numerical loss of A-orthogonality of the truncated versions, where for larger $trunc$ values  less iterations are required till convergence for the same $t$ partitions. Moreover, in all cases and even for $trunc=2$, Truncated SRE-CG2 converges in less iterations than CG. However, this is not the case for restarted SRE-CG2, where for restarting every $j=20$ iterations the method doesn't converge in $k_{max}$ iterations to the relative residual tolerance of $tol=10^{-8}$ (${\ltwonorm{r_k}}/{\ltwonorm{r_0}}<tol$). In the 2 cases where  restarted SRE-CG2(20) converges, it requires more iterations than CG and a larger relative error than the truncated SRE-CG2 versions. 
\noindent Even restarted SRE-CG2(50) doesn't converge in most cases for $t=2,4,8,16$ in $k_{max}$ iterations or in less iterations than CG. This is due to the stagnation of the norm of residual when restarting. 

For SRE-CG2, it is clear that Truncated SRE-CG2 for $trunc=2$ converges in less iterations than the restarted version, and requires much less memory. However, the choice of the $trunc$ value affects the convergence behavior as shown in Figure \ref{fig:trunc} for the matrices \skyto , \skyo , and \anio \,with 
$trunc$ values from $2,20,50,100,200,300,400,k_{max}$. Moreover, if it is possible to  double the memory, then it is more efficient in terms of convergence to double $t$ than $trunc$. Doubling the $trunc$ value doesn't always lead to a proportional reduction in 
\vspace{-4mm}
\begin{figure}[H] 
 \centering
 \hspace{-8mm} 
 \includegraphics[width=0.55\textwidth]{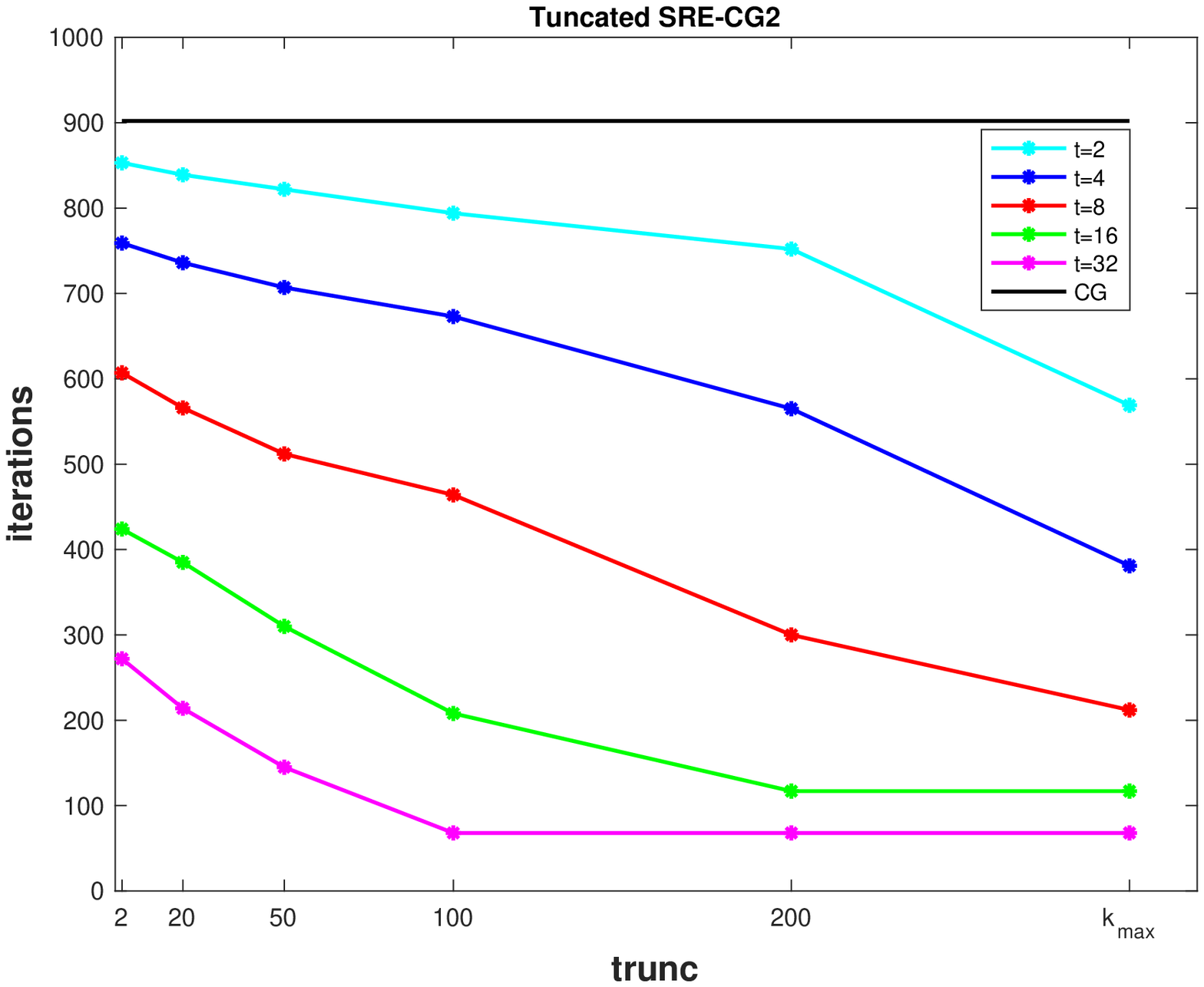} \hspace{-8mm} 
\includegraphics[width=0.55\textwidth]{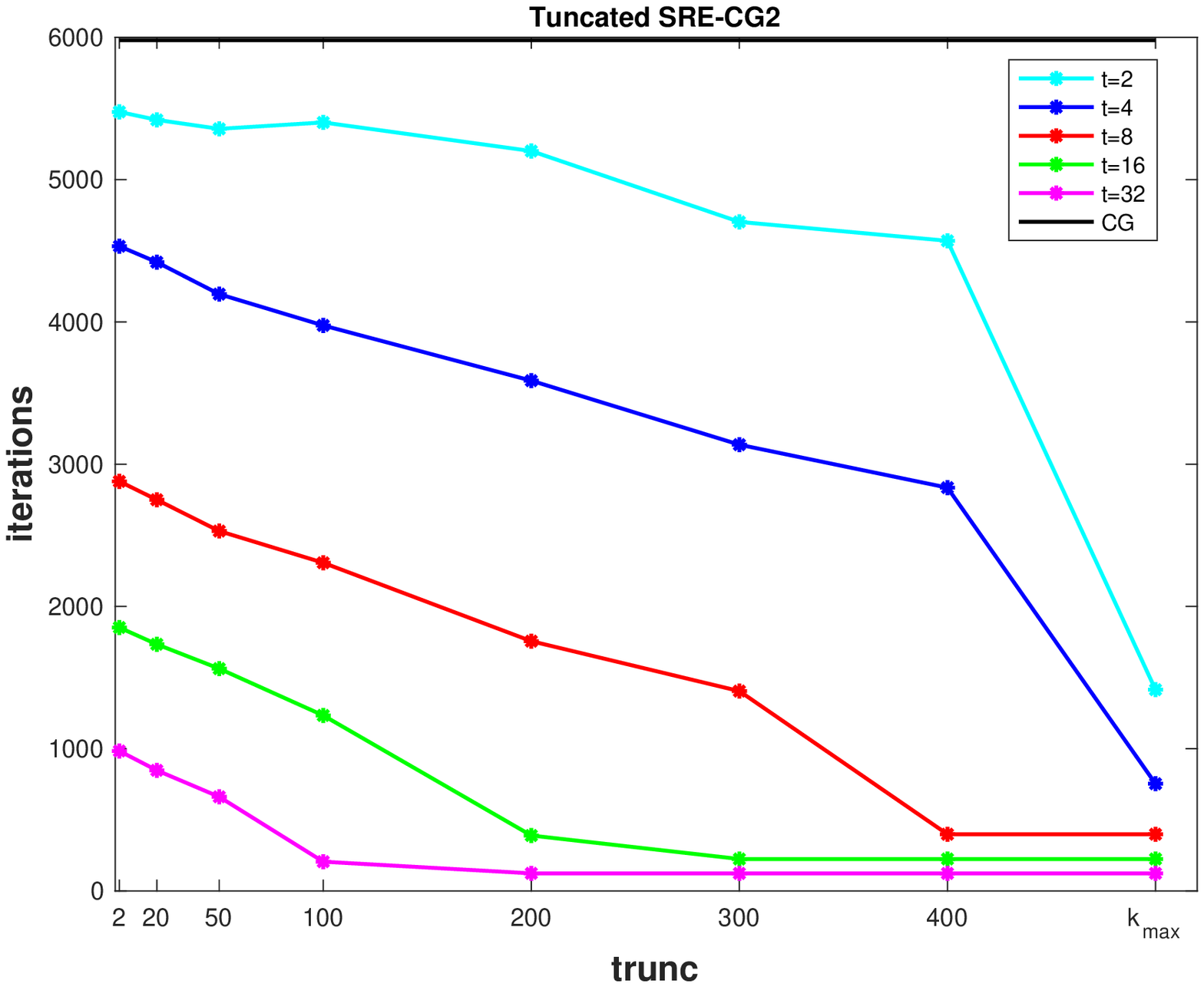}\vspace{-1mm}
\includegraphics[width=0.58\textwidth]{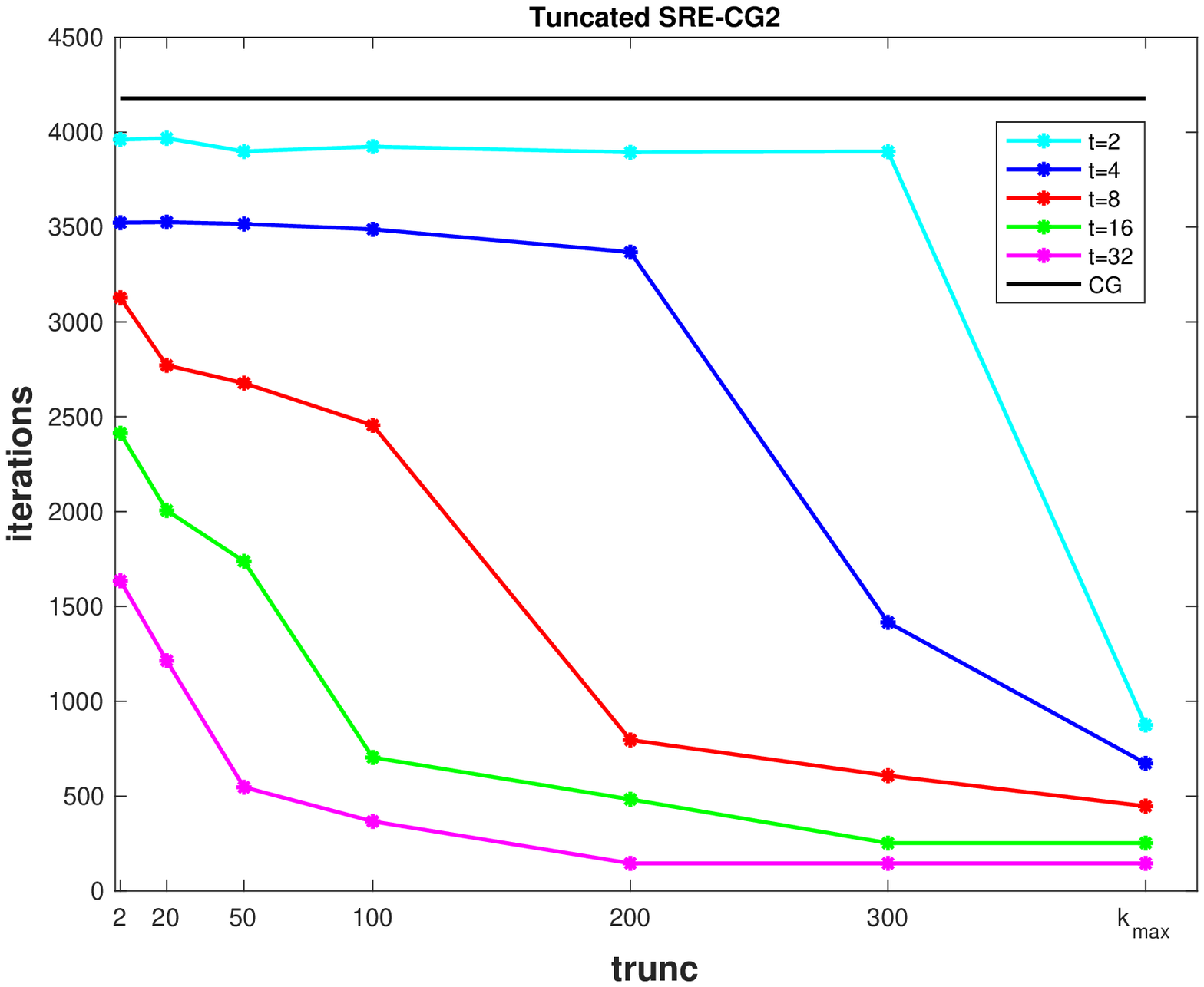}\vspace{-2mm}
   \caption{Convergence of SRE-CG2(trunc) for different $trunc$ and $t$ values for matrices \skyto \,(left),  \skyo \,(right), and \anio \,(bottom)} \label{fig:trunc}\vspace{-10mm}
\end{figure}
\newpage
\noindent iterations. For example, the flat blue curve ($t=4$) of matrix \anio \, shows that doubling $trunc$
from $50$ to $100$ to $200$  barely reduces the iterations from 3516 to 3488 to 3368. We see a steep descent in number of iterations between $trunc=200$ and $300$. This steep descent occurs earlier for larger $t$, between $trunc=100$ and $200$ for $t=8$, between $trunc=50$ and $100$ for $t=16$, and  between $trunc=20$ and $50$ for $t=32$. For $t=32$, SRE-CG(50) converges in 547 iterations , SRE-CG(100) converges in 704 iterations for $t=16$, and SRE-CG(200) converges in 796 iterations for $t=8$, all requiring the storage of 1600 vectors. 

The case of MSDO-CG is different than SRE-CG2 since, as mentioned earlier, there is no guarantee that the method will converge with a truncated A-orthonormalization process. This is clear from the results in Table \ref{tab:MSDOCG-rest}. For the  \nho  \,matrix, all the truncated versions converge, but most requiring more iterations than CG, with less iterations for larger $trunc$ values. However, for the other 3 matrices the truncated MSDO-CG doesn't converge for $trunc=2$. It converges only for \skyo \,and \anio \,matrices with $trunc=20, t=64$, and
\vspace{-3mm}
\begin{table}[H]
\setlength{\tabcolsep}{2.5pt}
\renewcommand{\arraystretch}{1.1}
\caption{\label{tab:MSDOCG-rest} Comparison of the convergence (iteration to convergence $\bf It$, relative error $\bf RelErr$) of truncated and restarted MSDO-CG versions, with respect to number of partitions $\bf t$, truncation values $\bf trunc$, and restart values $\bf rstrt$. }
    \centering
    \begin{tabular}{||c||c||c|c||c|c||c|c||c|c||c|c||c|c||}
    \cline{3-14}
     \multicolumn{1}{c}{} &\multicolumn{1}{c||}{}&\multicolumn{6}{c||}{\textbf{Truncated MSDO-CG}} &\multicolumn{2}{c||}{\textbf{ MSDOCG}}&\multicolumn{4}{c||}{\textbf{Restarted MSDO-CG} } \\
         \cline{3-14}
    \multicolumn{1}{c}{} &\multicolumn{1}{c||}{}&\multicolumn{2}{c||}{$\bf trunc=2$}&\multicolumn{2}{c||}{$\bf trunc=20$}&\multicolumn{2}{c||}{$\bf trunc=50$}&\multicolumn{2}{c||}{$\bf trunc=k_{max}$}&\multicolumn{2}{c||}{$\bf rstrt=20$}&\multicolumn{2}{c||}{$\bf rstrt=50$} \\
        \cline{2-14}
    \multicolumn{1}{c||}{} &$\bf t$&$\bf It$&$\bf RelErr$&$\bf It$&$\bf RelErr$&$\bf It$&$\bf RelErr$&$\bf It$&$\bf RelErr$&$\bf It$&$\bf RelErr$&$\bf It$&$\bf RelErr$\\
    \hline
     \multirow{6}{*}{\rotatebox[origin=c]{90}{\nh} }   &2& 301 & 6.4E-7 & 312 & 6.0E-7 & 314 & 6.4E-7 & 256 & 1.8E-7 & - & 6.7E-3 & - & 2.2E-6 \\ \cline{2-14}
       &4& 338 & 9.8E-7 & 328 & 1.1E-6 & 298 & 4.0E-7 & 206 & 1.3E-7 & - & 6.4E-3 & - & 4.9E-6 \\ \cline{2-14}
       &8& 380 & 9.9E-7 & 324 & 4.3E-7 & 240 & 3.6E-7 & 169 & 9.0E-8 & - & 5.1E-3 & - & 1.1E-6 \\ \cline{2-14}
       &16& 395 & 9.8E-7 & 275 & 7.7E-7 & 186 & 1.6E-7 & 139 & 3.7E-8 & - & 3.3E-3 & 401 & 5.5E-7 \\ \cline{2-14}
        &32&373 & 1.3E-6 & 210 & 1.1E-6 & 131 & 2.5E-7 & 107 & 2.0E-8 & - & 1.5E-3 & 297 & 2.0E-7 \\ \cline{2-14}
      &64&  367 & 1.2E-6 & 155 & 5.8E-7 & 82 & 1.5E-7 & 77 & 1.0E-8 & - & 1.2E-3 & 177 & 1.2E-6 \\ \hline\hline
     \multirow{6}{*}{\rotatebox[origin=c]{90}{\skyt}} &2&   - & 6.5E-3 & - & 5.7E-2 & - & 2.9E-3 & 646 & 1.6E-5 & - & 5.2E-1 & - & 1.5E-1 \\ \cline{2-14}
     &4&   - & 3.1E-2 & - & 8.0E-2 & - & 3.4E-3 & 426 & 2.2E-5 & - & 4.8E-1 & - & 1.8E-1 \\\cline{2-14}
     &8&   - & 8.9E-2 & - & 2.7E-2 & - & 7.2E-4 & 231 & 2.3E-5 & - & 2.1E-1 & - & 8.2E-2 \\ \cline{2-14}
     &16&   - & 8.1E-2 & - & 1.8E-2 & - & 1.2E-4 & 133 & 1.4E-5 & - & 1.3E-1 & - & 2.6E-2 \\ \cline{2-14}
     &32&   - & 7.6E-2 & - & 8.5E-4 & 755 & 3.8E-6 & 79 & 8.7E-6 & - & 9.1E-2 & 916 & 1.7E-5 \\ \cline{2-14}
      &64&  - & 2.6E-2 & 1294 & 6.9E-6 & 50 & 8.5E-6 & 50 & 8.5E-6 & - & 4.1E-3 & 50 & 8.5E-6 \\ \hline\hline
               \multirow{6}{*}{\rotatebox[origin=c]{90}{\ani}} 
          & 2 & - & 5.4E-3 & - & 2.5E-4 & - & 2.9E-4 & 933 & 7.0E-5 & - & 5.2E-1 & - & 3.7E-1 \\ \cline{2-14}
        ~ & 4 & - & 2.9E-3 & - & 2.6E-4 & - & 2.4E-4 & 734 & 9.4E-5 & - & 5.0E-1 & - & 3.7E-1 \\ \cline{2-14}
        ~ & 8 & - & 1.2E-3 & - & 3.4E-4 & - & 2.4E-4 & 471 & 2.1E-4 & - & 3.5E-1 & - & 2.7E-1 \\ \cline{2-14}
        ~ & 16 & - & 4.2E-2 & - & 5.6E-4 & - & 2.1E-4 & 328 & 7.6E-5 & - & 3.5E-1 & - & 2.3E-1 \\ \cline{2-14}
        ~ & 32 & - & 4.0E-2 & - & 3.1E-4 & 4347 & 3.3E-5 & 202 & 8.6E-5 & - & 3.0E-1 & - & 1.1E-1 \\ \cline{2-14}
        ~ & 64 & - & 6.3E-2 & 4716 & 3.1E-5 & 1838 & 3.9E-5 & 116 & 1.3E-4 & - &1.8E-1 & - & 6.8E-3 \\ \hline\hline

        \multirow{6}{*}{\rotatebox[origin=c]{90}{\sky}} &2&   - & 3.2E-2 & - & 3.9E-2 & - & 1.0E-1 & 1514 & 8.1E-4 & - & 7.5E-1 & - & 6.7E-1 \\\cline{2-14}
    &4&    - & 4.2E-2 & - & 2.5E-1 & - & 3.2E-2 & 845 & 3.3E-4 & - & 7.1E-1 & - & 6.2E-1 \\ \cline{2-14}
     &8&   - & 7.6E-2 & - & 3.3E-1 & - & 1.1E-2 & 513 & 2.3E-4 & - & 6.8E-1 & - & 5.8E-1 \\ \cline{2-14}
      &16&  - & 2.1E-1 & - & 3.4E-1 & - & 5.9E-3 & 291 & 2.2E-4 & - & 6.7E-1 & - & 5.3E-1 \\ \cline{2-14}
    &32&    - & 1.3E-1 & - & 2.6E-2 & - & 1.7E-3 & 162 & 1.1E-4 & - & 6.4E-1 & - & 3.9E-1 \\ \cline{2-14}
     &64&   - & 3.1E-1 & - & 3.3E-3 & 2653 & 4.6E-5 & 120 & 2.4E-5 & - & 5.3E-1 & - & 5.0E-2 \\ \hline\hline
    \end{tabular}
\end{table}
\noindent $trunc=50, t=32,64$ , and  for \skyto \, matrix with $trunc=50, t=64$. Note that when restarting every $j=20$ iterations, the method doesn't converge in $k_{max}$ iterations for all the matrices. As for restarting every $j=50$ iterations, the method converges in less iterations than CG for \nho \,and \skyto \, matrices and $t=64$.

It is clear that truncating the A-orthonormalization process is better than restarting in the cases where the methods converge, as the truncated versions converge in less iterations. 
Moreover, the effect of the condition number of the matrices on the convergence of the different methods is clear, as the larger the condition number is, more iterations are required for convergence.  

We also tested the restarted versions where the restart is done once the relative difference of residual norm is less than some $restartTol$ for $restartTol=10^{-3}, 10^{-5}, 10^{-7}$. The results are shown in tables \ref{tab:SRECG-rest2} and \ref{tab:MSDOCG-rest2}, where in addition to the number of iterations ($\bf It$) to convergence and the corresponding relative error ($\bf RelEr$), the number of restarts ($\bf \#r$) and the largest number of previous block vectors ($\bf mem$) that are stored throughout the different restarted cycles are shown.
 \vspace{-5mm}
\begin{table}[H]
\setlength{\tabcolsep}{2.5pt}
\renewcommand{\arraystretch}{1.1}
\caption{\label{tab:SRECG-rest2} Convergence (iteration to convergence $\bf It$, relative error $\bf RelEr$) of  restarted SRE-CG2 versions, with respect to number of partitions $\bf t$, restart tolerances  $\bf restartTol$. Moreover, the number of performed restarts $\bf \#r$ and the largest restart cycle memory requirement $\bf mem$ are shown. }\vspace{-1mm}
    \centering
    \begin{tabular}{||c||c||c|c|c||c|c|c|c||c|c|c|c||c|c|c|c||}
    \cline{3-17}
     \multicolumn{1}{c}{} &\multicolumn{1}{c||}{}&\multicolumn{3}{c||}{\multirow{2}{*}{\textbf{ SRE-CG2}}}&\multicolumn{12}{c||}{\textbf{Restarted SRE-CG2} } \\
         \cline{6-17}
    \multicolumn{1}{c}{} &\multicolumn{1}{c||}{}&\multicolumn{1}{c}{} &\multicolumn{1}{c}{} & &\multicolumn{4}{c||}{$\bf restartTol=10^{-3}$}&\multicolumn{4}{c||}{$\bf restartTol=10^{-5}$}&\multicolumn{4}{c||}{$\bf restartTol=10^{-7}$} \\
        \cline{2-17}
    \multicolumn{1}{c||}{} &$\bf t$&$\bf It$&$\bf mem$&$\bf RelErr$&$\bf It$&$\bf\#r$&$\bf mem$&$\bf RelEr$&$\bf It$&$\bf\#r$&$\bf mem$&$\bf RelEr$&$\bf It$&$\bf\#r$&$\bf mem$&$\bf RelEr$\\
    \hline
\multirow{6}{*}{\rotatebox[origin=c]{90}{\nh}}  
& 2 & 241 & 482 & 2.6E-7 & 387 & 10 & 174 & 9.1E-7 & 328 & 3 & 312 & 7.0E-7 & 247 & 1 & 286 & 3.3E-7 \\ \cline{2-17}
        ~ & 4 & 186 & 744 & 1.2E-7 & 470 & 14 & 244 & 1.6E-6 & 262 & 2 & 580 & 2.0E-7 & 195 & 1 & 616 & 1.9E-7 \\ \cline{2-17}
        ~ & 8 & 147 & 1176 & 5.3E-8 & 217 & 6 & 664 & 1.3E-6 & 178 & 2 & 992 & 8.4E-8 & 182 & 1 & 1152 & 5.3E-8 \\ \cline{2-17}
        ~ & 16 & 111 & 1776 & 3.0E-8 & 187 & 4 & 1184 & 1.7E-7 & 141 & 1 & 1184 & 8.5E-8 & 113 & 1 & 1568 & 7.1E-8 \\ \cline{2-17}
        ~ & 32 & 82 & 2624 & 2.3E-8 & 116 & 5 & 1792 & 1.7E-6 & 110 & 2 & 2240 & 3.0E-8 & 84 & 1 & 2368 & 3.4E-8 \\ \cline{2-17}
        ~ & 64 & 59 & 3776 & 1.0E-8 & 88 & 4 & 2752 & 2.0E-7 & 69 & 1 & 2880 & 2.9E-7 & 60 & 1 & 3520 & 1.3E-8 \\ \hline\hline
                \multirow{6}{*}{\rotatebox[origin=c]{90}{\skyt}}   & 2 & 569 & 1138 & 1.3E-5 & 1337 & 10 & 566 & 2.6E-5 & 793 & 2 & 838 & 8.7E-6 & 640 & 1 & 1064 & 1.7E-5 \\ \cline{2-17}
        ~ & 4 & 381 & 1524 & 1.3E-5 & 1238 & 11 & 816 & 4.2E-5 & 538 & 2 & 1388 & 9.1E-6 & 584 & 1 & 1232 & 9.4E-6 \\ \cline{2-17}
        ~ & 8 & 212 & 1696 & 1.5E-5 & 502 & 7 & 1584 & 4.2E-6 & 335 & 2 & 1664 & 6.8E-6 & 356 & 1 & 1624 & 8.4E-6 \\ \cline{2-17}
        ~ & 16 & 117 & 1872 & 1.1E-5 & 275 & 3 & 1712 & 1.5E-5 & 199 & 1 & 1824 & 1.2E-5 & 202 & 1 & 1792 & 4.0E-6 \\ \cline{2-17}
        ~ & 32 & 68 & 2176 & 1.2E-5 & 183 & 4 & 1888 & 2.0E-5 & 119 & 1 & 2016 & 1.4E-5 & 114 & 1 & 2080 & 2.5E-6 \\ \cline{2-17}
        ~ & 64 & 42 & 2688 & 8.7E-6 & 80 & 2 & 2240 & 1.5E-5 & 69 & 1 & 2240 & 1.1E-5 & 67 & 1 & 2560 & 1.8E-6 \\ \hline\hline
        \multirow{6}{*}{\rotatebox[origin=c]{90}{\ani}}   & 2 & 875 & 1750 & 7.2E-5 & 4412 & 19 & 976 & 1.0E-4 & 1677 & 3 & 1290 & 1.1E-4 & 1238 & 1 & 1646 & 7.3E-5 \\ \cline{2-17}
        ~ & 4 & 673 & 2692 & 8.3E-5 & 2080 & 11 & 1340 & 1.9E-4 & 1011 & 2 & 1824 & 1.2E-4 & 876 & 1 & 2124 & 7.7E-5 \\ \cline{2-17}
        ~ & 8 & 447 & 3576 & 1.5E-4 & 1549 & 12 & 1808 & 1.3E-4 & 924 & 3 & 2480 & 1.2E-4 & 579 & 1 & 3480 & 1.4E-4 \\ \cline{2-17}
        ~ & 16 & 253 & 4048 & 1.8E-4 & 1181 & 15 & 2528 & 1.4E-4 & 517 & 2 & 3712 & 1.3E-4 & 362 & 1 & 3904 & 2.5E-4 \\ \cline{2-17}
        ~ & 32 & 146 & 4672 & 2.8E-4 & 559 & 8 & 4000 & 1.1E-4 & 261 & 2 & 4416 & 1.5E-4 & 252 & 1 & 4064 & 2.4E-4 \\ \cline{2-17}
        ~ & 64 & 91 & 5824 & 1.7E-4 & 256 & 6 & 4736 & 3.0E-4 & 124 & 1 & 5696 & 2.4E-4 & 143 & 1 & 5440 & 2.1E-4 \\ \hline\hline
\multirow{6}{*}{\rotatebox[origin=c]{90}{\sky}}   & 2 & 1415 & 2830 & 7.4E-4 & 6000 & 51 & 1460 & 4.9E-2 & 2903 & 3 & 2814 & 1.1E-4 & 1476 & 1 & 2830 & 1.1E-3 \\ \cline{2-17}
        ~ & 4 & 754 & 3016 & 2.2E-4 & 6000 & 41 & 2824 & 1.5E-3 & 1309 & 2 & 2984 & 3.7E-4 & 904 & 1 & 3024 & 3.4E-4 \\ \cline{2-17}
        ~ & 8 & 399 & 3192 & 1.7E-4 & 1653 & 13 & 3096 & 8.9E-5 & 799 & 2 & 3176 & 1.2E-4 & 498 & 1 & 3224 & 1.7E-4 \\ \cline{2-17}
        ~ & 16 & 225 & 3600 & 1.0E-4 & 1466 & 10 & 3264 & 1.6E-4 & 489 & 2 & 3472 & 1.3E-4 & 404 & 1 & 3440 & 1.1E-4 \\ \cline{2-17}
        ~ & 32 & 124 & 3968 & 7.1E-5 & 429 & 6 & 4000 & 6.9E-5 & 268 & 2 & 3872 & 3.3E-5 & 210 & 1 & 3584 & 4.9E-4 \\ \cline{2-17}
        ~ & 64 & 73 & 4672 & 4.4E-5 & 211 & 4 & 4672 & 3.0E-5 & 125 & 1 & 4800 & 2.6E-5 & 118 & 1 & 4224 & 1.8E-4 \\ \hline\hline
    \end{tabular}
\end{table}
As shown in Tables \ref{tab:SRECG-rest2} and \ref{tab:MSDOCG-rest2}, using $restartTol=10^{-3}$ leads to early and several restarts that might lead to stagnation in a few cases. Moreover, the total number of iterations is increased with respect to the non-restarted method, and in some cases for $t=2,4$ they require more iterations than CG. 
However, the memory reduction with respect to SRE-CG2 and MSDO-CG is between $20\% \mbox{ and } 60\%$. On the other hand, using $restartTol=10^{-7}$ leads to a very late and mostly one restart, thus defeating the purpose of restart (reducing memory requirements). This is clear as the memory requirements for $t=16,32,64$ of the restarted versions is almost identical to the initial versions.
The moderate choice is setting $restartTol=10^{-5}$ which leads to a few restarts without stagnation with a memory reduction of  $5\% \mbox{ to } 50\%$.
 For some cases, even though for $restartTol=10^{-5}$ the  methods restart earlier and more than the case of $restartTol=10^{-7}$, yet the needed memory is more. For example, \skyo \,with $t=16,32,64$,  \skyto \,with $t=4,8,16$, and \anio  \,with $t=32$.

Comparing the restarted versions with a fixed restarted cycle (Tables \ref{tab:SRECG-rest} and \ref{tab:MSDOCG-rest}) to the versions with the restart tolerance (Tables \ref{tab:SRECG-rest2} and \ref{tab:MSDOCG-rest2}) it is clear that restarting once the  relative difference of the residuals is less then some tolerance is more adaptive than fixing the\vspace{-3mm}

\begin{table}[H]
\setlength{\tabcolsep}{2.5pt}
\renewcommand{\arraystretch}{1.1}
\caption{\label{tab:MSDOCG-rest2} Comparison of the convergence (iteration to convergence $\bf It$, relative error $\bf RelEr$) of  restarted MSDO-CG versions, with respect to number of partitions $\bf t$, restart tolerances  $\bf restartTol$. The number of performed restarts $\bf \#r$ and largest restart cycle memory requirement $\bf mem$ are shown. \vspace{-3mm}}
  \hspace{-5mm}  \begin{tabular}{||c||c||c|c|c||c|c|c|c||c|c|c|c||c|c|c|c||}
    \cline{3-17}
     \multicolumn{1}{c}{} &\multicolumn{1}{c||}{}&\multicolumn{3}{c||}{\multirow{2}{*}{\textbf{ MSDOCG}}}&\multicolumn{12}{c||}{\textbf{Restarted MSDO-CG} } \\
         \cline{6-17}
    \multicolumn{1}{c}{} & \multicolumn{1}{c||}{} &\multicolumn{1}{c}{}&\multicolumn{1}{c}{} & &\multicolumn{4}{c||}{$\bf restartTol=10^{-3}$}&\multicolumn{4}{c||}{$\bf restartTol=10^{-5}$}&\multicolumn{4}{c||}{$\bf restartTol=10^{-7}$} \\
        \cline{2-17}
    \multicolumn{1}{c||}{} &$\bf t$&$\bf It$&$\bf mem$&$\bf RelErr$&$\bf It$&$\bf\#r$&$\bf mem$&$\bf RelEr$&$\bf It$&$\bf\#r$&$\bf mem$&$\bf RelEr$&$\bf It$&$\bf\#r$&$\bf mem$&$\bf RelEr$\\
    \hline
     \multirow{6}{*}{\rotatebox[origin=c]{90}{\nh} } 
 & 2 & 256 & 512 & 1.8E-7 & 458 & 14 & 162 & 1.7E-6 & 301 & 3 & 278 & 1.3E-6 & 304 & 2 & 284 & 6.5E-7 \\ \cline{2-17}
        ~ & 4 & 206 & 824 & 1.3E-7 & 352 & 10 & 372 & 2.7E-6 & 266 & 2 & 592 & 5.9E-7 & 251 & 1 & 648 & 8.8E-7 \\ \cline{2-17}
        ~ & 8 & 169 & 1352 & 9.0E-8 & 302 & 6 & 720 & 6.6E-7 & 220 & 2 & 1048 & 3.8E-7 & 200 & 1 & 1152 & 5.4E-7 \\ \cline{2-17}
        ~ & 16 & 139 & 2224 & 3.7E-8 & 252 & 8 & 1232 & 1.2E-6 & 180 & 2 & 1696 & 2.7E-7 & 173 & 1 & 1952 & 5.2E-7 \\ \cline{2-17}
        ~ & 32 & 107 & 3424 & 2.0E-8 & 222 & 7 & 2048 & 5.0E-7 & 150 & 2 & 2784 & 1.3E-7 & 134 & 1 & 3136 & 3.7E-7 \\ \cline{2-17}
        ~ & 64 & 77 & 4928 & 1.0E-8 & 163 & 6 & 3328 & 1.9E-7 & 94 & 2 & 4096 & 7.0E-8 & 79 & 1 & 4672 & 1.4E-7 \\ \hline\hline
        \multirow{6}{*}{\rotatebox[origin=c]{90}{\skyt} }  & 2 & 646 & 1292 & 1.6E-5 & - & 11 & 668 & 9.3E-4 & 943 & 2 & 1000 & 2.4E-5 & 900 & 1 & 958 & 4.5E-5 \\ \cline{2-17}
        ~ & 4 & 426 & 1704 & 2.2E-5 & - & 14 & 1296 & 6.7E-5 & 730 & 2 & 1636 & 1.3E-5 & 739 & 1 & 1520 & 2.7E-5 \\ \cline{2-17}
        ~ & 8 & 231 & 1848 & 2.3E-5 & 953 & 9 & 1480 & 1.1E-5 & 346 & 2 & 1824 & 8.3E-6 & 398 & 1 & 1712 & 1.9E-5 \\ \cline{2-17}
        ~ & 16 & 133 & 2128 & 1.4E-5 & 339 & 4 & 2000 & 5.2E-6 & 239 & 2 & 2096 & 9.4E-6 & 227 & 1 & 2096 & 7.3E-6 \\ \cline{2-17}
        ~ & 32 & 79 & 2528 & 8.7E-6 & 153 & 2 & 2240 & 6.5E-6 & 140 & 1 & 2272 & 4.3E-6 & 131 & 1 & 2464 & 5.6E-6 \\ \cline{2-17}
        ~ & 64 & 50 & 3200 & 8.5E-6 & 95 & 2 & 2816 & 7.4E-6 & 85 & 1 & 2752 & 8.3E-6 & 82 & 1 & 3072 & 5.7E-6 \\ \hline\hline
        \multirow{6}{*}{\rotatebox[origin=c]{90}{\ani} }  & 2 & 933 & 1866 & 7.0E-5 & - & 48 & 954 & 1.7E-4 & 1810 & 3 & 1278 & 7.7E-5 & 1571 & 2 & 1572 & 7.9E-5 \\ \cline{2-17}
        ~ & 4 & 734 & 2936 & 9.4E-5 & 4653 & 44 & 1080 & 1.5E-4 & 1493 & 4 & 1708 & 9.6E-5 & 944 & 1 & 2780 & 8.5E-5 \\ \cline{2-17}
        ~ & 8 & 471 & 3768 & 2.1E-4 & 3793 & 36 & 2560 & 8.7E-5 & 1092 & 3 & 3200 & 1.3E-4 & 709 & 1 & 3680 & 1.7E-4 \\ \cline{2-17}
        ~ & 16 & 328 & 5248 & 7.6E-5 & 1417 & 17 & 3648 & 6.3E-5 & 707 & 3 & 4656 & 7.0E-5 & 477 & 1 & 5008 & 1.2E-4 \\ \cline{2-17}
        ~ & 32 & 202 & 6464 & 8.6E-5 & 1059 & 13 & 4672 & 4.6E-5 & 281 & 1 & 6304 & 1.1E-4 & 302 & 1 & 6112 & 1.2E-4 \\ \cline{2-17}
        ~ & 64 & 116 & 7424 & 1.3E-4 & 399 & 5 & 6336 & 3.9E-5 & 227 & 2 & 6720 & 6.9E-5 & 202 & 1 & 7104 & 9.0E-5 \\ \hline\hline
\multirow{6}{*}{\rotatebox[origin=c]{90}{\sky} }  & 2 & 1514 & 3028 & 8.1E-4 & - & 59 & 2360 & 7.1E-4 & 2200 & 2 & 2962 & 9.1E-4 & 1839 & 1 & 3022 & 7.2E-4 \\ \cline{2-17}
        ~ & 4 & 845 & 3380 & 3.3E-4 & 4835 & 41 & 2796 & 5.0E-4 & 1374 & 2 & 3324 & 3.7E-4 & 973 & 1 & 3368 & 3.3E-4 \\ \cline{2-17}
        ~ & 8 & 513 & 4104 & 2.3E-4 & 1940 & 14 & 3472 & 3.0E-4 & 917 & 2 & 4016 & 1.7E-4 & 692 & 1 & 4088 & 2.6E-4 \\ \cline{2-17}
        ~ & 16 & 291 & 4656 & 2.2E-4 & 1201 & 14 & 4384 & 2.1E-4 & 569 & 2 & 4480 & 4.4E-4 & 539 & 1 & 4432 & 1.5E-4 \\ \cline{2-17}
        ~ & 32 & 162 & 5184 & 1.1E-4 & 443 & 7 & 4896 & 1.1E-4 & 327 & 2 & 5024 & 5.5E-5 & 289 & 1 & 4800 & 8.6E-4 \\ \cline{2-17}
        ~ & 64 & 120 & 7680 & 2.4E-5 & 316 & 5 & 7232 & 1.6E-5 & 185 & 1 & 7488 & 2.8E-5 & 217 & 1 & 7424 & 2.6E-5 \\ \hline\hline
    \end{tabular}
\end{table}
 \noindent restart iteration, which leads to convergence within the $k_{max}$ iterations. But this comes at the expense of requiring more memory storage than the fixed restart iteration. 

To summarize, for SRE-CG2 theoretically it is possible to truncate the A-orthonormalization process for some preset $trunc$ value depending on the available memory. However, if this  $trunc$ value is too small relative to the required iterations to convergence by SRE-CG2, then the truncated version will require much more iterations,  but still less than CG as shown in Figure \ref{fig:trunc}. Moreover, it is preferable to double $t$ rather than $trunc$. For MSDO-CG, truncation doesn't necessarily lead to convergence, as it is not guaranteed theoretically. As for restarting after a fixed number of iterations, it leads to stagnation and the method may not converge in $k_{max}$ iterations, even though it requires the same storage as the truncated versions. Restarting  after the relative difference of the residuals is less than the restart tolerance, is more flexible as it leads to  convergence within the $k_{max}$ iterations and in less iterations than CG, but more iterations than the corresponding method. However, its memory requirement is not known beforehand.

\subsection{Flexibly  Enlarged CG Methods}\label{sec:flextest} We test the convergence of flexibly SRE-CG2 and MSDO-CG for different $t$ values and $switchTol=10^{-3}, 10^{-5},10^{-7}$  (Tables \ref{tab:SRECG-flex}, \ref{tab:MSDOCG-flex})   .\vspace{-4mm}
\begin{table}[H]
\setlength{\tabcolsep}{2.5pt}
\renewcommand{\arraystretch}{1.1}
\caption{\label{tab:SRECG-flex} Convergence (iteration to convergence $\bf It$, switch iteration $\bf sIt$, relative error $\bf RelEr$) of  flexibly SRE-CG2 with respect to number of partitions $\bf t$, switch tolerances  $\bf switchTol$.\vspace{-3mm} }
    \centering
    \begin{tabular}{||c||c||c|c||c|c|c||c|c|c||c|c|c||}
    \cline{3-13}
     \multicolumn{1}{c}{} &\multicolumn{1}{c||}{}&\multicolumn{2}{c||}{\multirow{2}{*}{\textbf{ SRE-CG2}}}&\multicolumn{9}{c||}{\textbf{Flexibly  SRE-CG2} } \\
         \cline{5-13}
    \multicolumn{1}{c}{} &\multicolumn{1}{c||}{}&\multicolumn{1}{c}{} & &\multicolumn{3}{c||}{$\bf switchTol=10^{-3}$}&\multicolumn{3}{c||}{$\bf switchTol=10^{-5}$}&\multicolumn{3}{c||}{$\bf switchTol=10^{-7}$} \\
        \cline{2-13}
    \multicolumn{1}{c||}{} &$\bf t$&$\bf It$&$\bf RelErr$&$\bf It$&$\bf sIt$&$\bf RelEr$&$\bf It$&$\bf sIt$&$\bf RelEr$&$\bf It$&$\bf sIt$&$\bf RelEr$\\
    \hline
     \multirow{6}{*}{\rotatebox[origin=c]{90}{\nh} }   & 2 & 241 & 2.6E-7 & 282 & 14 & 6.6E-7 & 293 & 42 & 1.3E-7 & 253 & 144 & 2.1E-7 \\ \cline{2-13}
        ~ & 4 & 186 & 1.2E-7 & 251 & 14 & 1.8E-7 & 207 & 107 & 1.2E-7 & 190 & 155 & 1.5E-7 \\ \cline{2-13}
        ~ & 8 & 147 & 5.3E-8 & 192 & 14 & 1.0E-7 & 177 & 35 & 1.0E-7 & 175 & 38 & 9.2E-8 \\ \cline{2-13}
        ~ & 16 & 111 & 3.0E-8 & 152 & 14 & 5.1E-8 & 114 & 75 & 4.6E-8 & 111 & 99 & 3.2E-8 \\ \cline{2-13}
        ~ & 32 & 82 & 2.3E-8 & 108 & 14 & 3.3E-8 & 96 & 32 & 2.8E-8 & 83 & 75 & 1.6E-8 \\ \cline{2-13}
        ~ & 64 & 59 & 1.0E-8 & 77 & 14 & 1.7E-8 & 60 & 46 & 1.6E-8 & 59 & 56 & 1.2E-8 \\ \hline \hline
         \multirow{6}{*}{\rotatebox[origin=c]{90}{\skyt} }  & 2 & 569 & 1.3E-5 & 710 & 19 & 3.0E-5 & 671 & 85 & 3.5E-5 & 679 & 108 & 2.1E-5 \\ \cline{2-13}
        ~ & 4 & 381 & 1.3E-5 & 577 & 18 & 1.3E-5 & 535 & 84 & 2.7E-5 & 399 & 309 & 2.3E-5 \\ \cline{2-13}
        ~ & 8 & 212 & 1.5E-5 & 368 & 15 & 1.2E-5 & 312 & 78 & 1.9E-5 & 253 & 153 & 1.4E-5 \\ \cline{2-13}
        ~ & 16 & 117 & 1.1E-5 & 200 & 13 & 2.1E-5 & 134 & 85 & 1.4E-5 & 117 & 113 & 1.2E-5 \\ \cline{2-13}
        ~ & 32 & 68 & 1.2E-5 & 107 & 14 & 1.7E-5 & 69 & 56 & 1.4E-5 & 68 & 66 & 1.3E-5 \\ \cline{2-13}
        ~ & 64 & 42 & 8.7E-6 & 60 & 12 & 7.4E-6 & 43 & 34 & 6.8E-6 & 42 & 41 & 9.0E-6 \\  \hline\hline
         \multirow{6}{*}{\rotatebox[origin=c]{90}{\ani} }  & 2 & 875 & 7.2E-5 & 1233 & 19 & 7.4E-5 & 1206 & 53 & 7.1E-5 & 884 & 415 & 7.9E-5 \\ \cline{2-13}
        ~ & 4 & 673 & 8.3E-5 & 861 & 19 & 7.0E-5 & 754 & 142 & 6.9E-5 & 689 & 345 & 8.3E-5 \\ \cline{2-13}
        ~ & 8 & 447 & 1.5E-4 & 683 & 11 & 6.2E-5 & 597 & 115 & 6.5E-5 & 584 & 144 & 9.8E-5 \\ \cline{2-13}
        ~ & 16 & 253 & 1.8E-4 & 440 & 11 & 1.5E-4 & 392 & 63 & 2.0E-4 & 353 & 118 & 2.2E-4 \\ \cline{2-13}
        ~ & 32 & 146 & 2.8E-4 & 243 & 11 & 1.9E-4 & 216 & 39 & 1.9E-4 & 150 & 125 & 2.2E-4 \\ \cline{2-13}
        ~ & 64 & 91 & 1.7E-4 & 138 & 11 & 2.2E-4 & 116 & 35 & 2.4E-4 & 100 & 58 & 2.0E-4 \\ \hline\hline
           \multirow{6}{*}{\rotatebox[origin=c]{90}{\sky} }  & 2 & 1415 & 7.4E-4 & 2396 & 15 & 7.4E-4 & 2416 & 22 & 6.7E-4 & 2415 & 61 & 5.9E-4 \\ \cline{2-13}
        ~ & 4 & 754 & 2.2E-4 & 1411 & 15 & 8.1E-4 & 1315 & 106 & 1.2E-3 & 1290 & 148 & 3.2E-4 \\ \cline{2-13}
        ~ & 8 & 399 & 1.7E-4 & 733 & 15 & 2.9E-4 & 710 & 44 & 4.6E-4 & 664 & 95 & 2.4E-4 \\ \cline{2-13}
        ~ & 16 & 225 & 1.0E-4 & 386 & 11 & 1.6E-4 & 326 & 79 & 1.6E-4 & 227 & 189 & 1.5E-4 \\ \cline{2-13}
        ~ & 32 & 124 & 7.1E-5 & 208 & 13 & 1.0E-4 & 169 & 55 & 1.2E-4 & 125 & 113 & 8.9E-5 \\ \cline{2-13}
        ~ & 64 & 73 & 4.4E-5 & 115 & 10 & 6.5E-5 & 84 & 50 & 5.0E-5 & 74 & 67 & 4.0E-5 \\ \hline \hline
    \end{tabular}\vspace{-5mm}
\end{table}
\newpage
In Tables \ref{tab:SRECG-flex} and \ref{tab:MSDOCG-flex} we compare the convergence behavior (number of iterations to convergence $\bf It$ and the relative error $\bf RelEr$) of Flexibly  SRE-CG2 with that of SRE-CG2. We also show the switch iteration, $\bf switchIt$, at which the $t$ value is halved.

Flexibly  SRE-CG2 and Flexibly  MSDO-CG have the same convergence behavior as SRE-CG2 and MSDO-CG, respectively, where for larger $t$ the methods converge in less iterations with similar order of relative errors, that increase with the matrices' condition number. Flexibly  SRE-CG2 with a given $t$ value is a mixture of SRE-CG2 with $t$ and $t/2$ vectors per iteration (similarly for MSDO-CG). Thus, it is expected that the convergence of Flexibly  SRE-CG2 to be within the range of SRE-CG2 with $t$ and $t/2$ vectors per iteration. To which end of the spectrum the Flexibly  SRE-CG2's convergence lies, depends on the values of $t$ and $switchTol$ and the matrix's condition number.

\noindent Setting $switchTol$ to $10^{-3}$ leads to an early switch in the first 10 to 20 iterations, thus converging in number of iterations closer to SRE-CG2 with $t/2$ (sometimes slightly more). In this case, Flexibly  SRE-CG2 will require more iterations and less memory than SRE-CG2 with $t$  as shown in Table \ref{tab:SRECG-mem}, however, more memory than SRE-CG2 with $t/2$. Thus it is more efficient to use SRE-CG2 with $t/2$ rather than Flexibly  SRE-CG2 with $t$ and $switchTol=10^{-3}$.\vspace{-7mm}

\begin{table}[H]
\setlength{\tabcolsep}{2.5pt}
\renewcommand{\arraystretch}{1.1}
\caption{\label{tab:MSDOCG-flex} Convergence (iteration to convergence $\bf It$, switch iteration $\bf sIt$, relative error $\bf RelEr$) of  flexibly MSDO-CG with respect to number of partitions $\bf t$, restart tolerances  $\bf restartTol$. }
    \centering
    \begin{tabular}{||c||c||c|c||c|c|c||c|c|c||c|c|c||}
    \cline{3-13}
     \multicolumn{1}{c}{} &\multicolumn{1}{c||}{}&\multicolumn{2}{c||}{\multirow{2}{*}{\textbf{ MSDOCG}}}&\multicolumn{9}{c||}{\textbf{Flexibly  MSDO-CG} } \\
         \cline{5-13}
    \multicolumn{1}{c}{} &\multicolumn{1}{c||}{}&\multicolumn{1}{c}{} & &\multicolumn{3}{c||}{$\bf switchTol=10^{-3}$}&\multicolumn{3}{c||}{$\bf switchTol=10^{-5}$}&\multicolumn{3}{c||}{$\bf switchTol=10^{-7}$} \\
        \cline{2-13}
    \multicolumn{1}{c||}{} &$\bf t$&$\bf It$&$\bf RelErr$&$\bf It$&$\bf sIt$&$\bf RelEr$&$\bf It$&$\bf sIt$&$\bf RelEr$&$\bf It$&$\bf sIt$&$\bf RelEr$\\
    \hline
     \multirow{6}{*}{\rotatebox[origin=c]{90}{\nh} } & 2 & 256 & 1.8E-7 & 267 & 14 & 4.1E-7 & 257 & 58 & 4.5E-7 & 257 & 143 & 2.9E-7 \\ \cline{2-13}
        ~ & 4 & 206 & 1.3E-7 & 254 & 14 & 1.7E-7 & 229 & 58 & 2.0E-7 & 208 & 163 & 1.3E-7 \\ \cline{2-13}
        ~ & 8 & 169 & 9.0E-8 & 204 & 14 & 1.2E-7 & 183 & 69 & 1.3E-7 & 170 & 145 & 9.8E-8 \\ \cline{2-13}
        ~ & 16 & 139 & 3.7E-8 & 165 & 14 & 9.9E-8 & 152 & 56 & 9.9E-8 & 139 & 123 & 4.3E-8 \\ \cline{2-13}
        ~ & 32 & 107 & 2.0E-8 & 135 & 14 & 4.1E-8 & 117 & 57 & 1.9E-8 & 107 & 99 & 2.4E-8 \\ \cline{2-13}
        ~ & 64 & 77 & 1.0E-8 & 100 & 16 & 2.7E-8 & 93 & 27 & 2.7E-8 & 77 & 74 & 1.1E-8 \\ \hline\hline
        \multirow{6}{*}{\rotatebox[origin=c]{90}{\skyt} } & 2 & 646 & 1.6E-5 & 751 & 18 & 2.2E-5 & 744 & 35 & 2.9E-5 & 660 & 479 & 3.1E-5 \\ \cline{2-13}
        ~ & 4 & 426 & 2.2E-5 & 650 & 23 & 1.4E-5 & 627 & 59 & 1.7E-5 & 441 & 380 & 2.3E-5 \\ \cline{2-13}
        ~ & 8 & 231 & 2.3E-5 & 412 & 17 & 1.9E-5 & 383 & 50 & 1.7E-5 & 238 & 214 & 1.7E-5 \\ \cline{2-13}
        ~ & 16 & 133 & 1.4E-5 & 219 & 16 & 1.5E-5 & 215 & 23 & 1.1E-5 & 133 & 131 & 1.5E-5 \\ \cline{2-13}
        ~ & 32 & 79 & 8.7E-6 & 120 & 14 & 1.9E-5 & 79 & 69 & 1.3E-5 & 79 & 77 & 9.2E-6 \\ \cline{2-13}
        ~ & 64 & 50 & 8.5E-6 & 68 & 11 & 1.5E-5 & 50 & 42 & 1.6E-5 & 50 & 48 & 9.3E-6 \\ \hline \hline
        \multirow{6}{*}{\rotatebox[origin=c]{90}{\ani} } & 2 & 933 & 7.0E-5 & 1232 & 19 & 8.0E-5 & 1172 & 107 & 6.9E-5 & 1036 & 271 & 6.6E-5 \\ \cline{2-13}
        ~ & 4 & 734 & 9.4E-5 & 917 & 19 & 6.8E-5 & 829 & 119 & 6.6E-5 & 750 & 249 & 7.4E-5 \\ \cline{2-13}
        ~ & 8 & 471 & 2.1E-4 & 733 & 20 & 8.5E-5 & 727 & 39 & 8.8E-5 & 605 & 249 & 1.7E-4 \\ \cline{2-13}
        ~ & 16 & 328 & 7.6E-5 & 452 & 21 & 2.1E-4 & 439 & 34 & 2.0E-4 & 342 & 164 & 1.8E-4 \\ \cline{2-13}
        ~ & 32 & 202 & 8.6E-5 & 320 & 10 & 6.8E-5 & 254 & 84 & 6.6E-5 & 231 & 111 & 9.9E-5 \\ \cline{2-13}
        ~ & 64 & 116 & 1.3E-4 & 186 & 17 & 1.1E-4 & 180 & 23 & 1.0E-4 & 123 & 91 & 2.0E-4 \\ \hline\hline
           \multirow{6}{*}{\rotatebox[origin=c]{90}{\sky} } & 2 & 1514 & 8.1E-4 & 2400 & 15 & 6.4E-4 & 2359 & 66 & 6.3E-4 & 2218 & 328 & 9.0E-4 \\ \cline{2-13}
        ~ & 4 & 845 & 3.3E-4 & 1497 & 15 & 5.4E-4 & 1469 & 46 & 8.3E-4 & 1392 & 131 & 7.8E-4 \\ \cline{2-13}
        ~ & 8 & 513 & 2.3E-4 & 830 & 15 & 2.8E-4 & 782 & 64 & 2.4E-4 & 671 & 181 & 4.3E-4 \\ \cline{2-13}
        ~ & 16 & 291 & 2.2E-4 & 503 & 11 & 2.1E-4 & 465 & 49 & 1.1E-4 & 291 & 262 & 2.2E-4 \\ \cline{2-13}
        ~ & 32 & 162 & 1.1E-4 & 277 & 13 & 2.9E-4 & 236 & 52 & 1.4E-4 & 163 & 150 & 1.4E-4 \\ \cline{2-13}
        ~ & 64 & 120 & 2.4E-5 & 150 & 13 & 9.9E-5 & 121 & 68 & 4.4E-5 & 120 & 101 & 2.5E-5 \\ \hline  \hline
    \end{tabular}\vspace{-3mm}
\end{table}
 On the other hand, setting $switchTol$ to $10^{-7}$ leads to a relatively late switch for $t=16,32,64$, implying that the number of iterations is closer to that of SRE-CG2 with the same $t$. Moreover, in this case the memory reduction is minimal for matrices \nho, \skyto \,and \skyo \, as shown in Table \ref{tab:SRECG-mem}, but is between $5\%$ and $20\%$ for  \anio. 

 As for $switchTol = 10^{-5}$, it is the moderate choice that balances between the number of iterations to convergence and the required memory, as the switch occurs earlier than the case of $10^{-7}$ and after that of $10^{-3}$. Thus, number of iterations to convergence for $switchTol = 10^{-5}$ is less than that of $switchTol = 10^{-3}$ and greater than or equal to that of $switchTol = 10^{-7}$. The converse is observed for the memory requirements, since the switch iteration increases as switchTol decreases.  For $t=16,32,64$, the reduction of memory of flexibly SRE-CG2 with respect to SRE-CG2 with the same $t$ is  between $10\%$ and $20\%$ (between $10\%$ and $20\%$ for \nho ,  $10\%$ and $12\%$ for \skyto, $10\%$ and $17\%$ for \anio, and $10\%$ for \skyo). Note that for the matrices \skyo \,and \anio \,with larger condition number requiring more iterations to convergence, leads to a larger number of vectors to be stored.
 
 A similar convergence behavior is observed for MSDO-CG in Table \ref{tab:MSDOCG-flex}, yet it requires more iterations than SRE-CG2, thus requiring more memory storage (Table \ref{tab:MSDOCG-mem}) \vspace{-9mm}
\hspace{-25mm} \begin{table}[H]
\setlength{\tabcolsep}{1pt}
\renewcommand{\arraystretch}{1.1}
\caption{\label{tab:SRECG-mem} Memory requirements of SRE-CG2, restarted SRE-CG2, flexibly SRE-CG2 and restarted flexibly SRE-CG2 with respect to number of partitions $\bf t$, and restart tolerances  $\bf restartTol$. }
\hspace{-13mm}    \begin{tabular}{||c||c||c|c||c|c||c|c||c|c||c|c||c|c||c|c||c|c||c|c||c|c||}
    \cline{3-22}
     \multicolumn{1}{c}{} &\multicolumn{1}{c||}{}&\multicolumn{2}{c||}{\multirow{2}{*}{\textbf{ SRE-CG2}}}&\multicolumn{6}{c||}{\textbf{Restarted, restartTol =} } &\multicolumn{6}{c||}{\textbf{ Flexibly Enlarged, switchTol =} }&\multicolumn{6}{c||}{\textbf{Restarted flexibly, restartTol =} } \\
         \cline{5-22}
    \multicolumn{1}{c}{} &\multicolumn{1}{c||}{}&\multicolumn{1}{c}{} & &\multicolumn{2}{c||}{$\bf 10^{-3}$}&\multicolumn{2}{c||}{$\bf 10^{-5}$}&\multicolumn{2}{c||}{$\bf 10^{-7}$} &\multicolumn{2}{c||}{$\bf 10^{-3}$}&\multicolumn{2}{c||}{$\bf 10^{-5}$}&\multicolumn{2}{c||}{$\bf 10^{-7}$}&\multicolumn{2}{c||}{$\bf 10^{-3}$}&\multicolumn{2}{c||}{$\bf 10^{-5}$}&\multicolumn{2}{c||}{$\bf 10^{-7}$} \\
        \cline{2-22}
    \multicolumn{1}{c||}{} &$\bf t$&$\bf It$&$\bf mem$&$\bf It$&$\bf mem$&$\bf It$&$\bf mem$&$\bf It$&$\bf mem$&$\bf It$&$\bf mem$&$\bf It$&$\bf mem$&$\bf It$&$\bf mem$&$\bf It$&$\bf mem$&$\bf It$&$\bf mem$&$\bf It$&$\bf mem$\\
    \cline{1-22}
     \multirow{6}{*}{\rotatebox[origin=c]{90}{\nh} }  & 2 & 241 & 482 & 240 & 452 & 282 & 480 & 247 & 288 & 282 & 296 & 293 & 335 & 253 & 397 & 264 & 250 & 287 & 245 & 271 & 288 \\ \cline{2-22}
        ~ & 4 & 186 & 744 & 198 & 736 & 260 & 612 & 195 & 620 & 251 & 530 & 207 & 628 & 190 & 690 & 242 & 456 & 286 & 428 & 196 & 620 \\ \cline{2-22}
        ~ & 8 & 147 & 1176 & 158 & 1152 & 175 & 1120 & 182 & 1152 & 192 & 824 & 177 & 848 & 175 & 852 & 198 & 736 & 219 & 736 & 215 & 708 \\ \cline{2-22}
        ~ & 16 & 111 & 1776 & 122 & 1728 & 141 & 1200 & 113 & 1584 & 152 & 1328 & 114 & 1512 & 111 & 1680 & 159 & 1160 & 156 & 1200 & 113 & 1584 \\ \cline{2-22}
        ~ & 32 & 82 & 2624 & 94 & 2560 & 108 & 2432 & 84 & 2400 & 108 & 1952 & 96 & 2048 & 83 & 2528 & 123 & 1744 & 134 & 1632 & 84 & 2400 \\ \cline{2-22}
        ~ & 64 & 59 & 3776 & 72 & 3712 & 69 & 2944 & 60 & 3584 & 77 & 2912 & 60 & 3392 & 59 & 3680 & 94 & 2560 & 70 & 2944 & 60 & 3584 \\ \hline \hline
         \multirow{6}{*}{\rotatebox[origin=c]{90}{\skyt} } & 2 & 569 & 1138 & 575 & 1112 & 623 & 1076 & 640 & 1064 & 710 & 729 & 671 & 756 & 679 & 787 & 718 & 699 & 755 & 670 & 763 & 655 \\ \cline{2-22}
        ~ & 4 & 381 & 1524 & 391 & 1492 & 459 & 1500 & 584 & 1236 & 577 & 1190 & 535 & 1238 & 399 & 1416 & 570 & 1104 & 634 & 1100 & 632 & 1236 \\ \cline{2-22}
        ~ & 8 & 212 & 1696 & 224 & 1672 & 287 & 1672 & 356 & 1624 & 368 & 1532 & 312 & 1560 & 253 & 1624 & 389 & 1496 & 451 & 1492 & 510 & 1428 \\ \cline{2-22}
        ~ & 16 & 117 & 1872 & 130 & 1872 & 199 & 1824 & 202 & 1808 & 200 & 1704 & 134 & 1752 & 117 & 1840 & 224 & 1688 & 293 & 1664 & 207 & 1808 \\ \cline{2-22}
        ~ & 32 & 68 & 2176 & 82 & 2176 & 119 & 2016 & 114 & 2112 & 107 & 1936 & 69 & 2000 & 68 & 2144 & 131 & 1872 & 167 & 1792 & 133 & 2112 \\ \cline{2-22}
        ~ & 64 & 42 & 2688 & 54 & 2688 & 69 & 2240 & 67 & 2624 & 60 & 2304 & 43 & 2464 & 42 & 2656 & 80 & 2176 & 94 & 2176 & 81 & 2624 \\ \hline \hline
         \multirow{6}{*}{\rotatebox[origin=c]{90}{\ani} } & 2 & 875 & 1750 & 893 & 1748 & 920 & 1734 & 1238 & 1646 & 1233 & 1252 & 1206 & 1259 & 884 & 1299 & 1266 & 1247 & 1302 & 1249 & 1615 & 1200 \\ \cline{2-22}
        ~ & 4 & 673 & 2692 & 690 & 2684 & 804 & 2648 & 876 & 2124 & 861 & 1760 & 754 & 1792 & 689 & 2068 & 893 & 1748 & 1000 & 1716 & 1057 & 1424 \\ \cline{2-22}
        ~ & 8 & 447 & 3576 & 457 & 3568 & 557 & 3536 & 579 & 3480 & 683 & 2776 & 597 & 2848 & 584 & 2912 & 684 & 2692 & 770 & 2620 & 776 & 2528 \\ \cline{2-22}
        ~ & 16 & 253 & 4048 & 263 & 4032 & 313 & 4000 & 362 & 3904 & 440 & 3608 & 392 & 3640 & 353 & 3768 & 457 & 3568 & 506 & 3544 & 552 & 3472 \\ \cline{2-22}
        ~ & 32 & 146 & 4672 & 157 & 4672 & 183 & 4608 & 252 & 4064 & 243 & 4064 & 216 & 4080 & 150 & 4400 & 263 & 4032 & 289 & 4000 & 355 & 4000 \\ \cline{2-22}
        ~ & 64 & 91 & 5824 & 102 & 5824 & 124 & 5696 & 143 & 5440 & 138 & 4768 & 116 & 4832 & 100 & 5056 & 156 & 4640 & 178 & 4576 & 199 & 4512 \\ \hline\hline
 \multirow{6}{*}{\rotatebox[origin=c]{90}{\sky}}  & 2 & 1415 & 2830 & 1428 & 2826 & 1436 & 2828 & 1476 & 2830 & 2396 & 2411 & 2416 & 2438 & 2415 & 2476 & 2406 & 2391 & 2413 & 2391 & 2660 & 2599 \\ \cline{2-22}
        ~ & 4 & 754 & 3016 & 769 & 3016 & 860 & 3016 & 904 & 3024 & 1411 & 2852 & 1315 & 2842 & 1290 & 2876 & 1431 & 2832 & 1521 & 2830 & 1563 & 2830 \\ \cline{2-22}
        ~ & 8 & 399 & 3192 & 416 & 3208 & 444 & 3200 & 498 & 3224 & 733 & 2992 & 710 & 3016 & 664 & 3036 & 767 & 3008 & 795 & 3004 & 842 & 2988 \\ \cline{2-22}
        ~ & 16 & 225 & 3600 & 232 & 3536 & 301 & 3552 & 404 & 3440 & 386 & 3176 & 326 & 3240 & 227 & 3328 & 409 & 3184 & 479 & 3200 & 587 & 3184 \\ \cline{2-22}
        ~ & 32 & 124 & 3968 & 137 & 3968 & 180 & 4000 & 210 & 3616 & 208 & 3536 & 169 & 3584 & 125 & 3808 & 234 & 3536 & 276 & 3536 & 295 & 3616 \\ \cline{2-22}
        ~ & 64 & 73 & 4672 & 83 & 4672 & 125 & 4800 & 118 & 4288 & 115 & 4000 & 84 & 4288 & 74 & 4512 & 133 & 3936 & 177 & 4064 & 160 & 4288 \\ \hline \hline 
    \end{tabular}\vspace{-3mm}
\end{table}

In Tables \ref{tab:SRECG-mem} and \ref{tab:MSDOCG-mem},  we also show the convergence and memory requirements of two ``restarted" versions, different than those of section \ref{sec:trunctest}. The first is restarted SRE-CG2 or MSDO-CG where once the relative residual norm is less than restartTol, the method is restarted with the same $t$ value. The second is restarted flexibly SRE-CG2 or flexibly MSDO-CG where once the relative residual norm is less than restartTol, the method is restarted with the $t/2$ value. In both cases, the method is only restarted once and this restart occurs at the same switch iteration of flexibly SRE-CG2 or flexibly MSDO-CG.

Restarting SRE-CG2 with restartTol = $10^{-3}$ has a negligible effect on memory reduction, and in some cases ($t=32,64$ for \skyto, \anio, \skyo) no reduction at all, since after restarting, the method requires the exact number of iterations to converge as SRE-CG2. Decreasing restartTol increases the number of iterations to convergence, i.e. increases the runtime and performed flops, without significantly reducing  the memory requirement. The same behavior is observed when comparing restarted Flexibly  SRE-CG2 to Flexibly  SRE-CG2. For restartTol $=10^{-7}$ or $10^{-5}$, the restarted  Flexibly  SRE-CG2 may converge in up to double the iterations of Flexibly  SRE-CG2 and reduce the memory requirement by at most $15\%$.
\vspace{-7mm} 

\hspace{-25mm} \begin{table}[H]
\setlength{\tabcolsep}{1pt}
\renewcommand{\arraystretch}{1.1}
\caption{\label{tab:MSDOCG-mem} Memory requirements of MSDO-CG, restarted MSDO-CG, flexibly MSDO-CG and restarted flexibly MSDO-CG with respect to number of partitions $\bf t$, and restart tolerances  $\bf restartTol$. }\vspace{-3mm}
\hspace{-13mm}    \begin{tabular}{||c||c||c|c||c|c||c|c||c|c||c|c||c|c||c|c||c|c||c|c||c|c||}
    \cline{3-22}
     \multicolumn{1}{c}{} &\multicolumn{1}{c||}{}&\multicolumn{2}{c||}{\multirow{2}{*}{\textbf{ MSDOCG}}}&\multicolumn{6}{c||}{\textbf{Restarted, restartTol =} } &\multicolumn{6}{c||}{\textbf{ Flexibly Enlarged, switchTol =} }&\multicolumn{6}{c||}{\textbf{Restarted flexibly, restartTol =} } \\
         \cline{5-22}
    \multicolumn{1}{c}{} &\multicolumn{1}{c||}{}&\multicolumn{1}{c}{} & &\multicolumn{2}{c||}{$\bf 10^{-3}$}&\multicolumn{2}{c||}{$\bf 10^{-5}$}&\multicolumn{2}{c||}{$\bf 10^{-7}$} &\multicolumn{2}{c||}{$\bf 10^{-3}$}&\multicolumn{2}{c||}{$\bf 10^{-5}$}&\multicolumn{2}{c||}{$\bf 10^{-7}$}&\multicolumn{2}{c||}{$\bf 10^{-3}$}&\multicolumn{2}{c||}{$\bf 10^{-5}$}&\multicolumn{2}{c||}{$\bf 10^{-7}$} \\
        \cline{2-22}
    \multicolumn{1}{c||}{} &$\bf t$&$\bf It$&$\bf mem$&$\bf It$&$\bf mem$&$\bf It$&$\bf mem$&$\bf It$&$\bf mem$&$\bf It$&$\bf mem$&$\bf It$&$\bf mem$&$\bf It$&$\bf mem$&$\bf It$&$\bf mem$&$\bf It$&$\bf mem$&$\bf It$&$\bf mem$\\
    \cline{1-22}
     \multirow{6}{*}{\rotatebox[origin=c]{90}{\nh} } 
 & 2 & 256 & 512 & 256 & 484 & 277 & 438 & 293 & 300 & 267 & 281 & 257 & 315 & 257 & 400 & 263 & 249 & 270 & 212 & 313 & 286 \\ \cline{2-22}
        ~ & 4 & 206 & 824 & 207 & 772 & 238 & 720 & 251 & 652 & 254 & 536 & 229 & 574 & 208 & 742 & 257 & 486 & 282 & 448 & 256 & 652 \\ \cline{2-22}
        ~ & 8 & 169 & 1352 & 169 & 1240 & 213 & 1152 & 200 & 1160 & 204 & 872 & 183 & 1008 & 170 & 1260 & 209 & 780 & 245 & 704 & 205 & 1160 \\ \cline{2-22}
        ~ & 16 & 139 & 2224 & 141 & 2032 & 172 & 1856 & 173 & 1968 & 165 & 1432 & 152 & 1664 & 139 & 2096 & 169 & 1240 & 202 & 1168 & 174 & 1968 \\ \cline{2-22}
        ~ & 32 & 107 & 3424 & 109 & 3040 & 148 & 2912 & 134 & 3168 & 135 & 2384 & 117 & 2784 & 107 & 3296 & 140 & 2016 & 178 & 1936 & 135 & 3168 \\ \cline{2-22}
        ~ & 64 & 77 & 4928 & 84 & 4352 & 93 & 4224 & 79 & 4736 & 100 & 3712 & 93 & 3840 & 77 & 4832 & 109 & 2976 & 118 & 2912 & 79 & 4736 \\ \hline\hline
         \multirow{6}{*}{\rotatebox[origin=c]{90}{\skyt} }  & 2 & 646 & 1292 & 656 & 1276 & 675 & 1280 & 900 & 958 & 751 & 769 & 744 & 779 & 660 & 1139 & 717 & 699 & 732 & 697 & 900 & 958 \\ \cline{2-22}
        ~ & 4 & 426 & 1704 & 448 & 1700 & 483 & 1696 & 739 & 1520 & 650 & 1346 & 627 & 1372 & 441 & 1642 & 655 & 1264 & 692 & 1266 & 818 & 1520 \\ \cline{2-22}
        ~ & 8 & 231 & 1848 & 247 & 1840 & 280 & 1840 & 398 & 1712 & 412 & 1716 & 383 & 1732 & 238 & 1808 & 442 & 1700 & 474 & 1696 & 552 & 1712 \\ \cline{2-22}
        ~ & 16 & 133 & 2128 & 148 & 2112 & 155 & 2112 & 227 & 2096 & 219 & 1880 & 215 & 1904 & 133 & 2112 & 247 & 1848 & 253 & 1840 & 276 & 2096 \\ \cline{2-22}
        ~ & 32 & 79 & 2528 & 92 & 2496 & 140 & 2272 & 131 & 2464 & 120 & 2144 & 79 & 2368 & 79 & 2496 & 146 & 2112 & 188 & 2208 & 165 & 2464 \\ \cline{2-22}
        ~ & 64 & 50 & 3200 & 60 & 3136 & 85 & 2752 & 82 & 3072 & 68 & 2528 & 50 & 2944 & 50 & 3136 & 89 & 2496 & 111 & 2688 & 106 & 3072 \\ \hline \hline
         \multirow{6}{*}{\rotatebox[origin=c]{90}{\ani} }  & 2 & 933 & 1866 & 949 & 1860 & 1025 & 1836 & 1178 & 1814 & 1232 & 1251 & 1172 & 1279 & 1036 & 1307 & 1266 & 1247 & 1349 & 1242 & 1497 & 1226 \\ \cline{2-22}
        ~ & 4 & 734 & 2936 & 753 & 2936 & 847 & 2912 & 944 & 2780 & 917 & 1872 & 829 & 1896 & 750 & 1998 & 950 & 1862 & 1040 & 1842 & 1139 & 1780 \\ \cline{2-22}
        ~ & 8 & 471 & 3768 & 490 & 3760 & 509 & 3760 & 709 & 3680 & 733 & 3012 & 727 & 3064 & 605 & 3416 & 750 & 2920 & 769 & 2920 & 904 & 2620 \\ \cline{2-22}
        ~ & 16 & 328 & 5248 & 349 & 5248 & 361 & 5232 & 477 & 5008 & 452 & 3784 & 439 & 3784 & 342 & 4048 & 492 & 3768 & 504 & 3760 & 623 & 3672 \\ \cline{2-22}
        ~ & 32 & 202 & 6464 & 211 & 6432 & 281 & 6304 & 302 & 6112 & 320 & 5280 & 254 & 5408 & 231 & 5472 & 338 & 5248 & 404 & 5120 & 427 & 5056 \\ \cline{2-22}
        ~ & 64 & 116 & 7424 & 132 & 7360 & 138 & 7360 & 202 & 7104 & 186 & 6496 & 180 & 6496 & 123 & 6848 & 217 & 6400 & 222 & 6368 & 280 & 6048 \\ \hline\hline
          \multirow{6}{*}{\rotatebox[origin=c]{90}{\sky} }  & 2 & 1514 & 3028 & 1528 & 3026 & 1578 & 3024 & 1908 & 3160 & 2400 & 2415 & 2359 & 2425 & 2218 & 2546 & 2406 & 2391 & 2447 & 2381 & 2761 & 2433 \\ \cline{2-22}
        ~ & 4 & 845 & 3380 & 859 & 3376 & 890 & 3376 & 1295 & 4656 & 1497 & 3024 & 1469 & 3030 & 1392 & 3046 & 1529 & 3028 & 1558 & 3024 & 1647 & 3032 \\ \cline{2-22}
        ~ & 8 & 513 & 4104 & 527 & 4096 & 576 & 4096 & 939 & 6064 & 830 & 3380 & 782 & 3384 & 671 & 3408 & 858 & 3372 & 908 & 3376 & 1022 & 3364 \\ \cline{2-22}
        ~ & 16 & 291 & 4656 & 302 & 4656 & 339 & 4640 & 539 & 4432 & 503 & 4112 & 465 & 4112 & 291 & 4424 & 524 & 4104 & 561 & 4096 & 753 & 4192 \\ \cline{2-22}
        ~ & 32 & 162 & 5184 & 174 & 5152 & 211 & 5088 & 289 & 4800 & 277 & 4640 & 236 & 4608 & 163 & 5008 & 304 & 4656 & 341 & 4624 & 412 & 4800 \\ \cline{2-22}
        ~ & 64 & 120 & 7680 & 132 & 7616 & 185 & 7488 & 217 & 7424 & 150 & 5216 & 121 & 6048 & 120 & 7072 & 174 & 5152 & 227 & 5088 & 255 & 6464 \\ \hline\hline
    \end{tabular}\vspace{-3mm}
\end{table}
Comparing restarted SRE-CG2, Flexibly  SRE-CG2, and restarted Flexibly  SRE-CG2 for restartTol = switchTol $=10^{-5}$, it is observed that restarted Flexibly  SRE-CG2 requires less memory than Flexibly  SRE-CG2, at the expense of more iterations to convergence. The rates of reduction in memory and increase in iterations differ with respect to the matrices and $t$ values. Moreover, Flexibly  SRE-CG2 requires less memory than restarted SRE-CG2, and less iterations for $t=16,32,64$. Thus, Flexibly  SRE-CG2  is the moderate choice that balances between the reduced memory and the augmented number of iterations. 
Moreover, in the case of MSDO-CG a similar convergence behavior is observed in Table \ref{tab:MSDOCG-mem}, where the it is clearer that the flexibly enlarged version requires less memory and iterations than the restarted version.
\subsection{Preconditioned versions}\label{sec:prectest}
We test the preconditioned flexibly SRE-CG2 and MSDO-CG with a Block Jacobi preconditioner.

Given the reordered/permuted matrix $A$ using Metis's kway partitioning 
for $128$ subdomains, we consider two options. The first is to define 64 subdomains by merging 2 consecutive ones. Then 
each of the 64 diagonal blocks of $A$ is  factorized using Cholesky decomposition (Table \ref{tab:precSRECG64chol}) or Incomplete Cholesky zero fill-in decomposition (Table \ref{tab:precSRECG64ichol}). 

The same preconditioner $M = LL^t$ is used for all the $t$ values $2,4,8,16,32,64$, where $L$ is given by \eqref{eq:cholprec}
 with the $L_i$'s being lower triangular blocks for $i=1,2,..,64$, \vspace{1mm}
\begin{eqnarray} L &=&  \begin{bmatrix}
L_1 &0&0&\hdots&0\\
0&\ddots&0&\hdots&0\\
0&0&L_i&\hdots&0\\
0&0&0&\ddots&0\\
0&0&\hdots&0&L_{64}
\end{bmatrix}. \label{eq:cholprec}\end{eqnarray}\vspace{1mm} 
For each $t<64$, the $t$ blocks are the union of $64/t$ consecutive blocks. 

The advantage of a block Jacobi preconditioner is that it is parallelizable without adding any layer of communication to the unpreconditioned algorithm, as processor/core $i$ would compute its corresponding part of the output vector (or block of vectors) using the lower triangular block $L_i$.
For example, computing $\widehat{r}_k = L^{-1} r_{k}$ is equivalent to preforming 64 independent forward substitutions, assuming $L$ has the form \eqref{eq:cholprec}. Computing $P_k = L^{-\mathsf{T}}[T^{t}(\widehat{r}_k)]$ is equivalent to preforming 64 independent block backward substitutions. Similarly, computing $ M^{-1} r_{k}$ is equivalent to performing 64 independent forward substitutions followed by 64 independent backward substitutions.

In Tables \ref{tab:precSRECG64chol} and \ref{tab:precSRECG64ichol}, the results for switchTol = $10^{-5}$ and $10^{-7}$ are only shown, as similarly to the unpreconditioned case, setting switchTol to $10^{-3}$ leads to a very early switch (switch iteration = $O(10)$). Moreover, setting switchTol = $10^{-7}$, leads to a late switch, and in some cases ($t=64$) no switching occurs, implying no memory reduction. Thus again, switchTol = $10^{-5}$ is the best option.

In general, by preconditioning, the number of iterations is significantly reduced, specifically for the ill conditioned matrices \anio\, (from 4179 to 73) and \skyo \, (from 5980 to 283), implying a reduction in required memory in all the methods. Moreover, the obtained solution is better since the relative error is smaller $O(10^{-6}), O(10^{-7}),O(10^{-7})$, as compared to the unpreconditioned case $O(10^{-4}), O(10^{-5}),O(10^{-6})$ for the ill-conditioned matrices.

On the other hand, the difference between the Cholesky and Incomplete Cholesky Block Jacobi Preconditioner is that in the first, $A_i$, the diagonal block of the matrix $A$, is equal to $L_iL_i^{\mathsf{T}}$ (up to numerical errors). However, in the Incomplete Cholesky version, the obtained lower triangular matrix $L_i$ has the same sparsity pattern as the lower triangular part of $A_i$, but $A_i\neq L_iL_i^{\mathsf{T}}$. Comparing the results of Incomplete Cholesky Block Jacobi Preconditioner in Table \ref{tab:precSRECG64ichol} to that of  Cholesky Block Jacobi Preconditioner in Table \ref{tab:precSRECG64chol} shows an increase of 5$\%$ (\anio), 10$\%$ (\skyto), 15$\%$ (\skyo) and 20$\%$ (\nho). 

Note that for more parallelism when applying the preconditioner, it is possible to consider the 128 diagonal blocks and factorize each using  Cholesky decomposition (Table \ref{tab:precSRECG128chol}) or Incomplete Cholesky zero fill-in decomposition (Table \ref{tab:precSRECG128ichol}). Moreover, it is possible to reorder/permute the matrix $A$ using Metis's kway partitioning with 256, 512, or 1024 partitions and use the corresponding diagonal blocks to define the preconditioner. However, with smaller diagonal blocks, and more off-diagonal information neglected, the preconditioner might become less efficient as it is not as good an approximation of $A^{-1}$, i.e. $A^{-1} = L^{-\mathsf{T}} L^{-1} + Err$. 

This effect can be seen Figures \ref{fig:cholnh}, \ref{fig:cholsky3}, \ref{fig:cholani}, and \ref{fig:cholsky2} where we plot the convergence (number of iterations) of preconditioned flexibly SRE-CG2 (left) and  flexibly MSDO-CG (right) for switchTol = $10^{-5}$ with different Block Jacobi preconditioners for $t=4,8,16,32,64$.  We compare the effect of using incomplete Cholesky (iChol) decomposition of the diagonal blocks $A_i$ instead of the full Cholesky (Chol) decomposition, and compare the effect of using 64 diagonal blocks (Chol64, iChol64) versus 128  diagonal blocks (Chol128, iChol128). \vspace{-8mm}

\begin{table}[H]
\setlength{\tabcolsep}{1pt}
\renewcommand{\arraystretch}{1.1}
\caption{\label{tab:precSRECG64chol} Convergence and memory requirements of Block Jacobi Cholesky 64 Preconditioned CG, SRE-CG2, flexibly SRE-CG2, MSDO-CG,and  flexibly MSDO-CG and with respect to number of partitions $\bf t=2,4,8,16,32,64$, and switch tolerances  $\bf switchTol=10^{-5},10^{-7}$. }\vspace{-3mm}
\hspace{-22mm} 
    \begin{tabular}{||c||c|c||c|c|c||c|c|c|c||c|c|c|c||c|c|c||c|c|c|c||c|c|c|c||}
    \cline{2-25}
     \multicolumn{1}{c||}{} &\multicolumn{2}{c||}{\multirow{2}{*}{\textbf{CG}}}&\multicolumn{3}{c||}{\multirow{2}{*}{\textbf{SRE-CG2} }} &\multicolumn{8}{c||}{\textbf{ Flexibly  SRE-CG2, switchTol =} }&\multicolumn{3}{c||}{\multirow{2}{*}{\textbf{MSDOCG} }} &\multicolumn{8}{c||}{\textbf{Flexibly  MSDOCG, switchTol =} } \\
    \multicolumn{1}{c||}{} &\multicolumn{1}{c}{} & &\multicolumn{1}{c}{}&\multicolumn{1}{c}{} & &\multicolumn{4}{c||}{$\bf 10^{-5}$}&\multicolumn{4}{c||}{$\bf 10^{-7}$}&\multicolumn{1}{c}{} &\multicolumn{1}{c}{} &\multicolumn{1}{c||}{}&\multicolumn{4}{c||}{$\bf 10^{-5}$}&\multicolumn{4}{c||}{$\bf 10^{-7}$} 
    \\
        \cline{2-25}
    \multicolumn{1}{c||}{} &$\bf It$&$\bf RE$&$\bf It$&$\bf mem$&$\bf RelEr$&$\bf It$&$\bf sIt$&$\bf mem$&$\bf RelEr$&$\bf It$&$\bf sIt$&$\bf mem$&$\bf RelEr$&$\bf It$&$\bf mem$&$\bf RelEr$&$\bf It$&$\bf sIt$&$\bf mem$&$\bf RelEr$&$\bf It$&$\bf sIt$&$\bf mem$&$\bf RelEr$
   \\
    \cline{1-25}
     \multirow{6}{*}{\rotatebox[origin=c]{90}{\nh} }  &  \multirow{6}{*}{\rotatebox[origin=c]{90}{91}} & \multirow{6}{*}{\rotatebox[origin=c]{90}{ 3.81E-7}}  & 79 & 158 & 1.1E-7 & 93 & 16 & 109 & 1.4E-7 & 79 & 66 & 145 & 1.2E-7 & 79 & 158 & 2.7E-7 & 89 & 18 & 107 & 2.3E-7 & 79 & 68 & 147 & 3.2E-7 \\ \cline{4-25}
        ~ & ~ & ~ & 64 & 256 & 8.5E-8 & 68 & 35 & 206 & 7.8E-8 & 64 & 57 & 242 & 8.2E-8 & 68 & 272 & 8.1E-8 & 72 & 43 & 230 & 7.8E-8 & 68 & 60 & 256 & 8.4E-8 \\ \cline{4-25}
        ~ & ~ & ~ & 51 & 408 & 3.0E-8 & 51 & 35 & 344 & 6.2E-8 & 51 & 47 & 392 & 3.5E-8 & 56 & 448 & 5.9E-8 & 57 & 40 & 388 & 7.7E-8 & 56 & 51 & 428 & 7.3E-8 \\ \cline{4-25}
        ~ & ~ & ~ & 39 & 624 & 2.3E-8 & 39 & 29 & 544 & 3.0E-8 & 39 & 37 & 608 & 2.3E-8 & 46 & 736 & 3.2E-8 & 46 & 35 & 648 & 4.9E-8 & 46 & 43 & 712 & 4.5E-8 \\ \cline{4-25}
        ~ & ~ & ~ & 31 & 992 & 8.1E-9 & 31 & 24 & 880 & 1.5E-8 & 31 & 30 & 976 & 8.5E-9 & 35 & 1120 & 2.1E-8 & 38 & 17 & 880 & 2.7E-8 & 35 & 34 & 1104 & 2.2E-8 \\ \cline{4-25}
        ~ & ~ & ~ & 24 & 1536 & 2.7E-9 & 24 & 20 & 1408 & 4.8E-9 & 24 & - & 1536 & 2.8E-9 & 27 & 1728 & 1.8E-8 & 28 & 24 & 1664 & 6.0E-9 & 27 & - & 1728 & 2.0E-8 \\ \hline \hline
        \multirow{6}{*}{\rotatebox[origin=c]{90}{\skyt} }  &  \multirow{6}{*}{\rotatebox[origin=c]{90}{227}} & \multirow{6}{*}{\rotatebox[origin=c]{90}{ 1.41E-5 }}  & 186 & 372 & 5.5E-6 & 202 & 42 & 244 & 1.3E-5 & 190 & 125 & 315 & 7.6E-6 & 210 & 420 & 1.4E-5 & 214 & 32 & 246 & 1.4E-5 & 214 & 147 & 361 & 1.5E-5 \\ \cline{4-25}
        ~ & ~ & ~ & 156 & 624 & 1.7E-6 & 167 & 64 & 462 & 4.8E-6 & 157 & 113 & 540 & 7.6E-6 & 205 & 820 & 7.6E-6 & 227 & 13 & 480 & 8.3E-6 & 206 & 130 & 672 & 1.3E-5 \\ \cline{4-25}
        ~ & ~ & ~ & 101 & 808 & 5.6E-7 & 127 & 40 & 668 & 4.9E-6 & 103 & 92 & 780 & 1.1E-6 & 141 & 1128 & 2.2E-6 & 177 & 31 & 832 & 7.1E-6 & 141 & 115 & 1024 & 1.2E-5 \\ \cline{4-25}
        ~ & ~ & ~ & 56 & 896 & 5.3E-7 & 66 & 40 & 848 & 3.2E-7 & 56 & - & 896 & 5.3E-7 & 96 & 1536 & 8.2E-6 & 130 & 19 & 1192 & 1.2E-5 & 105 & 68 & 1384 & 1.2E-5 \\ \cline{4-25}
        ~ & ~ & ~ & 34 & 1088 & 1.6E-7 & 40 & 18 & 928 & 8.7E-7 & 34 & - & 1088 & 1.6E-7 & 72 & 2304 & 1.6E-6 & 88 & 23 & 1776 & 4.5E-6 & 73 & 67 & 2240 & 5.8E-7 \\ \cline{4-25}
        ~ & ~ & ~ & 22 & 1408 & 3.3E-7 & 22 & 18 & 1280 & 4.8E-7 & 22 & - & 1408 & 3.3E-7 & 45 & 2880 & 5.8E-7 & 63 & 10 & 2336 & 1.4E-6 & 45 & - & 2880 & 5.8E-7 \\ \hline\hline
          \multirow{6}{*}{\rotatebox[origin=c]{90}{\ani} }  &  \multirow{6}{*}{\rotatebox[origin=c]{90}{73}} & \multirow{6}{*}{\rotatebox[origin=c]{90}{1.17E-6  }} & 71 & 142 & 1.6E-6 & 73 & 19 & 92 & 1.3E-6 & 73 & 55 & 128 & 2.2E-6 & 73 & 146 & 3.4E-7 & 74 & 31 & 105 & 5.5E-7 & 73 & 57 & 130 & 3.6E-7 \\ \cline{4-25}
        ~ & ~ & ~ & 66 & 264 & 1.2E-6 & 72 & 17 & 178 & 6.1E-7 & 67 & 54 & 242 & 7.6E-7 & 68 & 272 & 8.0E-7 & 70 & 26 & 192 & 5.2E-7 & 68 & 55 & 246 & 8.1E-7 \\ \cline{4-25}
        ~ & ~ & ~ & 61 & 488 & 6.6E-7 & 67 & 22 & 356 & 8.0E-7 & 61 & 50 & 444 & 7.0E-7 & 66 & 528 & 6.7E-7 & 67 & 26 & 372 & 7.7E-7 & 66 & 55 & 484 & 6.8E-7 \\ \cline{4-25}
        ~ & ~ & ~ & 56 & 896 & 1.0E-6 & 58 & 32 & 720 & 9.6E-7 & 56 & 44 & 800 & 1.7E-6 & 61 & 976 & 1.9E-6 & 63 & 29 & 736 & 1.1E-6 & 61 & 48 & 872 & 2.0E-6 \\ \cline{4-25}
        ~ & ~ & ~ & 44 & 1408 & 7.3E-6 & 49 & 23 & 1152 & 3.9E-6 & 46 & 39 & 1360 & 4.5E-6 & 54 & 1728 & 3.9E-6 & 56 & 25 & 1296 & 3.4E-6 & 54 & 47 & 1616 & 4.4E-6 \\ \cline{4-25}
        ~ & ~ & ~ & 36 & 2304 & 5.7E-6 & 38 & 22 & 1920 & 9.3E-6 & 37 & 31 & 2176 & 7.4E-6 & 43 & 2752 & 5.0E-6 & 47 & 23 & 2240 & 4.1E-6 & 43 & 39 & 2624 & 5.5E-6 \\ \hline\hline
        \multirow{6}{*}{\rotatebox[origin=c]{90}{\sky} }  &  \multirow{6}{*}{\rotatebox[origin=c]{90}{283}} & \multirow{6}{*}{\rotatebox[origin=c]{90}{ 1.21E-4}}  &  183 & 366 & 2.6E-4 & 230 & 26 & 256 & 1.7E-4 & 229 & 28 & 257 & 1.8E-4 & 239 & 478 & 1.8E-6 & 242 & 26 & 268 & 1.5E-4 & 221 & 124 & 345 & 6.4E-4 \\ \cline{4-25}
        ~ & ~ & ~ & 106 & 424 & 9.1E-7 & 160 & 22 & 364 & 1.4E-4 & 106 & 94 & 400 & 1.2E-6 & 168 & 672 & 1.3E-6 & 201 & 25 & 452 & 2.7E-4 & 221 & 34 & 510 & 1.3E-6 \\ \cline{4-25}
        ~ & ~ & ~ & 68 & 544 & 8.9E-7 & 89 & 23 & 448 & 1.3E-6 & 68 & 65 & 532 & 9.2E-7 & 117 & 936 & 1.7E-6 & 138 & 29 & 668 & 1.6E-6 & 117 & 116 & 932 & 1.7E-6 \\ \cline{4-25}
        ~ & ~ & ~ & 44 & 704 & 6.4E-7 & 53 & 19 & 576 & 1.0E-6 & 44 & 42 & 688 & 6.7E-7 & 80 & 1280 & 6.2E-7 & 101 & 21 & 976 & 9.1E-7 & 80 & - & 1280 & 6.2E-7 \\ \cline{4-25}
        ~ & ~ & ~ & 32 & 1024 & 5.3E-7 & 35 & 18 & 848 & 1.3E-6 & 32 & - & 1024 & 5.7E-7 & 52 & 1664 & 9.3E-7 & 69 & 16 & 1360 & 1.0E-6 & 52 & - & 1664 & 9.3E-7 \\ \cline{4-25}
        ~ & ~ & ~ & 20 & 1280 & 4.7E-7 & 21 & 16 & 1184 & 5.3E-7 & 20 & - & 1280 & 5.6E-7 & 37 & 2368 & 3.5E-7 & 45 & 13 & 1856 & 4.2E-7 & 37 & - & 2368 & 3.5E-7 \\ \hline\hline
    \end{tabular}\vspace{-3mm}
\end{table}

\begin{figure}[H] 
 \centering
 \hspace{-8mm} 
 \includegraphics[width=0.55\textwidth]{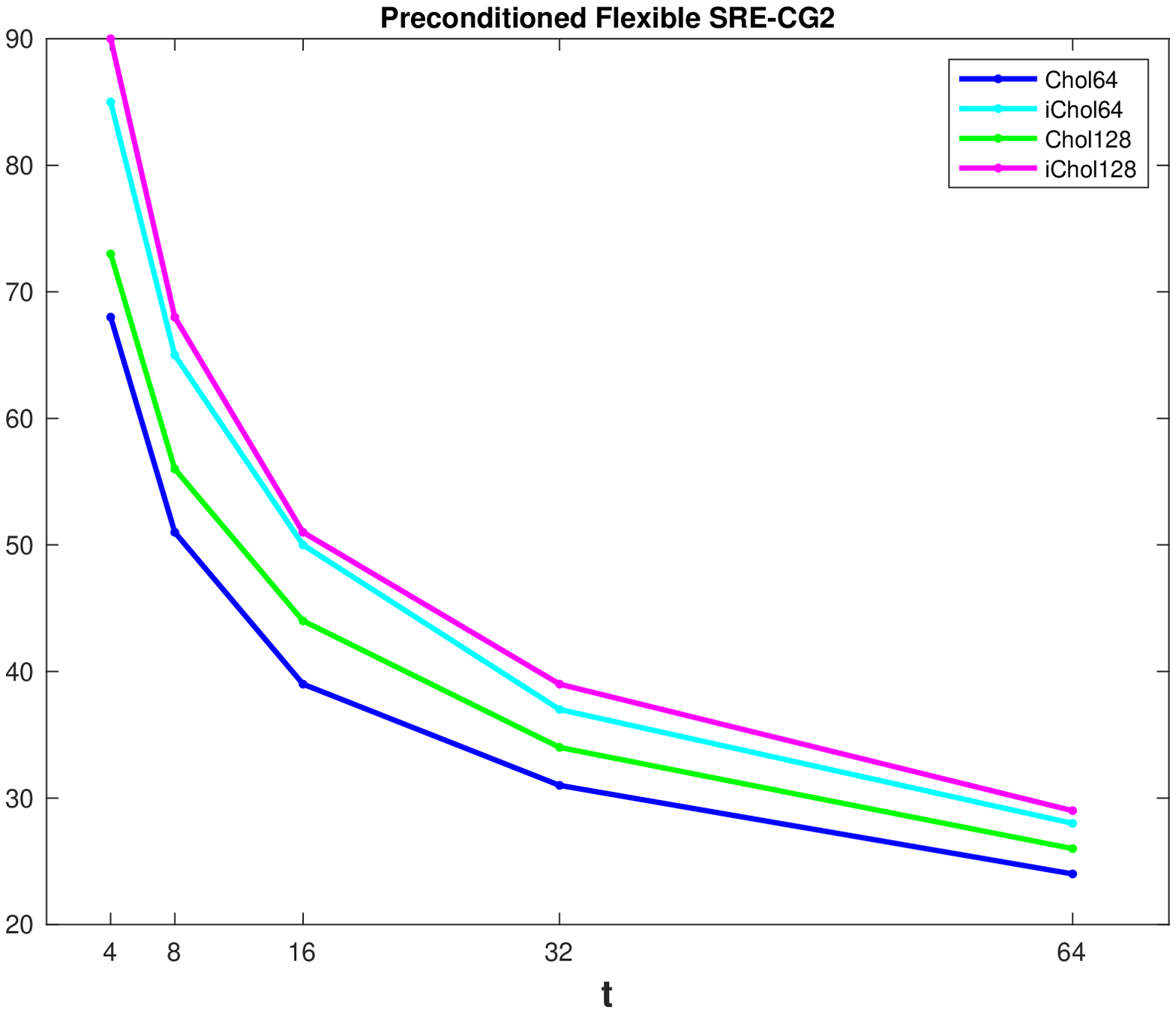} \hspace{-8mm} 
\includegraphics[width=0.55\textwidth]{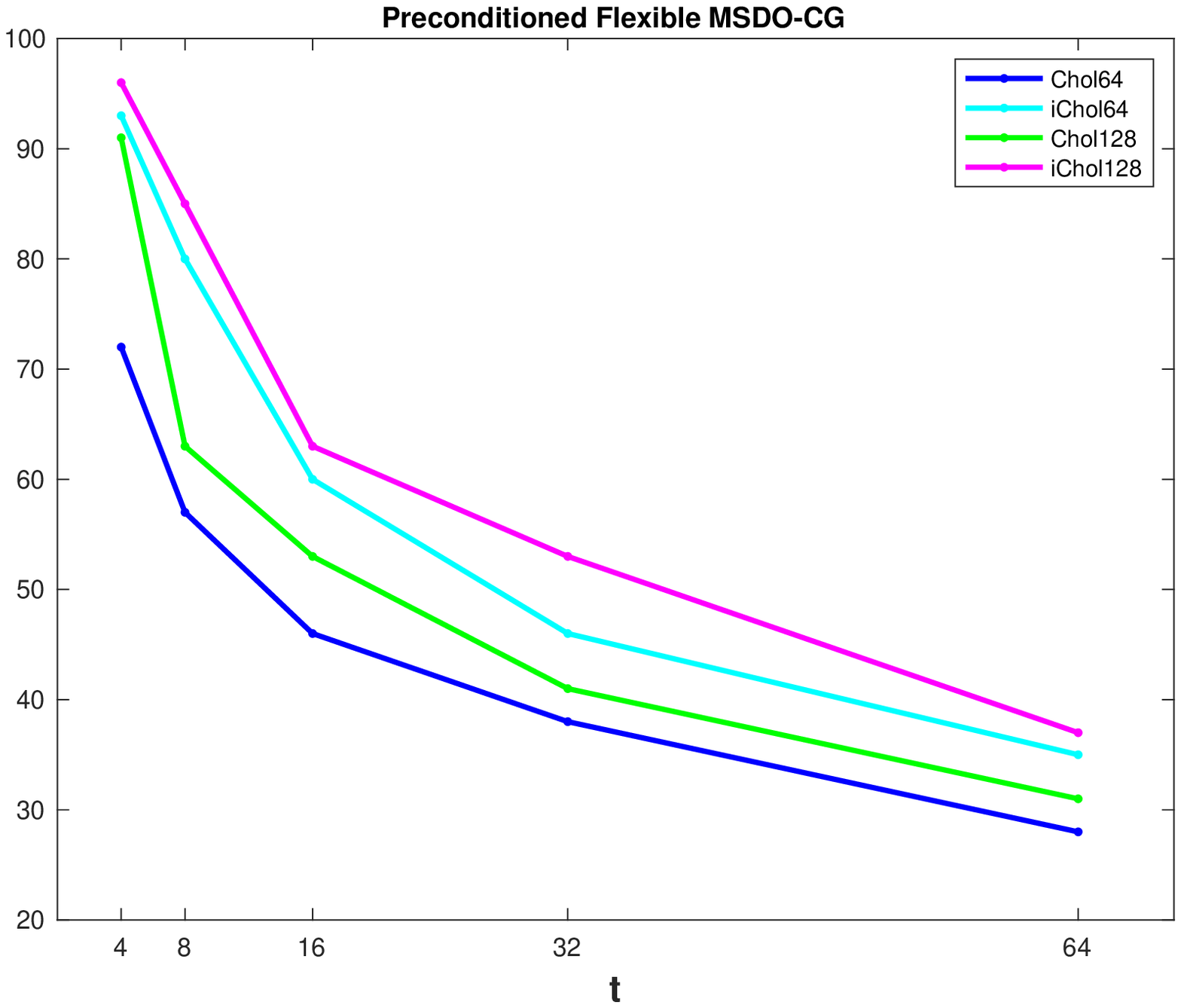}\\
\vspace{-2mm}
   \caption{Convergence of preconditioned flexibly enlarged methods with switchTol = $10^{-5}$ for matrix \nho} \label{fig:cholnh}\vspace{-10mm}
\end{figure}

\begin{table}[H]
\setlength{\tabcolsep}{1pt}
\renewcommand{\arraystretch}{1.1}
\caption{\label{tab:precSRECG64ichol} Convergence and memory requirements of Block Jacobi Incomplete Cholesky 64 Preconditioned CG, SRE-CG2, flexibly SRE-CG2, MSDO-CG,and  flexibly MSDO-CG and with respect to number of partitions $\bf t=2,4,8,16,32,64$, and switch tolerances  $\bf switchTol=10^{-5},10^{-7}$. }\vspace{-3mm}
\hspace{-22mm} 
    \begin{tabular}{||c||c|c||c|c|c||c|c|c|c||c|c|c|c||c|c|c||c|c|c|c||c|c|c|c||}
    \cline{2-25}
     \multicolumn{1}{c||}{} &\multicolumn{2}{c||}{\multirow{2}{*}{\textbf{CG}}}&\multicolumn{3}{c||}{\multirow{2}{*}{\textbf{SRE-CG2} }} &\multicolumn{8}{c||}{\textbf{ Flexibly  SRE-CG2, switchTol =} }&\multicolumn{3}{c||}{\multirow{2}{*}{\textbf{MSDOCG} }} &\multicolumn{8}{c||}{\textbf{Flexibly  MSDOCG, switchTol =} } \\
    \multicolumn{1}{c||}{} &\multicolumn{1}{c}{} & &\multicolumn{1}{c}{}&\multicolumn{1}{c}{} & &\multicolumn{4}{c||}{$\bf 10^{-5}$}&\multicolumn{4}{c||}{$\bf 10^{-7}$}&\multicolumn{1}{c}{} &\multicolumn{1}{c}{} &\multicolumn{1}{c||}{}&\multicolumn{4}{c||}{$\bf 10^{-5}$}&\multicolumn{4}{c||}{$\bf 10^{-7}$} 
    \\
        \cline{2-25}
    \multicolumn{1}{c||}{} &$\bf It$&$\bf RE$&$\bf It$&$\bf mem$&$\bf RelEr$&$\bf It$&$\bf sIt$&$\bf mem$&$\bf RelEr$&$\bf It$&$\bf sIt$&$\bf mem$&$\bf RelEr$&$\bf It$&$\bf mem$&$\bf RelEr$&$\bf It$&$\bf sIt$&$\bf mem$&$\bf RelEr$&$\bf It$&$\bf sIt$&$\bf mem$&$\bf RelEr$
   \\
    \cline{1-25}
     \multirow{6}{*}{\rotatebox[origin=c]{90}{\nh} }   &  \multirow{6}{*}{\rotatebox[origin=c]{90}{116}} & \multirow{6}{*}{\rotatebox[origin=c]{90}{ 3.19E-7}}  & 102 & 204 & 1.6E-7 & 118 & 34 & 152 & 2.2E-7 & 105 & 66 & 171 & 2.2E-7 & 109 & 218 & 2.3E-7 & 118 & 56 & 174 & 1.6E-7 & 111 & 75 & 186 & 1.6E-7 \\ \cline{4-25}
        ~ & ~ & ~ & 81 & 324 & 1.1E-7 & 85 & 53 & 276 & 1.0E-7 & 81 & 71 & 304 & 1.5E-7 & 88 & 352 & 1.7E-7 & 93 & 56 & 298 & 9.3E-8 & 89 & 75 & 328 & 1.5E-7 \\ \cline{4-25}
        ~ & ~ & ~ & 63 & 504 & 5.3E-8 & 65 & 44 & 436 & 6.0E-8 & 64 & 58 & 488 & 4.8E-8 & 72 & 576 & 4.7E-8 & 80 & 26 & 424 & 7.6E-8 & 72 & 65 & 548 & 5.6E-8 \\ \cline{4-25}
        ~ & ~ & ~ & 49 & 784 & 2.0E-8 & 50 & 35 & 680 & 2.5E-8 & 49 & 45 & 752 & 2.4E-8 & 59 & 944 & 4.1E-8 & 60 & 45 & 840 & 3.7E-8 & 59 & 56 & 920 & 4.3E-8 \\ \cline{4-25}
        ~ & ~ & ~ & 36 & 1152 & 1.8E-8 & 37 & 28 & 1040 & 1.9E-8 & 36 & 35 & 1136 & 1.8E-8 & 46 & 1472 & 2.2E-8 & 46 & 38 & 1344 & 3.0E-8 & 46 & 44 & 1440 & 2.6E-8 \\ \cline{4-25}
        ~ & ~ & ~ & 28 & 1792 & 7.1E-9 & 28 & 23 & 1632 & 1.1E-8 & 28 & - & 1792 & 7.0E-9 & 35 & 2240 & 1.2E-8 & 35 & 31 & 2112 & 1.7E-8 & 35 & - & 2240 & 1.3E-8 \\ \hline\hline
               \multirow{6}{*}{\rotatebox[origin=c]{90}{\skyt} }   &  \multirow{6}{*}{\rotatebox[origin=c]{90}{250}} & \multirow{6}{*}{\rotatebox[origin=c]{90}{ 1.04E-5}}  & 198 & 396 & 1.4E-5 & 225 & 26 & 251 & 1.6E-5 & 208 & 161 & 369 & 8.6E-6 & 250 & 500 & 1.2E-5 & 240 & 55 & 295 & 8.9E-6 & 258 & 195 & 453 & 1.0E-5 \\ \cline{4-25}
        ~ & ~ & ~ & 164 & 656 & 3.1E-6 & 208 & 14 & 444 & 4.5E-6 & 168 & 131 & 598 & 2.9E-6 & 211 & 844 & 1.3E-5 & 229 & 34 & 526 & 2.1E-5 & 218 & 180 & 796 & 9.7E-6 \\ \cline{4-25}
        ~ & ~ & ~ & 116 & 928 & 3.8E-7 & 141 & 41 & 728 & 3.4E-6 & 120 & 104 & 896 & 6.0E-7 & 146 & 1168 & 1.7E-5 & 199 & 22 & 884 & 4.7E-6 & 158 & 98 & 1024 & 9.3E-6 \\ \cline{4-25}
        ~ & ~ & ~ & 65 & 1040 & 3.0E-7 & 90 & 32 & 976 & 2.1E-7 & 68 & 58 & 1008 & 4.3E-7 & 115 & 1840 & 3.7E-6 & 140 & 23 & 1304 & 8.3E-6 & 118 & 96 & 1712 & 3.3E-6 \\ \cline{4-25}
        ~ & ~ & ~ & 39 & 1248 & 2.5E-7 & 41 & 32 & 1168 & 4.3E-7 & 39 & - & 1248 & 2.5E-7 & 80 & 2560 & 9.5E-6 & 90 & 28 & 1888 & 8.0E-6 & 81 & 76 & 2512 & 9.5E-6 \\ \cline{4-25}
        ~ & ~ & ~ & 25 & 1600 & 2.6E-7 & 31 & 10 & 1312 & 2.8E-7 & 25 & - & 1600 & 2.7E-7 & 53 & 3392 & 3.6E-7 & 72 & 16 & 2816 & 2.6E-7 & 53 & - & 3392 & 3.6E-7 \\ \hline\hline
        \multirow{6}{*}{\rotatebox[origin=c]{90}{\ani} }   &  \multirow{6}{*}{\rotatebox[origin=c]{90}{77}} & \multirow{6}{*}{\rotatebox[origin=c]{90}{ 2.02E-6}}    & 75 & 150 & 6.5E-7 & 77 & 28 & 105 & 2.1E-6 & 78 & 61 & 139 & 8.2E-7 & 74 & 148 & 3.7E-7 & 79 & 27 & 106 & 5.2E-7 & 74 & 61 & 135 & 4.0E-7 \\ \cline{4-25}
        ~ & ~ & ~ & 68 & 272 & 1.0E-6 & 69 & 28 & 194 & 5.8E-7 & 68 & 56 & 248 & 9.0E-7 & 70 & 280 & 8.0E-7 & 72 & 32 & 208 & 5.3E-7 & 70 & 58 & 256 & 8.0E-7 \\ \cline{4-25}
        ~ & ~ & ~ & 64 & 512 & 7.3E-7 & 65 & 28 & 372 & 9.5E-7 & 64 & 52 & 464 & 6.8E-7 & 68 & 544 & 7.1E-7 & 68 & 31 & 396 & 7.4E-7 & 68 & 58 & 504 & 7.0E-7 \\ \cline{4-25}
        ~ & ~ & ~ & 58 & 928 & 1.5E-6 & 62 & 26 & 704 & 7.8E-7 & 58 & 51 & 872 & 1.7E-6 & 63 & 1008 & 1.9E-6 & 63 & 31 & 752 & 1.4E-6 & 63 & 53 & 928 & 1.5E-6 \\ \cline{4-25}
        ~ & ~ & ~ & 48 & 1536 & 4.6E-6 & 51 & 23 & 1184 & 4.4E-6 & 49 & 40 & 1424 & 4.6E-6 & 56 & 1792 & 3.8E-6 & 59 & 26 & 1360 & 2.8E-6 & 57 & 50 & 1712 & 2.9E-6 \\ \cline{4-25}
        ~ & ~ & ~ & 38 & 2432 & 6.2E-6 & 40 & 24 & 2048 & 8.7E-6 & 39 & 33 & 2304 & 5.5E-6 & 47 & 3008 & 3.0E-6 & 49 & 23 & 2304 & 5.1E-6 & 48 & 42 & 2880 & 2.2E-6 \\ \hline\hline
          \multirow{6}{*}{\rotatebox[origin=c]{90}{\sky} }   &  \multirow{6}{*}{\rotatebox[origin=c]{90}{325}} & \multirow{6}{*}{\rotatebox[origin=c]{90}{ 1.61E-4}}  & 229 & 458 & 5.2E-6 & 267 & 28 & 295 & 4.6E-4 & 258 & 70 & 328 & 3.2E-4 & 299 & 598 & 1.2E-4 & 287 & 28 & 315 & 1.7E-4 & 262 & 112 & 374 & 8.2E-4 \\ \cline{4-25}
        ~ & ~ & ~ & 139 & 556 & 6.7E-7 & 216 & 27 & 486 & 1.1E-6 & 147 & 100 & 494 & 1.2E-6 & 199 & 796 & 2.2E-6 & 260 & 18 & 556 & 7.7E-4 & 236 & 110 & 692 & 2.2E-6 \\ \cline{4-25}
        ~ & ~ & ~ & 83 & 664 & 1.4E-6 & 113 & 26 & 556 & 1.4E-6 & 112 & 31 & 572 & 1.1E-6 & 140 & 1120 & 8.9E-7 & 180 & 19 & 796 & 1.3E-6 & 157 & 89 & 984 & 2.5E-6 \\ \cline{4-25}
        ~ & ~ & ~ & 54 & 864 & 1.1E-6 & 62 & 23 & 680 & 1.7E-6 & 54 & 51 & 840 & 1.2E-6 & 98 & 1568 & 8.4E-7 & 132 & 20 & 1216 & 1.1E-6 & 117 & 41 & 1264 & 2.0E-6 \\ \cline{4-25}
        ~ & ~ & ~ & 40 & 1280 & 2.7E-7 & 43 & 23 & 1056 & 7.0E-7 & 40 & 39 & 1264 & 3.0E-7 & 71 & 2272 & 1.9E-7 & 88 & 20 & 1728 & 5.5E-7 & 71 & - & 2272 & 1.9E-7 \\ \cline{4-25}
        ~ & ~ & ~ & 25 & 1600 & 3.2E-7 & 25 & 20 & 1440 & 5.3E-7 & 25 & - & 1600 & 3.3E-7 & 46 & 2944 & 8.5E-7 & 62 & 17 & 2528 & 2.8E-7 & 46 & - & 2944 & 8.5E-7 \\ \hline\hline
         \end{tabular}
\end{table}
\newpage

\begin{figure}[H] 
 \centering
 \hspace{-8mm}  \includegraphics[width=0.55\textwidth]{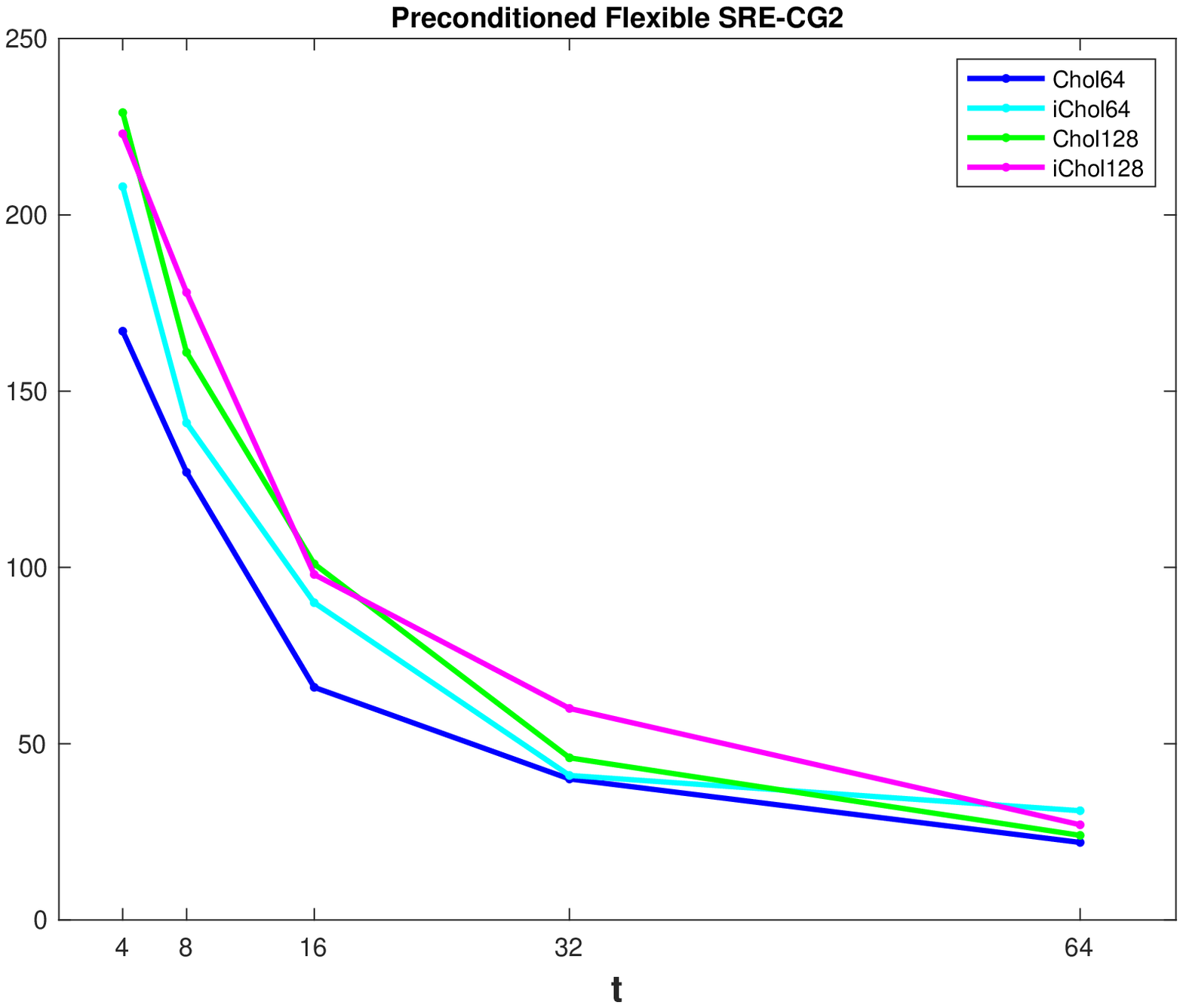} \hspace{-8mm} 
 \includegraphics[width=0.55\textwidth]{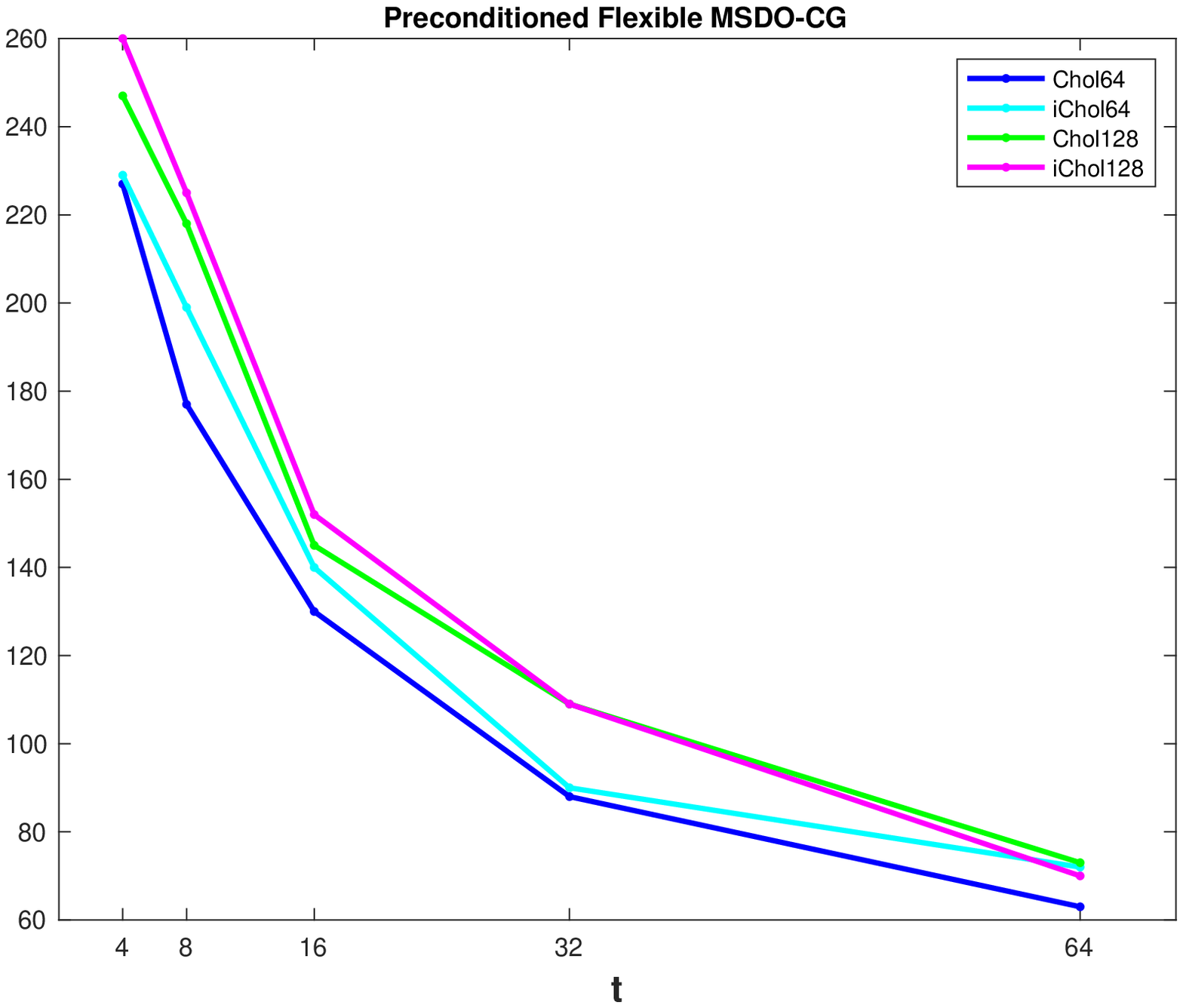}\\
\vspace{-2mm}
   \caption{Convergence of preconditioned flexibly enlarged methods with switchTol = $10^{-5}$ for matrix \skyto} \label{fig:cholsky3}\vspace{-10mm}
\end{figure}

\begin{table}[H]
\setlength{\tabcolsep}{1pt}
\renewcommand{\arraystretch}{1.1}
\caption{\label{tab:precSRECG128chol} Convergence and memory requirements of Block Jacobi Cholesky 128 Preconditioned CG, SRE-CG2, flexibly SRE-CG2, MSDO-CG,and  flexibly MSDO-CG and with respect to number of partitions $\bf t=2,4,8,16,32,64$, and switch tolerances  $\bf switchTol=10^{-5},10^{-7}$. }\vspace{-3mm}
\hspace{-22mm} 
    \begin{tabular}{||c||c|c||c|c|c||c|c|c|c||c|c|c|c||c|c|c||c|c|c|c||c|c|c|c||}
    \cline{2-25}
     \multicolumn{1}{c||}{} &\multicolumn{2}{c||}{\multirow{2}{*}{\textbf{CG}}}&\multicolumn{3}{c||}{\multirow{2}{*}{\textbf{SRE-CG2} }} &\multicolumn{8}{c||}{\textbf{ Flexibly  SRE-CG2, switchTol =} }&\multicolumn{3}{c||}{\multirow{2}{*}{\textbf{MSDOCG} }} &\multicolumn{8}{c||}{\textbf{Flexibly  MSDOCG, switchTol =} } \\
    \multicolumn{1}{c||}{} &\multicolumn{1}{c}{} & &\multicolumn{1}{c}{}&\multicolumn{1}{c}{} & &\multicolumn{4}{c||}{$\bf 10^{-5}$}&\multicolumn{4}{c||}{$\bf 10^{-7}$}&\multicolumn{1}{c}{} &\multicolumn{1}{c}{} &\multicolumn{1}{c||}{}&\multicolumn{4}{c||}{$\bf 10^{-5}$}&\multicolumn{4}{c||}{$\bf 10^{-7}$} 
    \\
        \cline{2-25}
    \multicolumn{1}{c||}{} &$\bf It$&$\bf RE$&$\bf It$&$\bf mem$&$\bf RelEr$&$\bf It$&$\bf sIt$&$\bf mem$&$\bf RelEr$&$\bf It$&$\bf sIt$&$\bf mem$&$\bf RelEr$&$\bf It$&$\bf mem$&$\bf RelEr$&$\bf It$&$\bf sIt$&$\bf mem$&$\bf RelEr$&$\bf It$&$\bf sIt$&$\bf mem$&$\bf RelEr$
   \\
    \cline{1-25}
     \multirow{6}{*}{\rotatebox[origin=c]{90}{\nh} } 
          &  \multirow{6}{*}{\rotatebox[origin=c]{90}{101}} & \multirow{6}{*}{\rotatebox[origin=c]{90}{ 2.77E-7}} & 89 & 178 & 1.6E-7 & 96 & 37 & 133 & 2.2E-7 & 89 & 77 & 166 & 1.6E-7 & 92 & 184 & 1.5E-7 & 99 & 49 & 148 & 1.4E-7 & 92 & 78 & 170 & 1.5E-7 \\ \cline{4-25}
        ~ & ~ & ~ & 71 & 284 & 7.6E-8 & 73 & 46 & 238 & 8.4E-8 & 71 & 62 & 266 & 8.0E-8 & 77 & 308 & 9.1E-8 & 91 & 8 & 198 & 1.4E-7 & 79 & 66 & 290 & 8.0E-8 \\ \cline{4-25}
        ~ & ~ & ~ & 55 & 440 & 3.8E-8 & 56 & 39 & 380 & 4.7E-8 & 55 & 50 & 420 & 3.91E-8 & 62 & 496 & 7.0E-8 & 63 & 44 & 428 & 7.8E-8 & 62 & 56 & 472 & 8.0E-8 \\ \cline{4-25}
        ~  & ~ & ~ & 43 & 688 & 2.4E-8 & 44 & 31 & 600 & 2.1E-8 & 43 & 40 & 664 & 2.5E-8 & 52 & 832 & 2.4E-8 & 53 & 39 & 736 & 3.6E-8 & 52 & 48 & 800 & 2.7E-8 \\ \cline{4-25}
        ~  & ~ & ~ & 34 & 1088 & 1.3E-8 & 34 & 26 & 960 & 1.9E-8 & 34 & 33 & 1072 & 1.3E-8 & 40 & 1280 & 2.4E-8 & 41 & 32 & 1168 & 1.8E-8 & 40 & 39 & 1264 & 2.4E-8 \\ \cline{4-25}
        ~  & ~ & ~ & 26 & 1664 & 7.5E-9 & 26 & 21 & 1504 & 1.3E-8 & 26 & - & 1664 & 7.6E-9 & 31 & 1984 & 1.1E-8 & 31 & 27 & 1856 & 1.8E-8 & 31 & - & 1984 & 1.3E-8 \\ \hline\hline
         \multirow{6}{*}{\rotatebox[origin=c]{90}{\skyt}}  & \multirow{6}{*}{\rotatebox[origin=c]{90}{250}} & \multirow{6}{*}{\rotatebox[origin=c]{90}{9.23E-6}} & 210 & 420 & 1.4E-5 & 243 & 15 & 258 & 8.1E-6 & 221 & 103 & 324 & 1.3E-5 & 245 & 490 & 1.4E-5 & 252 & 60 & 312 & 1.4E-5 & 244 & 180 & 424 & 1.9E-5 \\ \cline{4-25}
        ~  & ~ & ~ & 166 & 664 & 6.8E-6 & 229 & 15 & 488 & 3.4E-6 & 178 & 106 & 568 & 4.9E-6 & 230 & 920 & 8.7E-6 & 247 & 39 & 572 & 1.6E-5 & 251 & 94 & 690 & 1.0E-5 \\ \cline{4-25}
        ~ & ~ & ~ & 119 & 952 & 2.8E-7 & 161 & 25 & 744 & 4.8E-6 & 119 & 101 & 880 & 8.4E-6 & 156 & 1248 & 3.3E-6 & 218 & 13 & 924 & 1.3E-5 & 161 & 128 & 1156 & 4.5E-6 \\ \cline{4-25}
        ~  & ~ & ~ & 65 & 1040 & 2.5E-7 & 101 & 18 & 952 & 4.9E-7 & 65 & 65 & 1040 & 2.6E-7 & 116 & 1856 & 9.1E-6 & 145 & 22 & 1336 & 7.2E-6 & 130 & 68 & 1584 & 3.0E-6 \\ \cline{4-25}
        ~  & ~ & ~ & 39 & 1248 & 2.3E-7 & 46 & 23 & 1104 & 5.2E-7 & 39 & - & 1248 & 2.3E-7 & 85 & 2720 & 5.8E-7 & 109 & 12 & 1936 & 7.6E-6 & 86 & 73 & 2544 & 3.6E-7 \\ \cline{4-25}
        ~ & ~ & ~ & 24 & 1536 & 6.4E-7 & 24 & 21 & 1440 & 7.0E-7 & 24 & - & 1536 & 6.5E-7 & 53 & 3392 & 8.2E-7 & 73 & 15 & 2816 & 2.00E-6 & 56 & 48 & 3328 & 1.6E-7 \\ \hline\hline
         \multirow{6}{*}{\rotatebox[origin=c]{90}{\ani} }   & \multirow{6}{*}{\rotatebox[origin=c]{90}{77}} & \multirow{6}{*}{\rotatebox[origin=c]{90}{5.8E-7}} & 74 & 148 & 8.9E-7 & 76 & 31 & 107 & 7.3E-7 & 74 & 61 & 135 & 9.2E-7 & 75 & 150 & 4.4E-7 & 78 & 31 & 109 & 5.6E-7 & 76 & 61 & 137 & 4.0E-7 \\\cline{4-25}
        ~  & ~ & ~ & 68 & 272 & 6.19E-7 & 71 & 27 & 196 & 4.7E-7 & 69 & 51 & 240 & 6.7E-7 & 72 & 288 & 5.7E-7 & 73 & 32 & 210 & 4.9E-7 & 72 & 59 & 262 & 5.6E-7 \\ \cline{4-25}
        ~ & ~ & ~ & 64 & 512 & 7.3E-7 & 70 & 22 & 368 & 6.7E-7 & 65 & 53 & 472 & 6.4E-7 & 71 & 568 & 4.3E-7 & 71 & 34 & 420 & 5.5E-7 & 71 & 59 & 520 & 4.6E-7 \\ \cline{4-25}
        ~ & ~ & ~ & 58 & 928 & 1.8E-6 & 67 & 15 & 656 & 6.8E-7 & 59 & 47 & 848 & 1.7E-6 & 64 & 1024 & 2.3E-6 & 66 & 29 & 760 & 1.4E-6 & 65 & 54 & 952 & 1.2E-6 \\ \cline{4-25}
        ~ & ~ & ~ & 48 & 1536 & 5.8E-6 & 51 & 30 & 1296 & 4.4E-6 & 50 & 40 & 1440 & 4.3E-6 & 58 & 1856 & 3.3E-6 & 60 & 27 & 1392 & 3.0E-6 & 58 & 49 & 1712 & 3.7E-6 \\ \cline{4-25}
        ~  & ~ & ~ & 39 & 2496 & 6.3E-6 & 41 & 24 & 2080 & 1.1E-5 & 40 & 34 & 2368 & 5.7E-6 & 48 & 3072 & 3.5E-6 & 53 & 23 & 2432 & 2.5E-6 & 49 & 39 & 2816 & 2.9E-6 \\ \hline\hline
        \multirow{6}{*}{\rotatebox[origin=c]{90}{\sky} }   & \multirow{6}{*}{\rotatebox[origin=c]{90}{287}} & \multirow{6}{*}{\rotatebox[origin=c]{90}{1.7E-4}} & 201 & 402 & 6.6E-5 & 236 & 22 & 258 & 3.1E-4 & 230 & 34 & 264 & 2.9E-4 & 210 & 420 & 6.6E-4 & 246 & 30 & 276 & 2.1E-4 & 217 & 151 & 368 & 4.6E-4 \\ \cline{4-25}
        ~  & ~ & ~ & 118 & 472 & 1.7E-6 & 174 & 27 & 402 & 1.6E-5 & 135 & 83 & 436 & 2.4E-6 & 176 & 704 & 1.5E-6 & 204 & 25 & 458 & 2.7E-4 & 221 & 71 & 584 & 1.9E-4 \\ \cline{4-25}
        ~ & ~ & ~ & 74 & 592 & 1.5E-6 & 100 & 27 & 508 & 8.0E-7 & 74 & 72 & 584 & 1.6E-6 & 125 & 1000 & 1.8E-6 & 175 & 24 & 796 & 2.3E-6 & 125 & 124 & 996 & 1.8E-6 \\ \cline{4-25}
        ~  & ~ & ~ & 49 & 784 & 1.0E-6 & 55 & 21 & 608 & 1.4E-6 & 49 & 47 & 768 & 1.0E-6 & 92 & 1472 & 3.7E-7 & 115 & 21 & 1088 & 1.1E-6 & 94 & 61 & 1240 & 1.6E-6 \\ \cline{4-25}
        ~  & ~ & ~ & 37 & 1184 & 3.7E-7 & 40 & 20 & 960 & 8.2E-7 & 37 & 36 & 1168 & 3.4E-7 & 62 & 1984 & 2.1E-7 & 83 & 18 & 1616 & 2.2E-7 & 62 & - & 1984 & 2.1E-7 \\ \cline{4-25}
        ~  & ~ & ~ & 23 & 1472 & 3.5E-7 & 23 & 17 & 1280 & 1.1E-6 & 23 & 22 & 1440 & 3.9E-7 & 42 & 2688 & 1.5E-7 & 56 & 15 & 2272 & 4.4E-7 & 42 & - & 2688 & 1.5E-7 \\ \hline\hline
    \end{tabular}
\end{table}

\begin{figure}[H] 
 \centering
 \hspace{-8mm}  \includegraphics[width=0.55\textwidth]{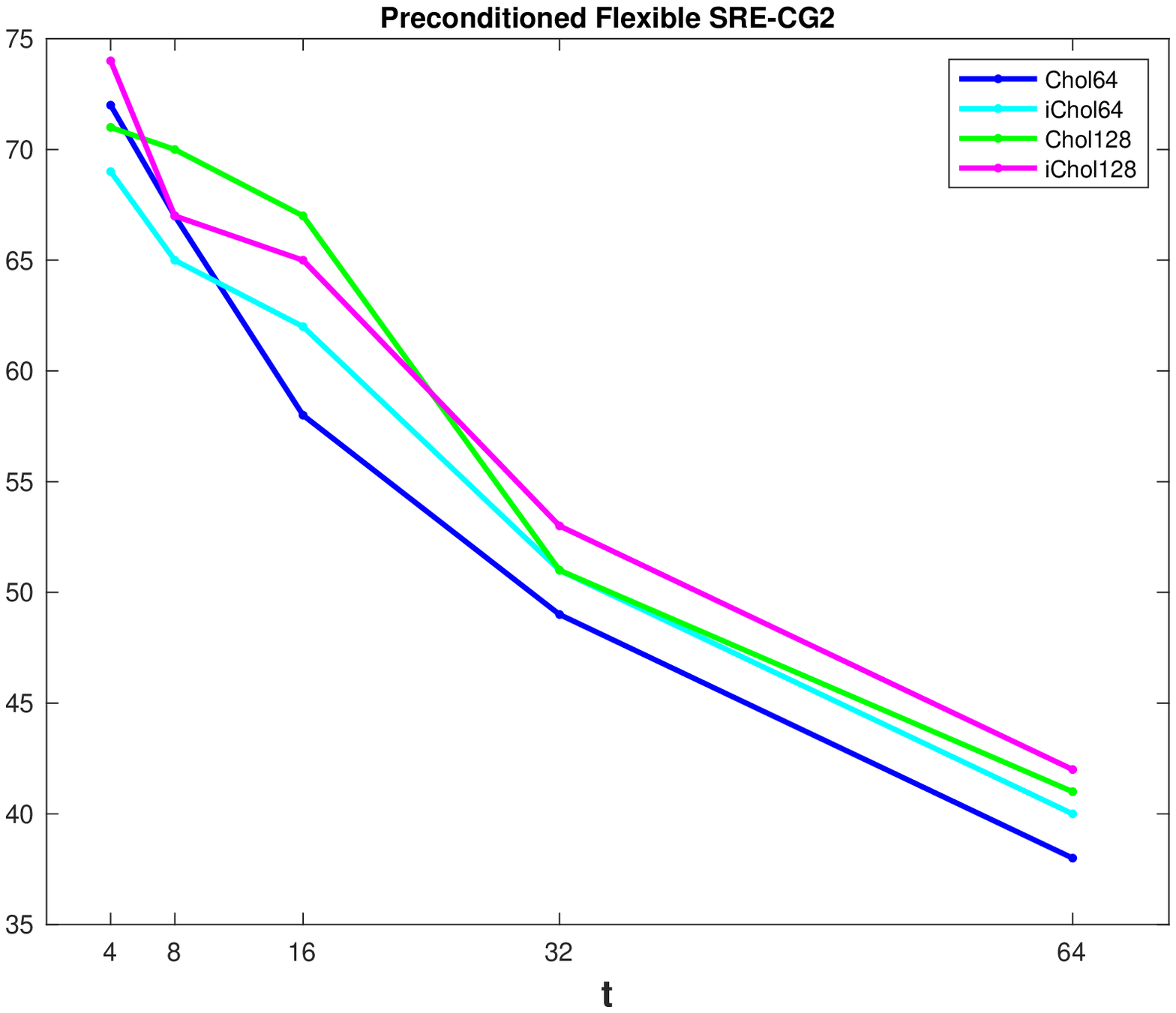} \hspace{-8mm} 
 \includegraphics[width=0.55\textwidth]{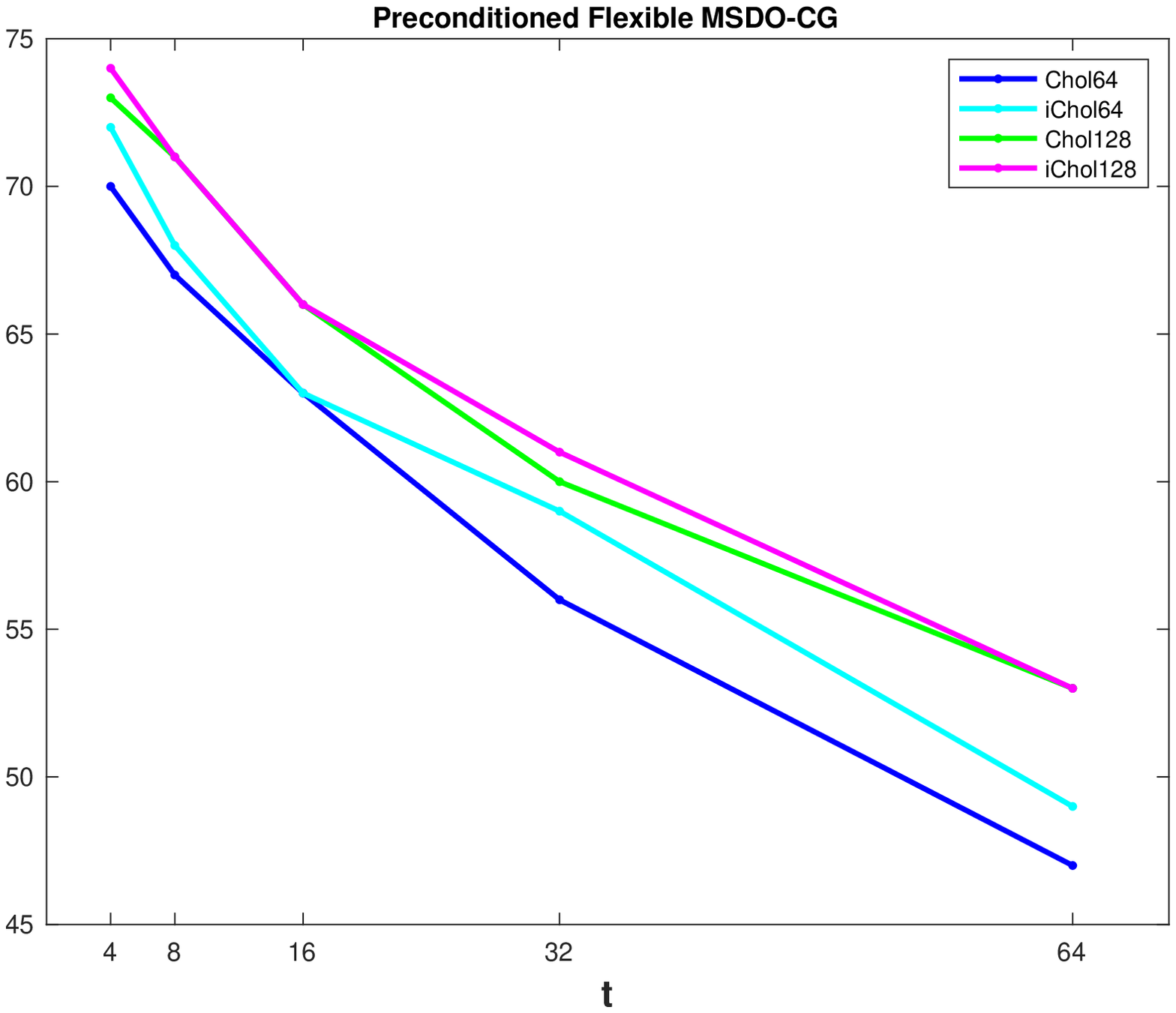}\\
\vspace{-2mm}
   \caption{Convergence of preconditioned flexibly enlarged methods with switchTol = $10^{-5}$ for matrix \anio} \label{fig:cholani}\vspace{-10mm}
\end{figure}

\begin{table}[H]
\setlength{\tabcolsep}{1pt}
\renewcommand{\arraystretch}{1.1}
\caption{\label{tab:precSRECG128ichol} Convergence and memory requirements of Block Jacobi Incomplete Cholesky 128 Preconditioned CG, SRE-CG2, flexibly SRE-CG2, MSDO-CG,and  flexibly MSDO-CG and with respect to number of partitions $\bf t=2,4,8,16,32,64$, and switch tolerances  $\bf switchTol=10^{-5},10^{-7}$. }\vspace{-3mm}
\hspace{-22mm} 
    \begin{tabular}{||c||c|c||c|c|c||c|c|c|c||c|c|c|c||c|c|c||c|c|c|c||c|c|c|c||}
    \cline{2-25}
     \multicolumn{1}{c||}{} &\multicolumn{2}{c||}{\multirow{2}{*}{\textbf{CG}}}&\multicolumn{3}{c||}{\multirow{2}{*}{\textbf{SRE-CG2} }} &\multicolumn{8}{c||}{\textbf{ Flexibly  SRE-CG2, switchTol =} }&\multicolumn{3}{c||}{\multirow{2}{*}{\textbf{MSDOCG} }} &\multicolumn{8}{c||}{\textbf{Flexibly  MSDOCG, switchTol =} } \\
    \multicolumn{1}{c||}{} &\multicolumn{1}{c}{} & &\multicolumn{1}{c}{}&\multicolumn{1}{c}{} & &\multicolumn{4}{c||}{$\bf 10^{-5}$}&\multicolumn{4}{c||}{$\bf 10^{-7}$}&\multicolumn{1}{c}{} &\multicolumn{1}{c}{} &\multicolumn{1}{c||}{}&\multicolumn{4}{c||}{$\bf 10^{-5}$}&\multicolumn{4}{c||}{$\bf 10^{-7}$} 
    \\
        \cline{2-25}
    \multicolumn{1}{c||}{} &$\bf It$&$\bf RE$&$\bf It$&$\bf mem$&$\bf RelEr$&$\bf It$&$\bf sIt$&$\bf mem$&$\bf RelEr$&$\bf It$&$\bf sIt$&$\bf mem$&$\bf RelEr$&$\bf It$&$\bf mem$&$\bf RelEr$&$\bf It$&$\bf sIt$&$\bf mem$&$\bf RelEr$&$\bf It$&$\bf sIt$&$\bf mem$&$\bf RelEr$
   \\
    \cline{1-25}
     \multirow{6}{*}{\rotatebox[origin=c]{90}{\nh} }   &  \multirow{6}{*}{\rotatebox[origin=c]{90}{119}} & \multirow{6}{*}{\rotatebox[origin=c]{90}{4.95E-7}} & 107 & 214 & 1.6E-7 & 115 & 59 & 174 & 2.4E-7 & 112 & 66 & 178 & 2.2E-7 & 114 & 228 & 1.8E-7 & 122 & 31 & 153 & 2.7E-7 & 114 & 93 & 207 & 1.9E-7 \\ \cline{4-25}
        ~ & ~ & ~ & 86 & 344 & 8.8E-8 & 90 & 55 & 290 & 1.2E-7 & 87 & 72 & 318 & 8.0E-8 & 92 & 368 & 1.2E-7 & 96 & 59 & 310 & 1.4E-7 & 92 & 80 & 344 & 1.8E-7 \\ \cline{4-25}
        ~ & ~ & ~ & 66 & 528 & 4.5E-8 & 68 & 46 & 456 & 5.0E-8 & 66 & 61 & 508 & 4.8E-8 & 74 & 592 & 8.0E-8 & 85 & 22 & 428 & 8.6E-8 & 75 & 68 & 572 & 5.6E-8 \\ \cline{4-25}
        ~ & ~ & ~ & 51 & 816 & 2.4E-8 & 51 & 36 & 696 & 3.5E-8 & 51 & 47 & 784 & 2.5E-8 & 62 & 992 & 4.7E-8 & 63 & 47 & 880 & 3.9E-8 & 62 & 59 & 968 & 4.8E-8 \\\cline{4-25}
        ~ & ~ & ~ & 39 & 1248 & 9.1E-9 & 39 & 30 & 1104 & 1.4E-8 & 39 & 37 & 1216 & 9.4E-9 & 49 & 1568 & 2.3E-8 & 53 & 23 & 1216 & 3.4E-8 & 49 & 47 & 1536 & 2.4E-8 \\ \cline{4-25}
        ~ & ~ & ~ & 29 & 1856 & 8.6E-9 & 29 & 24 & 1696 & 1.1E-8 & 29 & 28 & 1824 & 8.8E-9 & 37 & 2368 & 1.1E-8 & 37 & 32 & 2208 & 1.9E-8 & 37 & - & 2368 & 1.2E-8 \\ \hline\hline
         \multirow{6}{*}{\rotatebox[origin=c]{90}{\skyt} }   &  \multirow{6}{*}{\rotatebox[origin=c]{90}{266}} &  \multirow{6}{*}{\rotatebox[origin=c]{90}{1.00E-5}} & 236 & 472 & 5.3E-6 & 251 & 20 & 271 & 1.0E-5 & 236 & 141 & 377 & 7.0E-6 & 265 & 530 & 1.9E-5 & 255 & 32 & 287 & 1.2E-5 & 271 & 144 & 415 & 1.7E-5 \\\cline{4-25}
        ~ & ~ & ~ & 180 & 720 & 3.8E-6 & 223 & 14 & 474 & 7.8E-6 & 177 & 141 & 636 & 9.8E-6 & 236 & 944 & 8.0E-6 & 260 & 22 & 564 & 1.1E-5 & 234 & 167 & 802 & 1.2E-5 \\ \cline{4-25}
        ~ & ~ & ~ & 129 & 1032 & 5.6E-7 & 178 & 14 & 768 & 2.3E-6 & 136 & 104 & 960 & 4.1E-6 & 163 & 1304 & 2.3E-6 & 225 & 31 & 1024 & 7.0E-6 & 170 & 117 & 1148 & 1.1E-5 \\ \cline{4-25}
        ~ & ~ & ~ & 72 & 1152 & 2.8E-7 & 98 & 34 & 1056 & 1.5E-6 & 73 & 68 & 1128 & 4.4E-7 & 126 & 2016 & 3.4E-6 & 152 & 20 & 1376 & 1.2E-5 & 128 & 105 & 1864 & 3.4E-6 \\ \cline{4-25}
        ~ & ~ & ~ & 42 & 1344 & 5.1E-7 & 60 & 16 & 1216 & 4.6E-7 & 42 & - & 1344 & 5.1E-7 & 89 & 2848 & 8.0E-7 & 109 & 30 & 2224 & 1.4E-6 & 89 & 83 & 2752 & 8.2E-7 \\ \cline{4-25}
        ~ & ~ & ~ & 27 & 1728 & 4.5E-7 & 27 & 25 & 1664 & 5.1E-7 & 27 & - & 1728 & 4.5E-7 & 57 & 3648 & 2.1E-7 & 70 & 28 & 3136 & 3.2E-6 & 60 & 51 & 3552 & 1.3E-7 \\ \hline \hline
         \multirow{6}{*}{\rotatebox[origin=c]{90}{\ani} }   &  \multirow{6}{*}{\rotatebox[origin=c]{90}{77}} &  \multirow{6}{*}{\rotatebox[origin=c]{90}{1.08E-6}} & 75 & 150 & 1.1E-6 & 77 & 32 & 109 & 1.2E-6 & 75 & 61 & 136 & 1.4E-6 & 76 & 152 & 3.8E-7 & 81 & 13 & 94 & 1.3E-6 & 76 & 62 & 138 & 3.7E-7 \\ \cline{4-25}
        ~ & ~ & ~ & 68 & 272 & 7.0E-7 & 74 & 22 & 192 & 4.9E-7 & 69 & 54 & 246 & 7.1E-7 & 72 & 288 & 7.4E-7 & 74 & 32 & 212 & 5.4E-7 & 73 & 59 & 264 & 5.9E-7 \\ \cline{4-25}
        ~ & ~ & ~ & 65 & 520 & 7.4E-7 & 67 & 27 & 376 & 8.5E-7 & 65 & 53 & 472 & 7.7E-7 & 71 & 568 & 5.2E-7 & 71 & 32 & 412 & 5.3E-7 & 71 & 58 & 516 & 5.2E-7 \\ \cline{4-25}
        ~ & ~ & ~ & 59 & 944 & 1.5E-6 & 65 & 22 & 696 & 1.8E-6 & 60 & 51 & 888 & 1.8E-6 & 64 & 1024 & 2.1E-6 & 66 & 29 & 760 & 1.4E-6 & 64 & 54 & 944 & 1.8E-6 \\ \cline{4-25}
        ~ & ~ & ~ & 49 & 1568 & 4.9E-6 & 53 & 28 & 1296 & 3.0E-6 & 50 & 40 & 1440 & 4.8E-6 & 59 & 1888 & 3.1E-6 & 61 & 28 & 1424 & 3.0E-6 & 60 & 50 & 1760 & 2.6E-6 \\ \cline{4-25}
        ~ & ~ & ~ & 39 & 2496 & 7.8E-6 & 42 & 24 & 2112 & 8.3E-6 & 40 & 31 & 2272 & 8.8E-6 & 50 & 3200 & 2.0E-6 & 53 & 24 & 2464 & 3.4E-6 & 51 & 43 & 3008 & 1.5E-6 \\ \hline \hline
             \multirow{6}{*}{\rotatebox[origin=c]{90}{\sky} }   &  \multirow{6}{*}{\rotatebox[origin=c]{90}{338}} &  \multirow{6}{*}{\rotatebox[origin=c]{90}{1.13E-4}} & 240 & 480 & 7.5E-5 & 279 & 33 & 312 & 4.3E-4 & 264 & 134 & 398 & 2.0E-4 & 281 & 562 & 5.9E-4 & 276 & 33 & 309 & 6.6E-4 & 282 & 65 & 347 & 3.5E-4 \\ \cline{4-25}
        ~ & ~ & ~ & 148 & 592 & 1.3E-6 & 226 & 30 & 512 & 2.0E-6 & 158 & 102 & 520 & 1.3E-6 & 195 & 780 & 2.7E-4 & 262 & 25 & 574 & 5.7E-4 & 301 & 36 & 674 & 9.9E-5 \\ \cline{4-25}
        ~ & ~ & ~ & 87 & 696 & 1.4E-6 & 119 & 24 & 572 & 7.9E-7 & 87 & 83 & 680 & 1.4E-6 & 143 & 1144 & 1.4E-6 & 179 & 24 & 812 & 3.7E-4 & 153 & 99 & 1008 & 2.6E-6 \\ \cline{4-25}
        ~ & ~ & ~ & 57 & 912 & 1.0E-6 & 67 & 23 & 720 & 2.5E-6 & 57 & 54 & 888 & 1.3E-6 & 101 & 1616 & 6.3E-7 & 132 & 23 & 1240 & 1.8E-6 & 101 & 100 & 1608 & 6.4E-7 \\ \cline{4-25}
        ~ & ~ & ~ & 42 & 1344 & 3.4E-7 & 45 & 24 & 1104 & 9.0E-7 & 42 & 41 & 1328 & 3.7E-7 & 72 & 2304 & 7.1E-7 & 93 & 21 & 1824 & 1.1E-6 & 72 & - & 2304 & 7.2E-7 \\ \cline{4-25}
        ~ & ~ & ~ & 26 & 1664 & 4.8E-7 & 27 & 20 & 1504 & 4.9E-7 & 26 & 25 & 1632 & 5.4E-7 & 50 & 3200 & 1.5E-7 & 63 & 17 & 2560 & 7.6E-7 & 50 & - & 3200 & 1.5E-7 \\ \hline \hline
         \end{tabular}
\end{table}

\begin{figure}[H] 
 \centering
 \hspace{-8mm}  \includegraphics[width=0.55\textwidth]{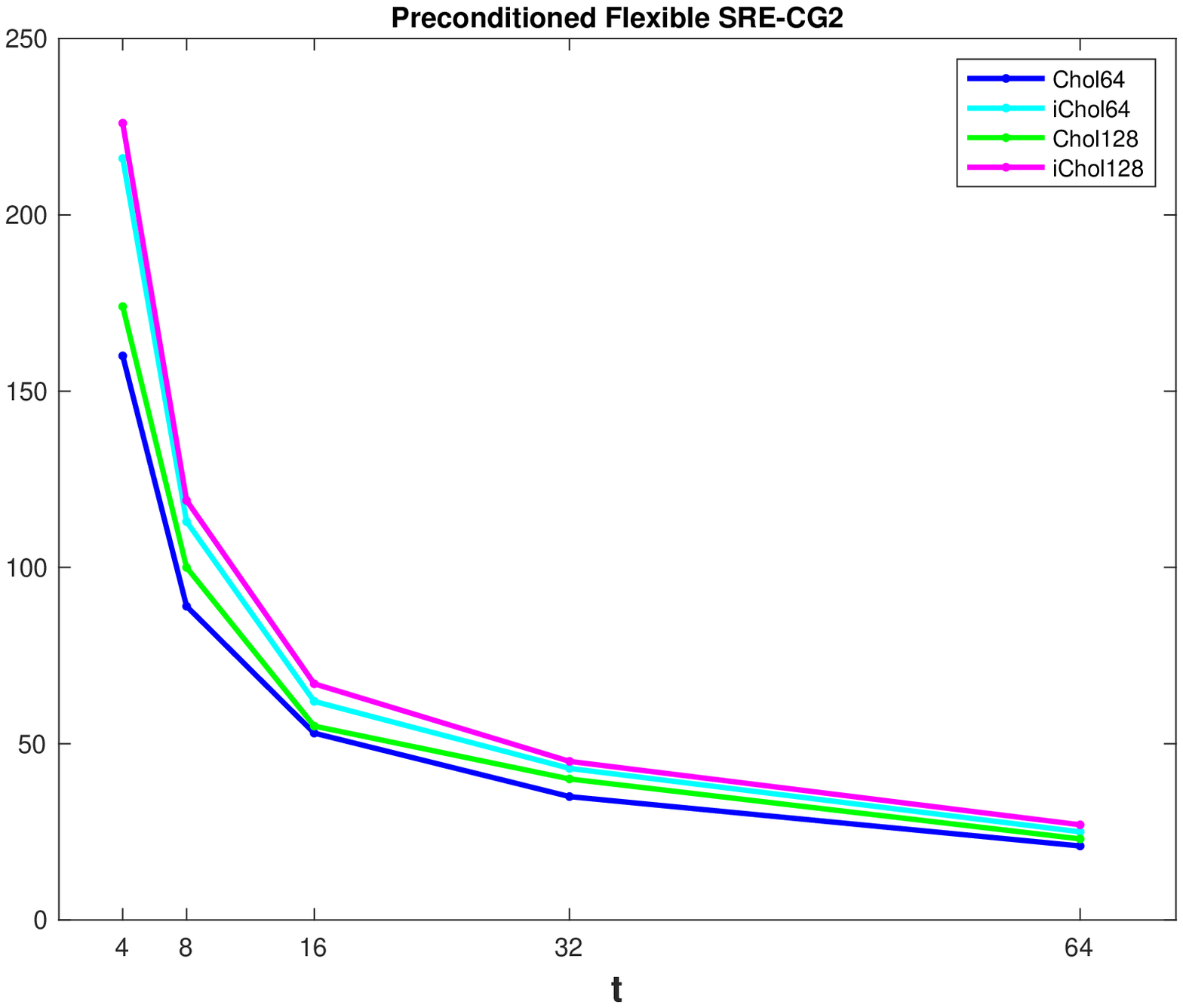} \hspace{-8mm} 
 \includegraphics[width=0.55\textwidth]{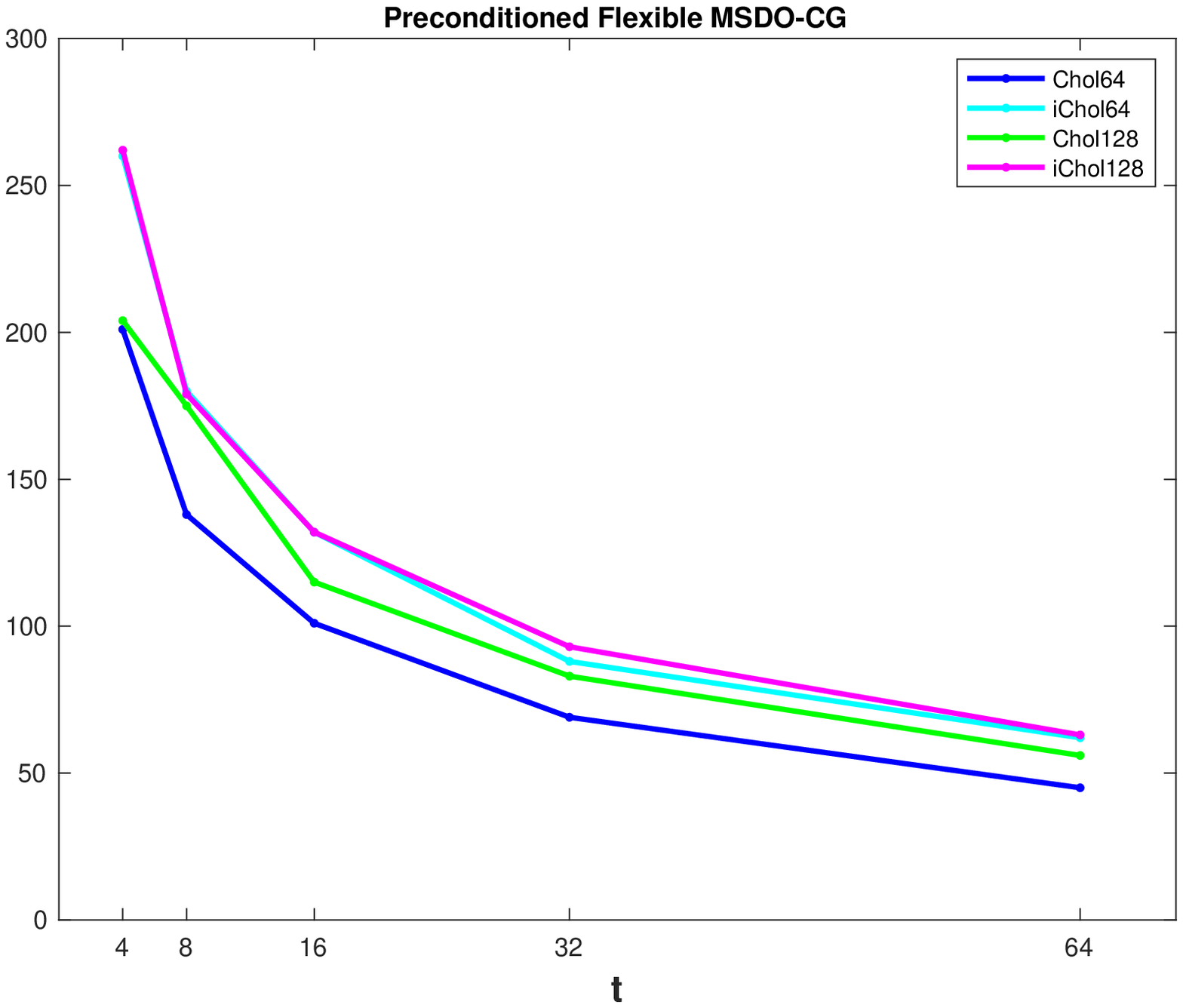}\\
\vspace{-2mm}
   \caption{Convergence of preconditioned flexibly enlarged methods with switchTol = $10^{-5}$ for matrix \skyo} \label{fig:cholsky2}
\end{figure}
Clearly, the Cholesky Block Jacobi preconditioners reduces the iterations more than the incomplete Cholesky version  (Chol64 vs iChol64, Chol128 vs iChol128). Similarly, using 64 blocks reduces the iterations more than 128 blocks (Chol64 vs Chol128, iChol64 vs iChol128).

 Comparing the 4 versions (Chol64, iChol64, Chol128, iChol128) for the 4 matrices, Chol64 requires the least number of iterations whereas iChol128 requires the most number of iterations with an increase of around 30$\%$. As for iChol64 and Chol128, for the matrices \nho \,and \skyo, Chol128 requires less iterations then iChol64, contrarily to the matrices \anio \, and \skyto.

\section{Conclusion}
In this paper we discussed different options for reducing the memory requirements of the introduced  enlarged CG methods, specifically SRE-CG2  and MSDO-CG, without affecting much their convergence.
We considered two options that require a fixed memory. The first is truncating the A-orthonormalization process, and the second is restarting after a fixed number of iteration.
For SRE-CG2, it was proven theoretically \cite{EKS} that it is possible to truncate the A-orthonormalization process for some preset $trunc$ value. However, if this  $trunc$ value is too small relative to the required iterations to convergence by SRE-CG2, then the truncated version will require much more iterations,  but still less than CG. For MSDO-CG, truncation doesn't necessarily lead to convergence, as it is not guaranteed theoretically. As for restarting after a fixed number of iterations, it leads to stagnation and the method may not converge in $k_{max}$ iterations.

Thus, we introduced
flexibly enlarged versions where after some ``switch" iteration, once the relative difference of residual norms is less than a given swichTol, the number of computed vectors is reduced to half. We proved that the flexibly enlarged Krylov subspace is a superset to the classical Krylov subspace, thus the introduced methods will converge in less iterations than CG (which is validated in the numerical testings). Moreover, the convergence of the flexibly enlarged methods is within the range of convergence of the corresponding methods with $t$ and $t/2$ vectors per iteration.
And it is observed that setting swichTol $=10^{-5}$ is the moderate choice that balances between the reduced memory and the augmented iterations to convergence.

Accordingly, we tested some restarted versions based on the same switching condition. The first restarts  every time the relative difference of the residuals is less than the restart tolerance. Unlike the restarted version with fixed restart iterations, it convergence within the $k_{max}$ iterations and in less iterations than CG, but more iterations than the corresponding method (up to double the iterations) where the memory is reduced.  The second restarts only once. This reduces the augmented iterations of the first restarted version, but increases the memory requirements. The third is restarted flexibly enlarged versions that  restart only once, but $t$ is halved after the restart. This further reduces the memory requirements, but at the the expense of requiring more iterations than the second restarted version and the flexibly enlarged version. 

Hence, the introduced flexibly enlarged methods with swichTol $=10^{-5}$ are the moderate choice that balances between the reduced memory and the augmented iterations to convergence. Moreover, by preconditioning these flexibly enlarged methods, the iterations to convergence are significantly reduced,  further reducing the memory requirements.

\end{document}